\documentclass[11pt,reqno]{amsart}
\usepackage{etex}
\usepackage{textcmds} 
\usepackage{amsmath, amssymb, amsfonts, amstext, verbatim, amsthm, mathrsfs}
\usepackage{microtype}
\usepackage[all]{xy}
\usepackage[modulo]{lineno}
\usepackage[usenames]{color}
\usepackage{aliascnt}
\usepackage{enumitem}
\usepackage{xspace}
\usepackage{amsfonts}
\usepackage[centertags]{amsmath}
\usepackage{rotating}
\usepackage[margin=3.5cm]{geometry}
\usepackage{dsfont}
\usepackage{bm}
\usepackage{color}
\usepackage{subfigure}
\usepackage{amsmath}
\usepackage{array}
\usepackage[all]{xy}
\usepackage{euscript}
\usepackage[T1]{fontenc}
\usepackage{mathbbol}
\usepackage{nccmath}
\usepackage{marginnote}
\usepackage{stmaryrd}
\usepackage{youngtab}
\usepackage{tikz,pgfplots}
\usepackage{blkarray}
\usepackage{ulem}
\usepackage[abs]{overpic}

\usepackage{graphics,graphicx}  

\let\oldtocsection=\tocsection
\let\oldtocsubsection=\tocsubsection

\renewcommand{\tocsection}[2]{\hspace{0em}\oldtocsection{#1}{#2}}
\renewcommand{\tocsubsection}[2]{\hspace{1em}\oldtocsubsection{#1}{#2}}

\usepackage[backref,colorlinks=true,linkcolor=black,citecolor=blue,urlcolor=blue,citebordercolor={0 0 1},urlbordercolor={0 0 1},linkbordercolor={0 0 1}]{hyperref} 

\tikzset{node distance=3cm, auto}

\makeatletter
\def\@secnumfont{\bfseries}
\def\section{\@startsection{section}{1}%
  \z@{.7\linespacing\@plus\linespacing}{.5\linespacing}%
  {\normalfont\Large\bfseries}}
\def\subsection{\@startsection{subsection}{2}%
  \z@{.5\linespacing\@plus.7\linespacing}{-.5em}%
  {\normalfont\large\bfseries}}
  \def\subsubsection{\@startsection{subsubsection}{3}%
  \z@{.5\linespacing\@plus.7\linespacing}{-.5em}%
  {\normalfont\bfseries}}
\makeatother

\newtheorem{thm}{Theorem}[subsection]
\newtheorem{lemma}[thm]{Lemma}
\newtheorem{prop}[thm]{Proposition}
\newtheorem{cor}[thm]{Corollary}

\newtheorem{remark}[thm]{Remark}
\newtheorem{rmk}[thm]{Remark}

\newtheorem{conjecture}[thm]{Conjecture}
\newtheorem{definition}[thm]{Definition}

\newtheorem{example}[thm]{Example}
\theoremstyle{remark}
\numberwithin{equation}{subsection} 

\numberwithin{figure}{section}
\numberwithin{table}{subsection}

\newcommand{\eend}{{\rm end}}

\newcommand{\whZ}{{\widehat{Z}}}

\newcommand{\pre}{{\rm pre}}
\newcommand{\ppre}{{\rm ppre}}
\newcommand{\oCP}{{\overline{{\C}P}}\!\,}
\newcommand{\ovm}{{\underline{m}}}
\newcommand{\ovd}{{\underline{d}}}
\newcommand{\ovq}{{\underline{q}}}
\newcommand{\ovp}{{\underline{p}}}
\newcommand{\id}{{\rm id}}

\newcommand{\N}{{\mathbb{N}}}
\newcommand{\Ee}{\mathcal{E}}

\newcommand{\Tt}{\mathcal{T}}

\newcommand{\vol}{{\rm vol\,}}
\newenvironment{itemlist}
   { \begin{list} {$\bullet$}
         { \setlength{\topsep}{.5ex}  \setlength{\itemsep}{.5ex} \setlength{\leftmargin}{2.5ex} } }
   { \end{list} }

\newcommand{\bbm}{{\bf{m}}}
\newcommand{\bw}{{\bf{w}}}
\newcommand{\bx}{{\bf{x}}}
\newcommand{\bn}{{\bf{n}}}

\newcommand{\bE}{{\bf{E}}}
\newcommand{\bB}{{\bf{B}}}

\newcommand{\intt}{{\rm int\,}}
\newcommand{\NI}{{\noindent}}

\newcommand{\Cc}{{\mathcal C}}
\newcommand{\Ss}{{\mathcal S}}

\newcommand{\Jj}{{\mathcal J}}

\newcommand{\ov}{\overline}
\newcommand{\al}{{\alpha}}

\newcommand{\be}{{\beta}}

\newcommand{\om}{{\omega}}
\newcommand{\vareps}{{\epsilon}}
\newcommand{\de}{{\delta}}

\newcommand{\io}{{\iota}}
\newcommand{\ka}{{\kappa}}
\newcommand{\la}{{\lambda}}

\newcommand{\si}{{\sigma}}

\newcommand{\less} {{\smallsetminus}}
\newcommand{\p}{{\partial}}
\newcommand{\MS}{{\medskip}}

\newcommand{\er}{{\Diamond}}
\newcommand{\Z}{\mathbb{Z}}

\newcommand{\R}{\mathbb{R}}
\newcommand{\Q}{\mathbb{Q}}
\newcommand{\C}{\mathbb{C}}
\newcommand{\CP}{\mathbb{CP}}

\newcommand{\eps}{\varepsilon}

\newcommand{\acc}{\mathrm{acc}}
\newcommand{\sembeds}{\stackrel{s}{\hookrightarrow}}
\newcommand{\se}{\stackrel{s}{\hookrightarrow}}

\makeatletter
\newcommand{\dashover}[2][\mathop]{#1{\mathpalette\df@over{{\dashfill}{#2}}}}
\newcommand{\fillover}[2][\mathop]{#1{\mathpalette\df@over{{\solidfill}{#2}}}}
\newcommand{\df@over}[2]{\df@@over#1#2}
\newcommand\df@@over[3]{%
  \vbox{
    \offinterlineskip
    \ialign{##\cr
      #2{#1}\cr
      \noalign{\kern1pt}
      $\m@th#1#3$\cr
    }
  }%
}
\newcommand{\dashfill}[1]{%
  \kern-.5pt
  \xleaders\hbox{\kern.5pt\vrule height.4pt width \dash@width{#1}\kern.5pt}\hfill
  \kern-.5pt
}
\newcommand{\dash@width}[1]{%
  \ifx#1\displaystyle
    2pt
  \else
    \ifx#1\textstyle
      1.5pt
    \else
      \ifx#1\scriptstyle
        1.25pt
      \else
        \ifx#1\scriptscriptstyle
          1pt
        \fi
      \fi
    \fi
  \fi
}
\newcommand{\solidfill}[1]{\leaders\hrule\hfill}
\makeatother

\title{Staircase patterns  in Hirzebruch surfaces}
\author{Nicki Magill}
\address{Mathematics Department, Cornell University}
\email{nm627@cornell.edu}
\thanks{NSF Graduate Research Grant DGE-1650441}
\author{Dusa McDuff}
\address{Mathematics Department, Barnard College}
\email{dusa@math.columbia.edu}
\author{Morgan Weiler}
\address{Mathematics Department, Cornell University}
\email{morgan.weiler@cornell.edu}
\thanks{NSF Research Grant 2103245}
\date{May 30, 2023}

\begin{document}

\begin{abstract}  The ellipsoidal capacity function of a symplectic four manifold $X$ measures how much the form on $X$ must be dilated in order for it to admit an embedded ellipsoid of eccentricity  $z$.  In most cases there are just finitely many obstructions to such an embedding besides the volume.  If there are infinitely many obstructions, $X$ is said to have a staircase.  This paper gives an almost complete description of the staircases 
in the ellipsoidal capacity functions of
the family of  symplectic Hirzebruch surfaces $H_b$ formed by 
blowing up  the projective plane with weight $b$. 
 We describe an interweaving, recursively defined, family of obstructions to  symplectic embeddings of ellipsoids that show there is an open dense set of 
shape parameters $b$ that are blocked, i.e. have no staircase, and an uncountable number of other values of  $b$ that do admit staircases.   The remaining $b$-values form a countable sequence 
of special rational numbers that are closely related to the symmetries discussed in Magill--McDuff
(arXiv:2106.09143). We show that none of them admit ascending staircases. Conjecturally, 
none admit  descending staircases.  Finally, we  show that, as long as $b$ is not one of these special rational values, any staircase in $H_b$ has irrational accumulation point.
  A crucial ingredient of our proofs is the new, more indirect approach to using  almost toric fibrations in the analysis of staircases  by Magill (arXiv:2204.12460).  In particular, the structure of the relevant mutations of 
the set of almost toric fibrations on $H_b$ is echoed in the structure of the set of blocked $b$-intervals.
\end{abstract}

\maketitle

\tableofcontents

\section{Introduction}

\subsection{Overview and statement of main theorem}\label{ss:intro}
The ellipsoidal capacity function $c_{X}: [1,\infty) \to \R$ for a general four-dimensional target manifold ($X,\om)$  is defined by
$$
c_{X}(z):=\inf\left\{\lambda\ \Big|\ E(1,z)\sembeds \lambda X\right\},$$
where $z\ge 1$ 
is a real variable,  $\lambda X: = (X,\lambda\omega)$,  the ellipsoid $E(c,d)\subset \mathbb{C}^2$ is the set
$$
E(c,d)=\left\{(\zeta_1,\zeta_2)\in\mathbb{C}^2 \ \big|\  \pi \left(  \frac{|\zeta_1|^2}{c}+\frac{|\zeta_2|^2}{d} \right)<1 \right\},
$$
and we write $E\se \la X$ if there is a symplectic embedding of $E$ into $\la X$.  

It is straightforward to see that
$c_{X}(z)$ is bounded below by the volume constraint function $$
V_X(z) =\sqrt{\frac{\vol E(1,z)}{\vol(X,\om)}}.
$$  
Using techniques developed in McDuff~\cite{Mell}, McDuff--Schlenk~\cite{ball} gave the first complete computation of this function for the case when $X$ is the standard 4-ball,  or, equivalently, $\CP^2$.  They found that the graph of this function has infinitely many nonsmooth points at values of $z$ that are ratios of Fibonacci numbers.  
The results of \cite{Mell} were generalized by
Cristofaro-Gardiner~\cite{CG}, whose work implies that if $(X,\om)$ is a 
four-dimensional toric manifold or rational convex toric domain the function $z\mapsto c_{X}(z)$ is piecewise linear when not identically equal to the volume constraint curve. When, as in the case of the ball,  its graph has infinitely many nonsmooth points lying above the volume curve, $(X,\om)$ is said to have a {\bf staircase}.\footnote{This is often referred to as an infinite staircase in the literature, but we presuppose that a staircase has infinitely many steps. On the other hand a staircase need not contain all the nonsmooth points in a neighborhood of the accumulation point.}

Four manifolds with staircases seem rather rare: Frenkel--Muller~\cite{FM}  used the methods of \cite{ball} to find a staircase for the monotone product $(S^2\oplus S^2, \om \times \om)$, while  
in \cite{AADT}, Cristofaro-Gardiner--Holm--Mandini--Pires used methods from ECH (embedded contact homology) to find staircases for the monotone blowup of $\CP^2$ by up to four points. Their conjecture  \cite[Conj.~1.23]{AADT} proposes (among other things) that that these are the only
rational toric  four manifolds with staircases.  In contrast, Usher~\cite{Usher} found infinitely many irrational $\beta$ such that the nonmonotone product $(S^2\times S^2,  \om \oplus \beta \om)$ admits a staircase.

The analogous question for the family of Hirzebruch surfaces $$
H_b: = \CP^2(1)\#\ov{\CP}^2(b), 0\le b < 1
$$ was first investigated by 
Bertozzi, Holm, Maw et al. in \cite{ICERM}, with work continued in Magill--McDuff~\cite{MM}.  The current paper completes this circle of ideas and provides an almost complete 
answer to the question of which $H_b $ admit staircases.  
Our main result, Theorem \ref{thm:main} below, is the key to the proof of a special case of \cite[Conj.~1.23]{AADT} that is given in the forthcoming Magill--Pires--Weiler~\cite{MPW}. Finally, the structure explained here will, we hope, provide a guide for classifying the ellipsoidal capacity functions for more general toric blowups, see Remark~\ref{rmk:fdgen}.

Throughout, we denote by $H_b$ the one-point blowup of $\CP^2$ with line of area $1$ by a ball of capacity $b$.\footnote
{Since $H_b$ is a rational four manifold, its symplectomorphism class is unique; see \cite[Ex.7.1.16]{MS}.}  Its volume constraint is 
 $$V_b(z) = \sqrt{\frac{z}{1-b^2}},
 $$
  where $1-b^2$ is the appropriately normalized volume of $H_b$.  
A key result from \cite[Thm.1.13]{AADT} is that 
if $c_{H_b}$ has a staircase, then the {\bf steps} (i.e. nonsmooth local maxima) of this staircase accumulate to the 
unique solution $z=\acc(b) >1$ of the following quadratic equation involving $b$:
\begin{align}\label{eq:accb0}
z^2-\left( \frac{\left(3-b\right)^2}{1-b^2}-2 \right)z+1=0.
\end{align}
Note that the coefficient of $z$ in this equation is determined by the shape of the moment polytope:
$1-b^2$ is its normalized volume, while $3-b$ is the affine length of its perimeter.   The function $b\mapsto \acc(b)$ is $2$-to-$1$ in general and takes its minimum value
$a_{\min} = 3+2\sqrt2$ at $b = 1/3$, the only positive rational value of $b$ that is known to admit a staircase; see Fig.~\ref{fig:1}. 
We say a staircase is \textbf{ascending} (resp. \textbf{descending})  if its steps have increasing (resp. \textbf{decreasing}) $z$-coordinates.

Another key point proved in \cite{AADT} is that when $H_b$ has a staircase, there is no obstruction at its accumulation point, i.e.
\begin{align}\label{eq:accb1}
c_{H_b}(\acc(b)) = V_b(\acc(b)).
\end{align}
More generally, a pair $(z,b)$ with  $z = \acc(b)$  such that $c_{H_b}(\acc(b)) = V_b(\acc(b))$ is called {\bf unobstructed}; otherwise $(z,b)$ (or simply $z$ or $b$) is said to be {\bf blocked}.
\MS

 Here is our main result.
 
 \begin{thm}\label{thm:main}
\begin{itemlist}\item[{\rm (i)}] The set  
$$
{\it Block} : =  \bigl\{ b\in[0,1) \ \big| \ c_{H_b}(\acc(b)) > V_b(\acc(b))\bigr\},
$$
is an open dense subset of $[0,1)$ that is invariant under the action of the symmetries defined below.

\item[{\rm (ii)}]  All other $b$-values, except possibly for those where $\acc(b)$ is one of the special rational points 
$(6,35/6, 204/35,\dots$ in \eqref{eq:specrat} below, admit staircases.  If $b \neq1/3$ is not an endpoint of a connected component of ${\it Block}$,
then $b$ admits both an ascending and a descending staircase; while if $b$ is an endpoint of a connected component of ${\it Block}$, then $b$ admits either an ascending or descending staircase with steps lying outside the corresponding blocked $z$-interval.

\item[{\rm (iii)}]   For $n\ge 0$ define 
\begin{align*}
{\it Block}_{[2n+6,2n+8]}: = \bigl\{ b \in {\it Block} \ \big| \ \acc(b) \in [2n+6,2n+8], \ b>1/3\bigr\}.
\end{align*}   
For each $n\ge 0$ there is a homeomorphism of ${\it Block}_{[2n+6,2n+8]}$ onto the complement $[-1,2]\less C$ of  the middle third Cantor set $C\subset [0,1]$.\MS

\item[{\rm (iv)}] The only rational numbers $z = p/q$ that might be staircase accumulation points  are the special rational points.  Any 
such staircase would have to be descending.
\end{itemlist}
\end{thm}

The {\bf symmetries}  studied in 
\cite{ICERM, MM} are an
 integral part of the structure of the set {\it Block} and its disjoint counterpart
$$
{\it Stair} : =  \bigl\{ b\in[0,1) \ \big| \ H_b \mbox{ has a staircase}\bigr\}.
$$
These symmetries stem from the arithmetic properties of the quadratic function \eqref{eq:accb0}, and 
 their existence reduces the problem of calculating {\it Block} and {\it Stair} to calculating the restriction of these sets  to $b\in [5/11,1)$, in other words to $b$ with $\acc(b)=z>7$. The sequence of points
\begin{align}\label{eq:specrat}
z= 6,\ 35/6,\ 204/35,\dots
\end{align}
given by the images of $z=6$ under repeated applications of the shift symmetry $S: p/q\mapsto (6p-q)/p$ play a special role. These are called {\bf the special rational points}, as are
the $b$-values that correspond to these points via the function $b\mapsto \acc(b)$. 
These $b$-values
are also rational by \cite[Lem.2.1.1]{MM}, and
are described more fully in 
equation~\eqref{eq:specialb}. A definition of the symmetries can be found in \S\ref{ss:symm}.
They are generated by the shift $S$ and a reflection $R$.

The proof of Theorem~\ref{thm:main}~(i) and~(iii)  is completed in \S\ref{ss:disj}, while those of (ii) and (iv) are completed in  \S\ref{ss:uncount} and \S\ref{ss:13stair} respectively. Since this work, the further developments in~\cite{MPW}  finish the cases of Theorem~\ref{thm:main} (ii) and (iv), thus giving a complete computation of {\it Stair}. We also expect Theorem~\ref{thm:main} to generalize to other toric domains. Evidence of this has been seen in the case of the polydisk in Farley--Holm--Magill et al.~\cite{spur} and \cite{Usher}, and a two-fold blowup of $\CP^2$ in Magill~\cite{M2}. See Remark~\ref{rmk:fdgen} for more details. 
\MS

The main work of the paper lies in constructing the set {\it Block} and in
computing infinitely many values of $c_b(z)$ for each $b$ claimed to be in {\it Stair}. As explained in detail in \S\ref{ss:mutation}, specific homology classes in various blowups of $\CP^2$, called {\bf exceptional classes} and denoted $\bE$, give lower bounds $\mu_{\bE,b}(z) \leq c_b(z)$ for the embedding function. These obstructions $\mu_{\bE,b}(z)$ vary continuously in $b$ and $z.$ We find particular exceptional classes $\bE$ called {\bf perfect blocking classes} such that there is a maximal interval $J_\bE \subset [0,1)$ where for all $b \in J_\bE,$
\[ V_b(\acc(b))<\mu_{\bE,b}(z).\] 
Thus, because $\mu_{\bE,b}(z)$ is a lower bound of $c_b(z)$, it follows from \eqref{eq:accb1} that $J_\bE \subset {\it Block}$. 

We calculate {\it Block} by finding all such perfect blocking classes; see Proposition~\ref{prop:1} for a  precise statement. 
These are built from the sequence of perfect blocking classes found in
 \cite[Thm.56]{ICERM}.
In Section~\ref{ss:genTt}, we organize these into triples of classes
called {\bf generating triples} that satisfy various compatibility conditions,  and then define a way to {\bf mutate} these triples to produce new triples in a recursive structure; see Example~\ref{ex:mainEx} for an illustration. We then use the existence of this whole family of 
perfect blocking classes to show that each 
 $J_\bB$ forms a connected component of {\it Block}.  
 The argument here
 relies on 
 the use by  Magill~\cite{M1} of almost toric fibrations to construct 
 full fillings\footnote
 {
 that is, symplectic embeddings  $\intt E(1,\acc(b)) \se H_b$}
  of $H_b$ by ellipsoids $E(1,\acc(b))$  when $b$ is the lower endpoint of $J_\bE$.  Thus these values of
$b$ are unobstructed, and 
by Lemma~\ref{lem:live2} this implies that all the classes obtained by mutation are exceptional classes.  Hence, by the results of \cite{MM} quoted in Proposition~\ref{prop:live},  infinitely many of these classes are {\bf live} (that is, the corresponding obstructions are visible in the 
capacity function $c_{H_b}$) at the relevant limiting $b$ value $b_\infty$. 

Theorem~\ref{thm:main1} states how the triples are generated and how each triple corresponds to two staircases.
 The structure of these triples allow us to conclude that ${\it Block}_{[2n+6,2n+8]}$ is homeomorphic to the complement of the Cantor set.  One crucial compatibility condition is called {\bf adjacency}, which  expresses the relation of a staircase to the perfect blocking class that blocks an interval  ending at its accumulation point; see Remark~\ref{rmk:adjac}.

\begin{rmk}\label{rmk:general}
\rm  (i) Theorem~\ref{thm:main}~(iii) implies that for each $n$ there is an order-preserving bijection from the centers of the steps in $[6,8]$ to those in $ [2n+6, 2n+8]$ that takes staircases with accumulation points   in $[6,8]$ (and $b>1/3$)  to those with accumulation points in $[2n+6, 2n+8]$.  However, this bijection
does not seem to have a natural extension to a homeomorphism.
 It is better thought of as an algebraic move (with an arithmetic description in terms of continued fractions) that is related to the process of $v$-mutation described in \cite{M1}; see \S\ref{ss:organ}.  \MS

 \NI (ii) Our conjecture that no special rational point has a staircase
 is related to the question of whether $b=1/3$ has a descending staircase. It is shown in \cite[Lem.2.2.12]{MM} that when $b=1/3$  either there  is a descending staircase or
there is
$\varepsilon>0$ so that for $\acc(b)<z<\acc(b)+\varepsilon$ the capacity function $c_{H_b}(z)$ equals the obstruction from the class $\bE=3L-E_0-2E_1-E_{2\dots6}$.  
This obstruction plays a special role because, by \cite[Ex.32]{ICERM} and \cite[Lem.2.2.7]{MM},  it  goes through the accumulation point for all special rational $b$ except for $1/5$. We extend the conjecture  in Theorem~\ref{thm:main}~(iv) by claiming  that for all such special $b$ as well as $b=1/3$ the capacity function should be given by  this obstruction at points  just above the corresponding special rational $z$. (This result is now proven in \cite{MPW}.) \MS

\NI (iii) Our staircases need not be sharp in the sense of Casals--Vianna~\cite{CV}; in other words the inner corners between the staircase steps need not lie on the volume obstruction $z\mapsto V_b(z)$.  In fact,  descending stairs with $z$-accumulation points  $< 6$ are never sharp because of the obstruction coming from the class $\bE = 3L-E_0 - 2E_1 - E_{2\dots 6}$; see \cite[Ex. 32 and Fig.5.3.1]{ICERM}. We did not explore this property for the other staircases.  However, it is known that the Fibonacci staircase is sharp while the staircase at $b=1/3$ is not; see \cite{AADT}.  Notice also that, although \cite{M1} does use almost toric fibrations to construct a full filling at the lower endpoint $z_\infty$ of a blocked $z$-interval, this full filling occurs for the corresponding $b$-value  $\acc_\eps^{-1}(z_\infty)$ rather than at one of the $b$-values for which the blocking class is a staircase step.  \hfill$\er$
 
\end{rmk}

\begin{center}
\begin{figure}[h]
\begin{tikzpicture}
\begin{axis}[
	axis lines = middle,
	xtick = {5.5,6,6.5,7},
	ytick = {2.5,3,3.5},
	tick label style = {font=\small},
	xlabel = $z$,
	ylabel = $\lambda$,
	xlabel style = {below right},
	ylabel style = {above left},
	xmin=5.4,
	xmax=7.2,
	ymin=2.4,
	ymax=3.6,
	grid=major,
	width=3in,
	height=2.25in]
\addplot [red, thick,
	domain = 0:0.7,
	samples = 120
]({ (((3-x)*(3-x)/(2*(1-x^2))-1)+sqrt( ((3-x)*(3-x)/(2*(1-x^2))-1)* ((3-x)*(3-x)/(2*(1-x^2))-1) -1 ))},{sqrt(( (((3-x)*(3-x)/(2*(1-x^2))-1)+sqrt( ((3-x)*(3-x)/(2*(1-x^2))-1)* ((3-x)*(3-x)/(2*(1-x^2))-1) -1 )))/(1-x^2))});
\addplot [black, only marks, very thick, mark=*] coordinates{(6,2.5)};
\addplot [blue, only marks, very thick, mark=*] coordinates{(6.85,2.62)};
\addplot [green, only marks, very thick, mark=*] coordinates{(5.83,2.56)};
\end{axis}
\end{tikzpicture}
\caption{ This plot found in \cite[Fig.1.1.4]{ICERM} shows the location of the accumulation point $(z,\lambda)=(\acc(b), V_b(\acc(b)))$ for $0\le b< 1$.
The blue point with $b=0$ is at $(\tau^4, \tau^2)$ and is the accumulation point for the Fibonacci stairs defined in \cite{ball}.  The green point  
with  $z = 3+2\sqrt2=:a_{\min},  b=1/3$  
is the accumulation point for the 
stairs in  $H_{1/3}$,   and is the minimum of the function $b\mapsto\acc(b)$.  
The black point  with $z=6, b=1/5$  is the place where $V_b(\acc(b))$ takes its minimum. 
} \label{fig:1} 
\end{figure}
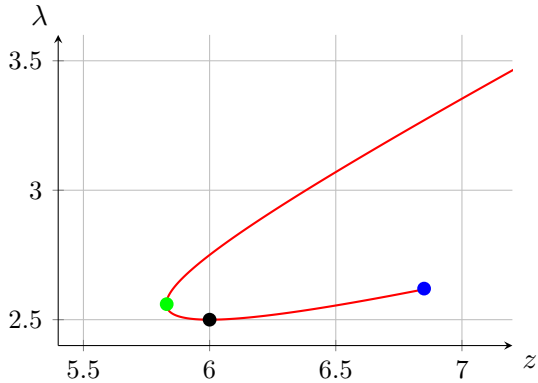
\end{center}

\subsection{Further results}\label{ss:mutation}
After summarizing some definitions and foundational results from \cite{ICERM, MM}, we state the  other main results of this paper.
Note that there are two approaches to calculating $c_{H_b}$. One is to work with ECH capacities as in  \cite{AADT}, which corresponds to identifying the $J$-holomorphic curves in $X\setminus E(1,z)$ counted by the ECH cobordism map. Here, we instead work with closed curves in a blowup of $H_b$, which in many cases neck-stretch to the ECH curves (see \cite{CGHM} and \S\ref{ss:stab}).

By \cite{Mell}, an ellipsoid $E(1,z)$ with rational eccentricity $z=p/q$ embeds into $\la H_b$ if and only if a certain finite collection $\sqcup_{i=1}^n B^4(w_i)$ of balls embeds into 
$\la H_b$.  This sequence $(w_i)$ is called the {\bf weight expansion} of $p/q$, see \eqref{eq:wpq}.
By \cite{MPolt}, the embedding of balls exists  if and only if there is a symplectic form $\om$ on the $n$-fold blowup $H_b\# n\ov{\CP}\,\!^{2} = \CP^2\# (n+1)\ov{\CP}\,\!^{2}$ of $H_b$ that takes the value $\la$ on the line, $\la b$ on the exceptional sphere $E_0$ of $H_b$,\footnote
{
This is the embedded $2$-sphere with self-intersection $-1$ obtained by 
blowing  $\CP^2$ up once, that is by removing an open $4$-ball and collapsing its boundary to a $2$-sphere via the Hopf map;
see \cite[Ch7.1]{MS} for more detail.}
 and $w_1,\dots,w_n$
on the other exceptional spheres $E_1,\dots, E_n$. Thus $\om$ should lie in the class $\al$ where $\al(L) = \la, \al(E_0) = \la b, \al(E_i) = w_i$.
As explained more fully in \cite{Mell}, it follows from
\cite{LL} that a class $\al \in H^2(\CP^2\# (n+1)\ov{\CP}\,\!^{2})$ has a symplectic representative if and only if
\begin{itemlist}\item [-] the volume is positive:  $\al^2>0$,
\item[-] the integral $\al(\bE)$ of $\al$ over every exceptional class $\bE$ in 
$\CP^2\# (n+1)\ov{\CP}\,\!^{2}$ is positive.
\end{itemlist}
Thus the significant constraints on the class $\al$ come from the {\bf exceptional classes}, which are defined to be the set of elements in $H_2(\CP^2\# (n+1)\ov{\CP}\,\!^{2})$ 
that are represented by symplectically embedded spheres of self-intersection $-1$.
These classes,
 denoted by
$$
\bE = dL - m E_0 - {\textstyle \sum_{i=1}^n} m_i E_i,
$$
are characterized by the fact that they 
satisfy the {\bf Diophantine equations}
\begin{equation}\label{eq:D}
3d-m=\sum_im_i+1 \mbox{ and } d^2-m^2=\sum_im_i^2-1,
\end{equation}
(these record certain restrictions on their first Chern class and self-intersection), and 
also reduce correctly under Cremona moves; see  \cite[Def.1.2.11ff]{ball}. 

Since $\al(\bE) = (d-mb)\la - \sum m_i w_i$,  
the condition $\alpha(\bE)>0$ implies that  each exceptional class $\bE$ with center $p/q$
 determines
an {\bf obstruction}
\[
\mu_{\bE,b}(p/q):=\frac{\sum m_i w_i}{d-bm}
\]
such that $\mu_{\bE,b}(p/q)\le \la$.
Therefore,
\[
c_{H_b}(p/q)\geq\mu_{\bE,b}(p/q) \mbox{ for all exceptional }  \bE.
\]
These obstruction functions extend to nearby $z$  
and, as in  \eqref{eq:muEb} below,
are piecewise linear.  Moreover, the
 above discussion implies that
\[
c_{H_b}(z)=\sup_{\bE\text{ exceptional}}\{\mu_{\bE,b}(z),V_{H_b}(z)\}.
\]
We say that an exceptional class $\bE$ is {\bf live at $z,b$} if it achieves this supremum, i.e. $\mu_{\bE,b}(z)=c_{H_b}(z)$.  Further, $\bE$ is called {\bf obstructive} at $z,b$  if
$\mu_{\bE,b}(z)> V_b(z)$.

It turns out that the most relevant obstructions are given by exceptional classes $\bE$ whose coefficients
$\mathbf{m}=(m_i)$
 equal the integral weight expansion 
$(q,\dots)$ of some rational number $p/q$. We call such classes {\bf perfect classes} and say that $p/q$ is their {\bf center}. If such a class is only known to satisfy the Diophantine equations, we say the class is {\bf quasi-perfect}.

If $\bE$ is quasi-perfect, \cite[Lem.16]{ICERM} shows that the obstruction function\footnote
{
Although we call this the obstruction function, we only know that
$\mu_{\bE,b}(z)\le c_{H_b}(z)$ for perfect classes $\bE$.}
  in the neighborhood of $p/q$ for which $\mu_{\bE,b}(z)>V_{H_b}(z)$  is given  by the formula
\begin{align}\label{eq:muEb}
\mu_{\bE,b}(z) = \frac {qz}{d-mb},\quad z\le p/q,\qquad \mu_{\bE,b}(z) = \frac {p}{d-mb},\quad z\ge p/q.
\end{align}
These obstructions are illustrated in 
 Fig.~\ref{fig:plotsIntro}.  
These functions are piecewise linear with break point (i.e. nonsmooth point) at the center $z=p/q$.  But  note that until we know that a given quasi-perfect class $\bE$ is in fact perfect, we cannot claim that  $\mu_{\bE,b}(z) \le c_{H_b}(z)$.
 Our staircases  are formed by infinite sequences of perfect classes whose break points $z_n$ (often called {\bf step centers} or simply {\bf steps})  converge. 
 Here are the key facts about these classes.

 \begin{lemma}\label{lem:basic}
 \begin{itemlist}\item[{\rm(i)}] If $\bE$ is perfect, $\mu_{\bE,b}$ is live at its center $z=p/q$ (i.e. $c_{H_b}(p/q) = \mu_{\bE,b}(p/q)$) when $b\approx m/d$.
\item[{\rm(ii)}] If $\bE$ is quasi-perfect, $\mu_{\bE,b}$ is obstructive at $z=p/q$ iff $|bd - m|< \sqrt{1-b^2}$.
 \end{itemlist} 
 \end{lemma}
 \begin{proof}  The claim in (i) is proved in \cite[Prop.21]{ICERM}; (ii) is proved in \cite[Lem.15]{ICERM}.
 \end{proof}

 We observed in \cite[\S2.2]{MM} that the degree coordinates of  all  quasi-perfect classes have the following  more precise form
\begin{align}\label{eq:pqt}
d &= \tfrac 18\bigl(3(p+q) + \eps t\bigr),\quad m = \tfrac 18\bigl(p+q + 3\eps t\bigr), \quad \eps \in \{\pm 1\},
\end{align}
where the positive integer $t$ is defined by $t: = \sqrt{p^2 - 6pq + q^2 + 8}$, and $\eps =1$ if and only if $m/d>1/3$.\footnote
{
No perfect class has  $m/d= 1/3$ by \cite[Lem.2.2.13]{MM}.}   For example if $p/q = [7;4] = \frac{29}{4}$ then $\bbm= (4^{\times 7}, 1^{\times 4})$, $t = 13$ and $(d,m) = (14,9)$.  
Thus we often describe a quasi-perfect class by an ordered subset of the ordered $6$-tuple $(d,m,p,q,t, \eps)$. We will also use the notation $\bE_{CF(p/q)}$, where $CF$ refers to the continued fraction. Conflating the notions of perfect classes, centers, and step centers, we will also refer to quasi-perfect classes as steps.

 By \cite[Prop.2.2.9]{MM}, a quasi-perfect class $\bE$ with center\footnote
{Note that $3+2\sqrt 2$ is the minimum value of $b\mapsto\acc(b)$.}
 $p/q> 3+2\sqrt 2=:a_{\min}$ is always a {\bf blocking class}, i.e. the  corresponding obstruction $\mu_{\bE,b}(z)$ 
is always nontrivial at its center point  $z = p/q$ when $b =  \acc_\eps^{-1}(p/q)$.\footnote
{
Here we take the branch of the inverse $\acc^{-1}$ given by the value of $\eps$; thus  
 $\acc_\eps^{-1}(p/q)>1/3$ exactly if $m/d>1/3$.} 
  This implies that the interval $$
  J_{\bE}: = \bigl\{b \ \big|\ \mu_{\bE,b}(\acc(b)) > V_b(\acc(b)\bigr\}
  $$  is an open neighborhood of $\acc_\eps^{-1}(p/q)$. We
  call this the {\bf $b$-blocked interval} of $\bE$, and
   define $\p^{+}(J_\bE):=\sup J_\bE$ and $\p^-(J_\bE):=\inf J_\bE.$  The interval $I_\bE:=\acc(J_\bE)$ is the corresponding {\bf $z$-blocked interval}.

We prove the following result  in \S\ref{ss:disj}. 
 
 \begin{prop}\label{prop:1}  ${\it Block}\subset [0,1)$ is the disjoint union of the intervals $J_{\bE}$ as $\bE$ ranges over the set of all perfect classes with centers $> a_{\min} = 3+2\sqrt2$.
 \end{prop}   
 
 \begin{rmk}\label{rmk:prop1} \rm The proof of Proposition~\ref{prop:1} allows us to find all perfect classes with center $>a_{\min}$; see Corollary~\ref{cor:disj1}.  In Lemma~\ref{lem:13perfect}, we show that the only perfect classes with centers $<a_{\min}$ are those appearing in the staircase of $c_{H_{1/3}}(z)$.  \hfill$\er$
 \end{rmk} 

As noted in Theorem~\ref{thm:main}, when $\bE$ is perfect and with center $> a_{\min}$ the endpoints of $J_\bE$ admit staircases. We now give an example of how the structure of the blocking classes relate to the staircases at the endpoints.

\begin{figure}[h]
    \centering
    \includegraphics[scale=0.37]{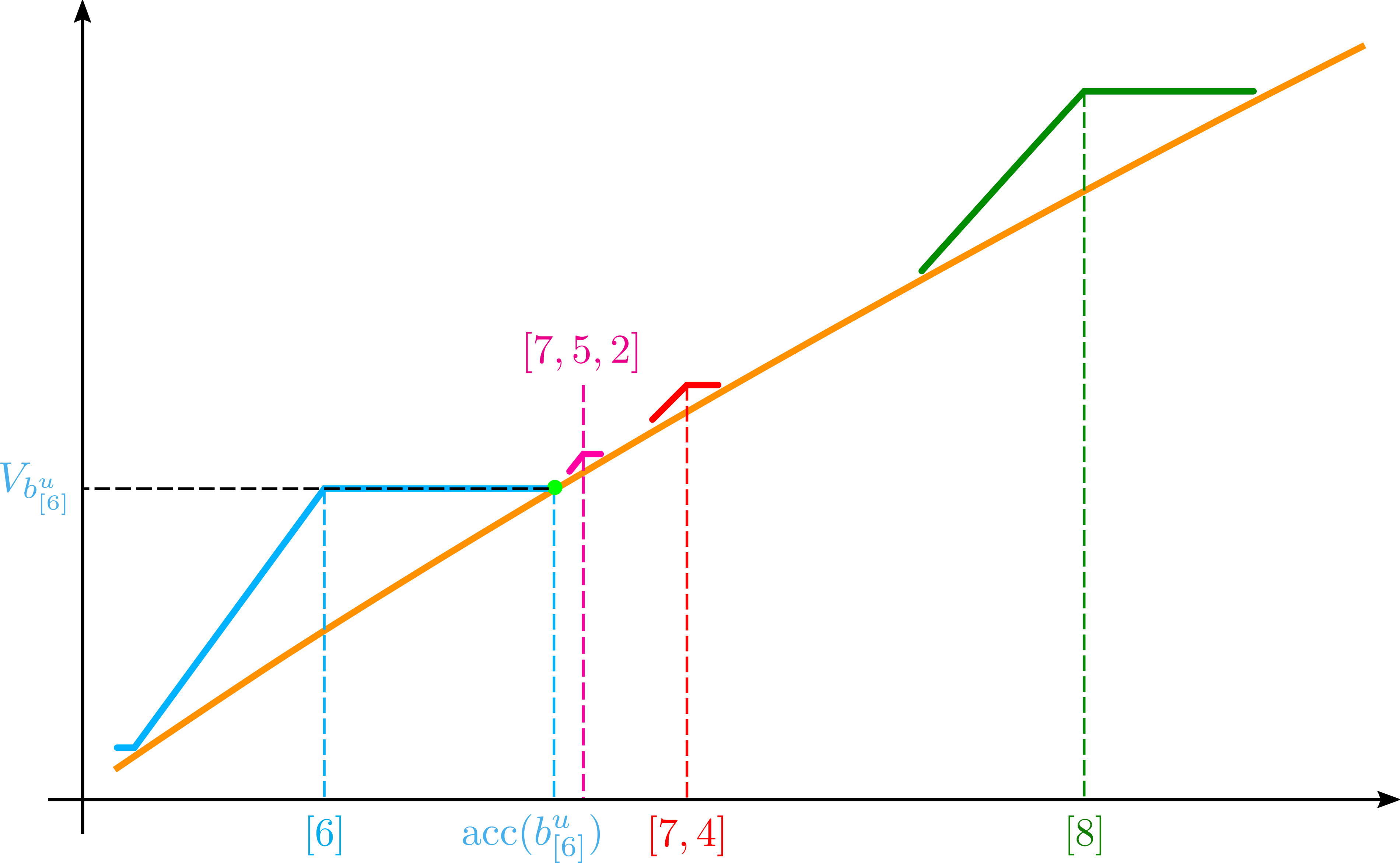}
    \includegraphics[scale=0.37]{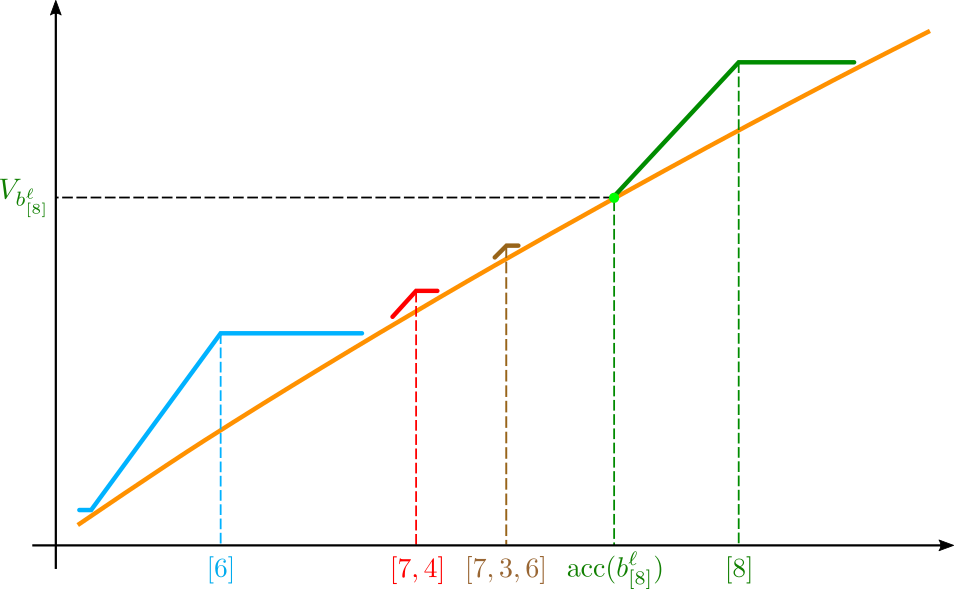}
    \includegraphics[scale=0.4]{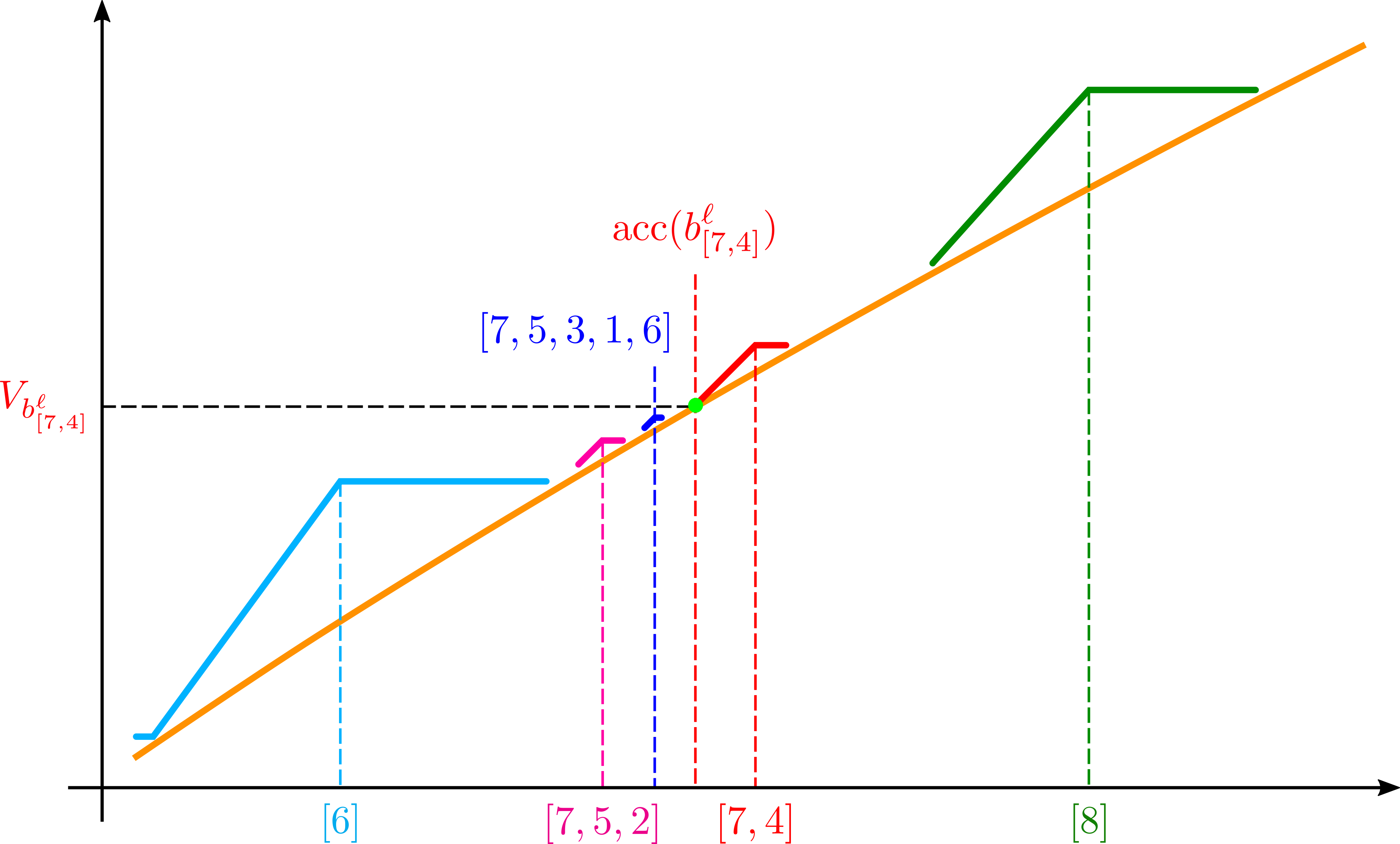}
    \caption{We have depicted part of three staircases to illustrate Example \ref{ex:mainEx}: the descending staircase with $b=\partial^+(J_{\bE_{[6]}})$ and the ascending staircases with $b=\partial^-(J_{\bB^U_1})$ and $b=\partial^-(J_{\bE_{[7;4]}})$. The orange curve is the volume obstruction and steps of the same color are given by the obstruction from the same perfect class, labeled by the continued fraction of its center. Note we only prove that the embedding function equals the solid lines on a neighborhood of the nonsmooth point, and we are not claiming that between those intervals there are no other obstructions (in fact, we know that there are).}
    \label{fig:plotsIntro}
\end{figure}

\begin{example}\label{ex:mainEx} \rm
This example is visualized in Figure~\ref{fig:plotsIntro}. As proved in \cite{ICERM}, there are two perfect classes $\bB^U_0=(d,m,p,q)=(3,2,6,1)$ with $t=3$  and $\bB^U_1=(4,3,8,1)$ with $t=5$ that both block an interval of $b$-values. The $b$-values at the left (resp. right) endpoint of these intervals have ascending (resp. descending) staircases. Let $b^u_{[6]}:=\p^+(J_{\bB^U_0})$ and $b^\ell_{[8]}:=\p^-(J_{\bB^U_1}).$ 
For each of these $b$-values, there is an infinite sequence of perfect classes $(\bE_k)_{k \geq 0}$, which are shown in \cite{ICERM} to be live for $c_b(z)$. The centers of these classes $p_k/q_k$ are determined by a recursion parameter $\nu$ and initial conditions 
$p_0/q_0,p_1/q_1$ such that for $x=p,q$: 
\[ x_{k+1}=\nu x_k-x_{k-1}.\]
For the staircase at $b^u_{[6]}$ (which is called 
 $\Ss^U_{u,0}$  in \cite{ICERM}), the initial conditions are $\bB^U_1, \bE_{[7;4]}$ with $\nu=3$ (note $\nu=3$ is the $t$-parameter of $\bB^U_0$);
and for the staircase at  $b^\ell_{[8]}$  (called $\Ss^U_{\ell,1}$ in \cite{ICERM}) the initial conditions are
$\bB^U_0,\bE_{[7;4]}$ with $\nu=5,$ (note $\nu=5$ is the $t$-parameter of $\bB^U_1$).  The obstructive functions are live and have a nonsmooth point at the center $z=p/q.$ See Figure~\ref{fig:plotsIntro} to visualize this. 

Key features of these staircases are:
\begin{itemize}
\item The step centers $\bB^U_1,\bE_{[7;4]},\dots$ of $\Ss^U_{u,0}$ decrease to $\acc(b^u_{[6]})$. When $b = b^u_{[6]}$, the class $\bB^U_0$ is live to the left of the accumulation point. 
\item The step centers $\bB^U_0,\bE_{[7;4]},\dots$ of $\Ss^U_{\ell,1}$ increase to $\acc(b^\ell_{[8]})$. When $b = b^\ell_{[8]}$, the class $\bB^U_1$ is live to the right of the accumulation point.
\item The staircases share the step $\bE_{[7;4]}$.
\end{itemize}
The features can be described using ``adjacency'' and ``$t$-compatibility'' properties. When such classes satisfy these properties, we say they form a generating triple and notate this as
$\Tt:=(\bB^U_0,\bE_{[7;4]},\bB^U_1)$. These properties are defined and discussed in \S\ref{ss:genTt}. 
This triple contains all of the necessary information to prove there are staircases at $b^u_{[6]}$ and $b^\ell_{[8]}$. For the notation in Theorem~\ref{thm:main1}, we let $\Ss^\ell_\Tt:=\Ss^U_{\ell,1}$ (resp. $\Ss^u_\Tt:=\Ss^U_{u,0}$) denote the staircase at $b^\ell_{[8]}$ (resp. $b^u_{[6]}$).

It turns out that the shared step $\bE_{[7;4]}$ is also a blocking class, and by using the formulas for the $x$-mutation and $y$-mutation defined above Proposition~\ref{prop:gener}, we can mutate the triple to get two new generating triples: 
\[
x\Tt:=(\bB^U_0,\bE_{[7;5,2]},\bE_{[7;4]}) \mbox{ and } y\Tt:=(\bE_{[7;4]},\bE_{[7;3,6]},\bB^U_1).
\]
Here, the middle entry $\bE_{[7;5,2]}$ in $x\Tt$ is the third step in the staircase $\Ss^U_{u,0}$ at $b^u_{[6]}$,  while the middle entry $\bE_{[7;3,6]}$ in $y\Tt$ is the third step in the staircase at 
$b^\ell_{[8]}$.
Just  as before, each of these new triples determines two staircases
that share one step, 
 and this allows us to propagate each triple to two new triples. This process continues forever. 
 
 Finally, we explain the effect of the symmetries. Applying the shift $S:p/q \mapsto (6p-q)/p$ to the centers of each of the classes in the triple $\Tt=(\bB^U_0,\bE_{[7;4]},\bB^U_1)$ gives three new classes: $S^\sharp(\Tt):=(S^\sharp(\bB^U_0),S^\sharp(\bE_{[7;4]}),S^\sharp(\bB^U_1))$. In fact, $S^\sharp(\Tt)$ is a generating triple as the symmetry preserves the compatibility conditions of a triple. As described above, this generating triple has two associated staircases, which were first shown to be live in \cite[Cor.60(iii)]{ICERM}. By repeatedly applying the shift, the triples $(S^i)^\sharp(\Tt)$ are also generating triples for all $i \geq 0$. There is another symmetry, the reflection $R$, that also sends generating triples to generating triples,
 but this relationship is a little more complicated to explain since $R$ is not globally defined and acts on $z$ by an order reversing transformation; see \S\ref{ss:symm} for more details. \hfill$\er$
\end{example}

We now make some definitions which are generalizations from the example to state Theorem~\ref{thm:main1}.

 \begin{definition}\label{def:prestair} A sequence of quasi-perfect classes $\bE_k = \bigl((d_k,m_k,p_k,q_k)\bigr)_{k\ge 0}$ is said to form a {\bf pre-staircase} if there are numbers $z_\infty>1$ and $b_\infty\in [0,1)$ such that
 $p_k/q_k\to z_\infty$ and $m_k/d_k\to b_\infty$.  The pre-staircase is said to be {\bf live} if, for some $k_0\ge 0$, the obstructions from the classes $\bE_k, k\ge k_0,$ are live at $b_\infty$, so that they
form a staircase in $c_{H_{b_\infty}}$.
\end{definition}

Thus a staircase is a live pre-staircase.    
A pre-staircase is called {\bf fake} if its classes are known not to be perfect; for  examples
see
 Remark~\ref{rmk:fakestair}. As noted in Lemma~\ref{lem:pseudst}, $z_\infty=\acc(b_\infty)$ for any pre-staircase.\MS

As explained in Example~\ref{ex:mainEx}, the staircases found it \cite{ICERM,MM} consist of a sequence $\Ss: = \big(\bE_k\bigr)_{k\ge 0}$ of perfect classes (with $\eps$ constant) whose coefficients $x_k: = p_k,q_k,d_k,m_k,t_k$ satisfy a recursion of the form
 $$
 x_{k+1} = \nu x_k - x_{k-1},\quad k\ge 1,\;\; \nu\ge 3.
 $$
 Thus the steps are determined by the two initial classes (called {\bf seeds}) together with the recursion coefficient $\nu$.  We proved that these pre-staircases are live for $b: = b_\infty = \lim_k m_k/d_k$ and have accumulation point  $\acc(b_\infty): = z_\infty= \lim_k p_k/q_k$. 
 Finally, it turned out that each staircase $\Ss$  is {\bf associated to a blocking class} $\bB$ with blocked $b$-interval $J_\bB$ and corresponding blocked $z$-interval $I_{\bB} = \acc(J_\bB)$ in the following way:
 \begin{itemlist}\item[-] 
for an ascending staircase, the limit point $(b_\infty, z_\infty)$ has $z_\infty$ equal to the lower endpoint $\p^-(I_\bB)$ (so that, if $m/d>1/3$, $b_\infty = \p^-(J_\bB)$, while if $m/d<1/3$, $b_\infty = \p^+(J_\bB)$);
\item[-] 
 for a descending staircase, the limit point $(b_\infty, z_\infty)$ has $z_\infty$ equal to the upper endpoint $\p^+(I_\bB)$ (so that, if $m/d>1/3$, $b_\infty = \p^+(J_\bB)$, while if $m/d<1/3$, $b_\infty = \p^-(J_\bB)$);
 \item[-]  the recursion parameter $\nu$ of the staircase equals the
 $t$-coordinate of  $\bB$.
 \end{itemlist}
Here $\p^+$, resp. $\p^-$, denotes the supremum (resp. infimum) of an open interval $J_\bB$ or $I_\bB$. Notice that  if $J_\bB\subset(1/3,1)$ then $\acc$ sends $\p^+(J_\bB)$ to $\p^+(I_\bB)$ and $\p^-(J_\bB)$ to $\p^-(I_\bB)$; otherwise it switches them. (For no $\bB$ do we have $J_\bB\ni1/3$ because $c_{H_{1/3}}$ has a staircase \cite{AADT}).
 
In \cite[Cor.3.2.3]{MM} we defined the {\bf staircase family} $\Ss^U$ to consist of the blocking classes $\bB^U_n: = (n+3,n+2,2n+6,1, 2n+3)$, for $n\ge0$, together with the two seeds\footnote
{
These seeds satisfy the  conditions in \eqref{eq:pqt} and should be considered as formally perfect, though clearly $\bE^U_{u,seed}$ does not correspond to a geometric class; see also Remark~\ref{rmk:quasitrip}.}  
 $\bE^U_{\ell,seed}: = (1,1,1,1,2)$, 
 $\bE^U_{u,seed}: = (-2,0,-5,-1,2)$ and the associated staircases:

 \begin{align}\label{eq:seed}
- & \mbox{ for each $n\ge 1$, the ascending staircases $\Ss^U_{\ell,n}$ with seeds 
 $\bE^U_{\ell,seed}, \bB^U_{n-1}$}; \\ \notag
 - &\mbox{ for each $n\ge 0$, the descending staircases $\Ss^U_{u,n}$ with seeds 
 $\bE^U_{u,seed}, \bB^U_{n+1}$}.
 \end{align}

As observed in \cite[Rmk.3.2.4(ii)]{MM} for each $n$ the staircases $\Ss^U_{u,n}$ and $\Ss^U_{\ell,n+1}$ share exactly one step $\bE_{[2n+7;2n+4]}$ with center at $[2n+7;2n+4]$. This is the next step after the two seeds. These are the classes seen in Example \ref{ex:mainEx} for $n=0$ and $n=1$.

The main result  in \cite{MM} was that there is a set of symmetries $S^iR^\de, i\ge 0, \de\in \{0,1\}$ that act on  $z=p/q$  
by fractional linear transformations and fix $t$, so that the action extends to the $d,m$ coordinates of perfect classes.  
Here $S$ is the shift $p/q\mapsto (p-6q)/q$ that fixes the point $a_{\min}= 3+2\sqrt2$, and $R$ is the reflection $p/q\mapsto (6p-35q)/(p-6q)$ that fixes $7$. Both $S$ and $R$ change the sign of $\eps$.
We showed that the image of $\Ss^U$ under each of these transformations $T$ is another staircase family; in particular $T$ takes blocking classes to blocking classes, and staircases to staircases.
(For example, the reflection $R$  takes the descending stairs $\Ss^U_{u,0}$ to the Fibonacci stairs in $\CP^2 = H_0$; see \cite[Rmk.3.2.4,Cor.3.2.7]{MM}.) This reduces the study of staircases to those with $6<\acc(b)$ and $b>1/3$.
\MS

 \begin{figure}[htbp]\label{fig:2} 
 \vspace{-1.1in}  
 \includegraphics[width=7in]{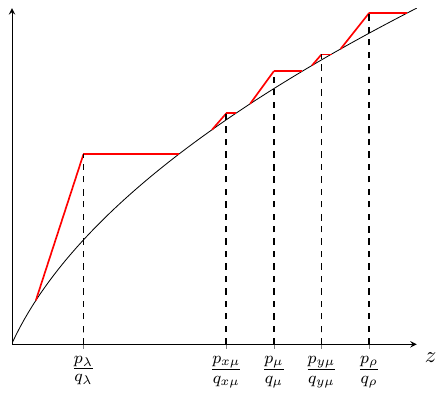} 
 \vspace{-5.7in}
    \caption{This schematic figure shows the graph of the function $z\mapsto  V_b(z)$ together with the
 obstruction functions given by a triple $\Tt$ and its left and right mutations for some appropriate value of $b$.
 }
 \end{figure}

A central result of the current paper can be stated as follows; the required definitions 
and the proof of (i), (ii) may be found in \S\ref{ss:genTt} and \S\ref{ss:symm}, while (iii) is proved  in \S\ref{ss:disj}.

\begin{thm} \label{thm:main1}\begin{itemlist}\item[{\rm (i)}]  Each triple\footnote
{
Here the subscripts $\la,\mu,\rho$ stand for \lq left', \lq middle' and \lq right'.} of classes $$
\Tt^n_*: = \bigl(\bE_\la, \bE_\mu, \bE_\rho\bigr): =\bigl(\bB^U_{n},  \bE_{[2n+7;2n+4]}, \bB^U_{n+1}\bigr)
$$
is a generating triple and generates two other triples called its \textbf{left and right mutations} 
 $$
x\Tt^n_*: = \bigl(\bE_\la, \bE_{x\mu},  \bE_\mu\bigr),\qquad y\Tt^n_*: = \bigl(\bE_\mu, \bE_{y\mu} \bE_\rho\bigr),
$$
which also are generating triples, and hence can be mutated further to form new triples $\Tt$. 
Moreover, each such  triple $\Tt$ determines two staircases, 
 \begin{itemlist}\item[-] an ascending staircase $\Ss^\Tt_\ell$  with seeds
$\bE_\la, \bE_\mu$ and blocking class $\bE_\rho$, 
\item[-] a descending staircase $\Ss^\Tt_u$ with seeds
$\bE_\rho,  \bE_\mu$ and blocking class $\bE_\la$.
 \end{itemlist}

\item[{\rm (ii)}]   The symmetries act on these triples.  Moreover the symmetries are compatible with mutation.  More precisely, if a symmetry takes $\Tt$ to $\Tt'$  and preserves (resp. reverses) the order on the $z$-axis, then it takes $x\Tt, y\Tt$  to  $x\Tt', y\Tt'$ (resp. $y\Tt', x\Tt'$).

\item[{\rm (iii)}]   Every perfect class with center $> a_{\min}$    is a step in at least one ascending and one descending  staircase that is the image under an appropriate symmetry of one of the staircases $\Ss^\Tt_\ell, \Ss^\Tt_u$ in (i) above.
    \end{itemlist} 
 \end{thm}

  In view of this result, we make the following definition. 
 
 \begin{definition}\label{def:complete} The {\bf complete family} $\Cc\Ss^U$ consists of the staircase family $\Ss^U$ defined as above, plus all staircases $\Ss^\Tt_\ell, \Ss^\Tt_u$ determined as in Lemma~\ref{lem:2} by any triple
$\Tt$ derived from one of the basic triples $\bigl(\bB^U_n, \bE_{[2n+7;2n+4]}, \bB^U_{n+1}\bigr)$.
Any staircase that is the image of 
$\Ss^\Tt_\ell$ or $ \Ss^\Tt_u$ under a symmetry $S^iR^\de$, where $i\ge 0, \de\in \{0,1\}$, is called a {\bf principal staircase.}
\end{definition}

Theorem~\ref{thm:main1} parts~(ii) and ~(iii) imply that each  perfect class with center $> a_{\min}$ belongs to a unique family $(S^iR^\de)^{\sharp}(\Cc\Ss^U)$.  We will see later that the principal staircases are distinguished from those described in Remark~\ref{rmk:decimal}~(ii) below by the fact that each is recursively defined and has a blocking class.  

\begin{rmk} \label{rmk:decimal}  \rm (i)
Instead of thinking of the complete family as a set of interconnected staircases, we can think of it as a countable family of classes ordered according to the position of their centers, labelled by ternary decimals; for details see \S\ref{ss:ternary}.  As explained in \cite{M1}, one can 
associate to each triple $\Tt$ in $\Cc\Ss^U$ an almost toric fibration whose base diagram is a quadrilateral $Q_\Tt$ in such a way that
 there is direct correspondence between the left/right derived triples $x\Tt, y\Tt$ and two  of the possible mutations of $Q_\Tt$.  This fact is a key step in the proof that the tuples $\bE = (d,m,p,q,t)$ we consider do in fact correspond to perfect classes that are represented by exceptional spheres.
\MS

\NI (ii) {\bf (Further staircases)} 
Once we give finite decimal  labels to the classes in $\Cc\Ss^U$,  it  becomes clear that they  can be organized into different convergent sequences, one for each infinite ternary decimal whose expansion does not contain the digit  $1$.    For a precise description, see Definition~\ref{def:newstair}.    The principal staircases $\Ss^\Tt_\ell, 
\Ss^\Tt_u$ described above are those that exist at the endpoints of the intervals in {\it Block}, while these new staircases correspond to all of the uncountably many other unblocked $b$-values in the Cantor-type set $[0,1)\less {\it Block}$ except the special rational points. 
  We will prove in \S\ref{ss:uncount}  that all the steps in the new ascending pre-staircases are live for the limiting $b$-value, while in the descending case all but finitely many are live for the limiting $b$-value.
 \hfill$\er$
  \end{rmk}

  \begin{rmk}{\bf (Properties of Staircases)}\label{rmk:main0} \rm 
(i)  The capacity functions $c_{H_b}(z), b\in [0,1),$  are not entirely determined by perfect classes, since there
 are other obstructions given by exceptional classes that are not perfect. For instance, 
the exceptional class $B=6L-3E_0-\sum_{i=1}^7 2E_i$ is not a perfect class even though the coefficients of the $E_i$ are multiples of the weights of $7$. One can check that this class is live at $z=7$ for some values of $b$. Furthermore, it is a blocking class, but the $b$-interval that it blocks is contained in the interval blocked by the perfect class $\bB^U_0=
3L-2E_0-\sum_{i=1}^6 E_i$. See Remark~\ref{rmk:except} for more discussion. Because it is not clear how non-perfect exceptional classes relate to the symmetries, we cannot claim that  the capacity functions themselves are invariant under the symmetries,  but  only that the images of the perfect staircase classes remain live.  This question can also phrased in terms of ECH capacities: is it possible to understand the action of the symmetries from the point of view of ECH capacities?  

\MS

\NI (ii)  It turns out that if $H_b$  has a staircase then it has one whose steps are given by  perfect classes, as long as $b$ is not a special rational point.  Moreover, if $b$ is blocked, it is blocked by a perfect class.  As we show in Proposition~\ref{prop:stab}, the
 argument in \cite{CGHM} then applies to show that such staircases  stabilize, at least to dimension $6$.  Thus  the function $$
c_{H_b,1}(z) = \inf\ \bigl\{ \la\ \Big| \ E(1,z)\times \R^{2} \se \la H_b\times \R^{2}\bigr\} 
$$
has a staircase for each $b\notin {\it Block}$ that is not one of the special rational points, and has no staircase if $b\in {\it Block}$.  It would be interesting to explore how this fits in with the stabilized folding construction in \cite{H,CGHS} for example.  In the case of the ball or the monotone polydisk, the staircases ascend and the stabilized embedding function is conjectured to exhibit a \lq phase change' at the accumulation point of  the staircase. For example in the case of the ball 
it is conjectured to agree with the folding curve $z\mapsto \frac{3z}{1+z}$ for $z> \tau^4$.  Notice that in this case the folding curve crosses the volume constraint precisely at the accumulation point.   However, there are some 
 $H_b$ with  descending stairs that stabilize, and so something rather different must happen in this case. \hfill$\er$
\end{rmk}

\begin{remark}{\bf (Further developments and possible generalizations)} \label{rmk:fdgen} \rm
(i) Since this paper was first posted, there have been further developments about the Hirzebruch surface. In \cite{MPW}, the authors prove that $b=1/3$ admits only an ascending staircase and that the special rational points do not have descending infinite staircases. These points complete the leftover cases of Theorem~\ref{thm:main} (ii) and (iv). Furthermore, they prove the only quadratic irrational numbers that might be staircase accumulation points are the endpoints of connected components of $Block$. By the fact that the accumulation point function is the root of a quadratic equation, this proves that $b=0, 1/3$ are the only rational values of $b$ with infinite staircases, proving \cite[Conj.~1.20]{AADT} for Hirzebruch surfaces. \MS
   
   \NI(ii) The connections between the Hirzebruch surface and the polydisk have also become more clear. In \cite{spur} following work of \cite{Usher}, the authors provide evidence that the structure of Theorem \ref{thm:main} also occurs for $S^2(1)\times S^2(\beta)$.  In \cite{Usher}, Usher considered ellipsoid embeddings into the polydisk $P(1, b)$ (which by \cite[Thm 1.4]{AADT} is equivalent  to $S^2(1) \times S^2(\beta)$).  While he did not use the language of blocking classes, we have verified that the staircases that he found do lie at the endpoints of blocked intervals determined by  perfect classes. Looking at a few examples, it seems that the other endpoint of the blocked interval has a corresponding descending staircase. (Note that  Usher only considered ascending staircases.)   He also found symmetries arising from the relevant Diophantine equation that are closely related to those of \cite{MM} that are discussed in \S\ref{ss:symm} below. The further work about the polydisk in \cite{spur} suggests: 
 \begin{itemize}
 \item Given a staircase in $c_{H_b}$,
 there should be  an infinite staircase for $S^2(1)\times S^2(\beta)$, where
 \[
 \beta=\acc_{S^2}^{-1}\circ CF^-\circ\acc.
 \]
 Here $\acc_{S^2}^{-1}$ denotes the accumulation point function for $S^2(1)\times S^2(\beta)$ and $CF^-$ denotes the operation which sends a number with continued fraction $[s_0,\dots,s_n]$ to $[s_0-1,\dots,s_n-1]$.
 \item  In the cases that he studied, Usher conjectured in \cite[Conj.~4.23]{Usher} a complete description of the ellipsoidal capacity function of $S^2(1)\times S^2(\beta)$ below the accumulation point.
 The method of proof in \cite{spur} via almost toric mutations suggests that 
 this conjecture is true, and that a similar description holds for the ascending principal staircases of $H_b$.
 \end{itemize}
 This provides evidence that the work done here to find all perfect classes is easily generalizable to the case of the polydisk. \MS

     \NI(iii)
In \cite{M2}, a similar structure is  found for 
embeddings into $X_{a,b}: = $
$ \CP^2(1) \# \oCP^2(a) \# \oCP^2(b)$   where $0 < a \leq b$ and $a+b < 1$. In this example, rather than having blocked intervals, there are blocked regions of $a,b$ values, and staircases have been found on their boundaries. Now, the regions blocked by two different  perfect classes may intersect, in contrast with the disjointness found in Proposition~\ref{prop:disj}. Thus, the curve that bounds a particular blocked region is not entirely made up of $b$-values with staircases: indeed it seems that the set of such parameter values with  staircases has much the same structure that we find for $z$-intervals such as $[6,8]$. Preliminary work suggests that there is an analogous definition of a generating triple, which we can mutate to find more generating triples resulting in a similar fractal of staircases. \MS

   \NI (iv) More generally, we expect our results to generalize to the examples in \cite{AADT} that compactify to a $k$-fold blowup of $\CP^2$, where $k\le4$ and we allow for irrational blowups. 
Simple extensions of these situations are still mysterious. It remains unclear if there are any staircases in a five-fold blowup of $\CP^2$ 
or in a convex toric domain such as an irrational  ellipsoid $E(1,x), x\in \R\less \Q,$ with irrational normal vectors.  For forthcoming work on these questions, see Cristofaro-Gardiner--Magill--McDuff~\cite{CGMM}.

\NI(v) It would also be interesting to know whether echoes of the combinatorial structures discovered here can be perceived in  other situations.  For example,  Casals--Vianna  develop in  \cite{CV} another approach to some of the staircases considered in \cite{AADT} via almost toric fibrations, emphasizing the relation to tropical geometry and quiver combinatorics.  Even though we are not working in a monotone manifold, it is possible that the structure of the generating triples developed here, with their tight connections to
the almost toric fibrations in \cite{M1},  have repercussions in these areas.
\hfill$\er$

\end{remark}

\subsection{Nature of the proofs}\label{ss:organ}
The proofs involve two kinds of arguments; algebraic/arithmetic to show that the classes $\bE$ do satisfy all the required compatibility conditions, and geometric  to show that the corresponding obstruction functions $\mu_{\bE,b}$ are live for the appropriate $b$-value. The point here is that a quasi-perfect class  $(d,m,p,q,t)$ may not have $\bE\cdot \bE'\ge 0$ for every class $\bE'\ne \bE$ that is represented by an exceptional sphere; see Corollary~\ref{cor:pos}.   If $\bE$ fails this test,\footnote{
 Although there certainly are plenty of fake classes (i.e. quasi-perfect classes that are not perfect),  
Remarks~\ref{rmk:fakestair}  and~\ref{rmk:fakesymm}  give intriguing hints  that one might be able to express the distinction  in purely  arithmetic terms.  For example, there may be no set of (integral)  fake classes that satisfy all the compatibility conditions    required of a generating triple.}
  then although the function 
 $\mu_{\bE,b}$ is obstructive at its center $p/q$  for $b = m/d$  (i.e. $\mu_{\bE,m/d}(z)> c_{H_{m/d}}(p/q)$) there is no guarantee that 
 \begin{itemize}\item[(a)]  it gives the maximal obstruction at $b = m/d, z=p/q$, or that
 \item[(b)]  it still gives the maximal obstruction at $z=p/q$ when $b = b_\infty$, the limiting $b$-value of the
 proposed  staircase.
 \end{itemize}
In fact, once we know (a) for all its steps,
 then (at least in the case of the principal pre-staircases) results that can be quoted from \cite{MM} imply that
 the pre-staircase is live provided that
 the slope estimate \eqref{eq:MDTineq0} holds. 
Moreover, by Lemma~\ref{lem:basic} we can establish (a) by showing that  $\bE$ is perfect.
The standard way to do this
 is to show as in~\cite{ICERM,MM} that 
 its components  reduce to those of a known exceptional class under Cremona transformations.  It turns out (see \cite[\S4.1]{MM}) that the symmetries act very nicely on these components, so that it suffices to prove this for the staircases in the family $\Cc\Ss^U$.   This was done by hand in \cite{ICERM} for the staircases 
 with blocking classes $\bB^U_n, n\ge 0$.  
However, this approach   seems impractical in the current context,  since the coefficients of the staircase classes in the complete family $\Cc\Ss^U$ increase so rapidly and the reduction process does not seem to behave nicely under mutations.

Instead we first  develop a more efficient way to determine when a quasi-perfect class in the complete family 
$\Cc\Ss^U$ is perfect.  Namely, we show in Lemma~\ref{lem:live2}  that a class $\bE$ with $m/d>1/3$ is perfect provided that  the lower endpoint $z^-_\infty: = \p^-(I_{\bE})$ is unobstructed, i.e.
$c_{H_{b^-_\infty}}(z^-_\infty) = V_{b^-_\infty}(z^-_\infty)$, where $\acc({b^-_\infty}) = z^-_\infty$.   
We then apply the main result in Magill~\cite{M1} which shows that these points are unobstructed by   using almost toric fibrations to construct a full filling for these values of $b$ and $z$. 
This is possible because there is an ascending  sequence of quasi-perfect classes in $\Cc\Ss^U$ (i.e. a  pre-staircase)  with limit 
$(b^-_\infty, z^-_\infty)$.   To do this, one starts with the toric model of  $H_b$ and performs a sequence of mutations at three of its corners.  There are three possibilities here: when suitably iterated, a mutation at the corner in the first quadrant (called a $v$-mutation) corresponds to the translation $[6,8]\to [2n+6,2n+8]$,  
 while mutations at the points on the $x$- and $y$- axis turn out to correspond precisely to the corresponding moves for the triples $\Tt$ that are described in Proposition~\ref{prop:gener}.

The upshot is that we use the existence of the whole  interweaving family of quasi-perfect classes to prove inductively that all the classes in the family are in fact perfect.  One then needs additional arguments to show that the staircase steps are live at the limiting $b$-value.  This turns out to be straightforward in the case of the ascending  staircases, but more problematic for the descending ones; see \S\ref{ss:live} and \S\ref{ss:uncount}.  It is important in several places of our arguments that the limiting values $(b_\infty, z_\infty)$ of the principal  staircases are both irrational.
\MS

  \NI {\bf Acknowledgements}    We thank all the other members of our  original group 
  M. Bertozzi, T. Holm, E. Maw,  G. Mwakyoma, and  A. R. Pires for helping us initiate this work, as well as AWM and ICERM for providing the initial organization.  Nicki Magill also thanks Tara Holm as her research advisor for introducing her to almost toric fibrations and encouraging her to use them as a proof method. We also thank the anonymous referee for very helpful comments.
  
\section{Description of the classes}\label{sec:triple}

This is the algebraic heart of the paper.
We  define the notion of a generating triple in Definition~\ref{def:gentr} and show that they propagate under left/right mutations in Proposition~\ref{prop:gener}.  Along with Lemma \ref{lem:2}, this allows us to prove Theorem \ref{thm:main1} (i). In \S\ref{ss:ternary} we describe the structure of the ensuing set of interwoven classes, and prove in Proposition~\ref{prop:dense} that they block a dense set of $z$-values (and hence a dense set of $b$-values).  The non-recursive pre-staircases 
are described in Definition~\ref{def:newstair}.  Finally, we show in Proposition~\ref{prop:symm} that 
 the symmetries from \cite{MM}  act on the set of generating triples, and hence on the set of pre-staircases. 
 
\subsection{Generating triples}\label{ss:genTt}

All staircase classes  $\bE$ are quasi-perfect; that is they are given by tuples $(d,m,p,q)$  where $$
\bE = d L - mE_0 -\sum_{i\ge1} m_iE_i, \qquad (m_1,m_2,\dots ) = W(p/q),
$$
where $W(p/q)$ is the integral weight expansion of $p/q$; see \cite[Ex.8]{ICERM} or \eqref{eq:Wpq} below.
It is easy to check  that if $\bx: = (p,q,t)$ satisfies the equation $p^2-6pq + q^2 +8= t^2$ and $t>0$, then the quantities $(d,m)$ defined by 
 \begin{align}\label{eq:formdm}
 d: = \tfrac 18( 3(p + q) +  \eps t), \quad  m: =  \tfrac 18((p + q) + 3\eps t),\quad \eps\in \{\pm 1\}
\end{align}
 satisfy the basic identities $3d = m+p+q$ and $d^2-m^2 = pq-1$ 
 that correspond to $c_1(\bE)=1, \bE\cdot\bE=-1$; see \cite[\S2.1]{ICERM}
   Moreover, by  \cite[Prop.2.2.9]{MM} the class is uniquely determined by $p/q$.  In other words,   given $p,q,t$ satisfying $p^2-6pq + q^2 +8= t^2$,
  the expressions for $d,m$ in \eqref{eq:formdm} are integral for at most one value of $\eps$.  
Thus we can think of $\bE$ as an integral point on the quadratic surface $$
 X_\Z: =  \bigl\{\bx \in \Z^3 \ | \ \bx^T A \bx = 8\bigr\},
 $$
  where
 $A$ is the symmetric matrix
 \begin{align}\label{eq:matrix}
&A : = \left(\begin{array}{ccc} -1&3&0\\3&-1&0\\0&0&1\end{array}\right),\quad  \bx: = 
 \left(\begin{array}{c} p\\q\\t\end{array}\right).
 \end{align}
Such classes $\bE$  
are  center-blocking\footnote
{
that is, $\mu_{\bE,b}(p/q) > V_b(p/q)$ when $\acc(b) =p/q$.}
 if $p/q> 3+2\sqrt2 = a_{\min}$ by \cite[Prop.2.2.9]{MM}.  Moreover, because the case $\eps=1$ corresponds to values of $m/d>1/3$ while $\eps = -1$ corresponds to  $m/d<1/3$, we often omit $\eps$ from the notation.\footnote
{
By \cite[Lem.2.2.13]{MM}, there are no perfect classes with $m/d=1/3$.}

 \begin{definition}\label{def:Tt}    Two quasi-perfect classes $\bE : = (d,m,p,q,t, \eps), \bE': =(d',m',p',q',t', \eps')$ 
 are said to be
 {\bf adjacent} if $\eps=\eps'$ and if after renaming so that $p/q < p'/q'$ (if necessary), 
 the following relation  holds:
   \begin{align}\label{eq:adjac}
 (p+q)(p'+q') - tt' = 8 pq'.
\end{align}
Further,  they are called
 {\bf $t''$-compatible} if $\eps = \eps'$ and
  \begin{align}\label{eq:tcompat}
&tt'-4t''= p p' - 3(pq'+qp') + qq', \qquad \mbox{ i.e. }\ \  \bx^T A \bx' = 4t''.
\end{align} 
We make corresponding definitions for a pair
 $\bx, \bx' \in X_\Z$ to be $t''$-compatible or  adjacent.
Further we write $ \bx<\bx'$ (or $\bE< \bE'$) to denote that\footnote
{
Note that this condition has nothing to do with the relative size  of the ratios $p/q, p'/q'$.}
 $p<p',  q<q', t<t'$; and
 say that $\bx\in X_\Z$ is {\bf positive} if $p,q, t>0$.
 \end{definition}

 Here is a useful result about the relation between adjacency and $t$-compatibility.

  \begin{lemma}\label{lem:recur0} \begin{itemlist}\item[{\rm (i)}]
 Suppose that the points $\bx_0, \bx_1 \in X_\Z$ are $t$-compatible for some $t\ge 3$ and  have $ \bx_0<\bx_1$.
Then  $\bx_2: = t \bx_1 - \bx_0$ is positive and in $X_\Z$.  Also, $\bx_1<\bx_2$ and the pair $\bx_1,\bx_2$  is $t$-compatible.
Further, if $\bx_0, \bx_1$ are adjacent, so are $\bx_1, \bx_2$.
Thus,  if  $\bE_0, \bE_1$ satisfy $p_0<p_1,q_0<q_1, t_0<t_1$ and are adjacent and  $t$-compatible, then so are the components of all successive pairs in the sequence obtained from $\bE_0, \bE_1$ by $t$-recursion.

\item[{\rm (ii)}]  The adjacency condition for classes with $p/q<p'/q'$ can also be written in  the following equivalent forms
\begin{align}\label{eq:adjac1}  
(3m'-d') d &= (m'-q')p + m'q,\quad\mbox{ or}\\ \notag
d'd-m'm &= pq'.
\end{align}

\item[{\rm (iii)}]   If $\bE, \bE'$ are adjacent, then they are $t''$-compatible exactly if
 $$
| p'q-pq'| =t''.
 $$
 \end{itemlist}
\end{lemma}
\begin{proof}
Since  $A$ is symmetric,  
$$
(t \bx_1 - \bx_0)^T A (t \bx_1 - \bx_0) = t^2 \bx_1^T A \bx_1 - 2t \bx_1^T A \bx_0 + \bx_0^T A \bx_0 = 8
$$
exactly if  $8 t^2 = 2t \bx_1^T A \bx_0$, which holds by \eqref{eq:tcompat}. Thus $\bx_2\in X_\Z$.  
Further \eqref{eq:tcompat} holds for $\bx_2: = t \bx_1 - \bx_0, \bx_1$ because
$$
(t \bx_1 - \bx_0)^T A \bx_1 = 8t - \bx_0^T A \bx_1  = 4t,
$$
again by \eqref{eq:tcompat}.  Thus $\bx_1, \bx_2$ are $t$-compatible.

To see that they are adjacent, it is useful to introduce\footnote
{
Here we are following a suggestion of Ana Rita Pires, exhibiting
the adjacency condition as an asymmetric version of the equation $\bx^T A \bx = 8$ that defines the set of quasi-perfect classes.}
the matrix
 \begin{align*}
&B : = \left(\begin{array}{ccc} -1&7&0\\-1&-1&0\\0&0&1\end{array}\right).
 \end{align*}
Then $A = \frac 12(B+B^T)$ and it is easy to check the following:
\begin{itemlist}\item[-]  the class represented by $\bx: = (p,q,t)^T$ is  quasi-perfect iff $\bx^T B \bx = 8$;
\item[-]  the classes  represented by $\bx, \bx'$ are adjacent iff, when ordered so that $p/q< p'/q'$ we have
$\bx^T B \bx' = 0$; and
\item[-] the classes  represented by $\bx, \bx'$ are  $t''$-compatible iff
$\bx^T (B+ B^T) \bx' = 8t''$. 
\end{itemlist}
Thus
if  $p_0/q_0 < p_1/q_1$, the points $\bx_0, \bx_1$ are adjacent precisely if
$\bx_0^T B \bx_1 = 0$.   But then $p_1/q_1<p_2/q_2$ and the adjacency condition for $\bE_1,\bE_2$ is  
$\bx_1^T B \bx_2 = 0$, which holds because 
$$\bx_1^T B \bx_2 = \bx_1^T B (t\bx_1 - \bx_0) =  8t -\bx_1^T B \bx_0= 8t - \bx_0^T (B^T + B)\bx_1 =0
$$
since $\bx_0<\bx_1$ and  $\bx_0,\bx_1$ are adjacent and $t$-compatible.
This proves (i).
\MS

Now consider (ii).  To see that the two conditions in \eqref{eq:adjac1} are equivalent use the relation $p+q=3d-m$ to obtain
$$
(3m'-d') d = (m'-q')p + m'q = m'(3d-m) - pq',
$$
 and then simplify.   Next use the formulas in~\eqref{eq:formdm} to obtain
\begin{align*}
64(d'd-m'm)& =  \bigl(3(p+q)+\eps t\bigr) \bigl(3(p'+q')+\eps t'\bigr)-  \bigl((p+q)+3\eps t\bigr) \bigl((p'+q')+3\eps t'\bigr)\\
& =
8(p+q)(p'+q') - 8tt'.
\end{align*}
It follows easily that \eqref{eq:adjac} is equivalent to the conditions in \eqref{eq:adjac1}.
\MS

Finally, to prove (iii), we may assume without loss of generality  that $p/q<p'/q'$.   Then
rewrite the LHS $tt'-4t''$ of \eqref{eq:tcompat}, using the expression for $tt'$ from \eqref{eq:adjac} and simplify.
\end{proof}

The following remarks explain the geometric significance of the  $t$-compatibility and adjacency conditions.

 \begin{rmk}\rm\label{rmk:recur0} {\bf (Pre-staircases)} (i)
  If $\bE_0, \bE_1$ are $t'$-compatible for some $t'\ge 3$ and $\bE_0< \bE_1$, 
  then by  Lemma~\ref{lem:recur0} the tuples obtained from them by the recursion
\begin{align}\label{eq:recurx}
 x_{k+1} = t' x_k - x_{k-1},  \qquad\mbox{ where }\  x \ =d,\ m,\ p,\ q,\ t 
 \end{align}
are integral, and
 also satisfy the Diophantine identities;\footnote
 {The linear identity $3d-m = p+q$ is automatic for any such recursively defined sequence, but the quadratic identity $d^2-m^2 = pq-1$ holds only under the  $t$-compatibility condition; see \cite[Lem.65]{ICERM} for a similar result.}
 moreover, all successive pairs $\bE_k, \bE_{k+1}$ are $t'$-compatible.    Thus, this collection of classes forms a {\bf pre-staircase} in the sense of Definition~\ref{def:prestair}.\footnote
 {
 In \cite[Def.46]{ICERM}, a pre-staircase  is defined to
  consist of a recursively defined sequence of quasi-perfect classes that satisfy a linear relation (or equivalently are all adjacent to a fixed class called its blocking class).  The definition in \cite[\S1.2]{MM} is very similar, though the 
  requirement for the linear relation is stated separately.  Our current use of the word is 
  more general.}
   and we say that $\bE_0, \bE_1$ are its {\bf seeds}.
 Further, by \cite[Lem.3.1.4]{MM}, 
 there is $\la = \la(t')>1$ such that for each $x \ =d,\ m,\ p,\ q,\ t$  there are  constants $X \ = D, M,P,Q,T$  such that 
 $x_k = X\la^k + \ov X \la^{-k}$.  Here  we write
 $X: = X'+X''\sqrt \si\in \Q(\sqrt\si)$, where $\si = (t')^2-4$, 
 and define  $\ov{X}: = X'-X''\sqrt \si$.    Moreover $X',X''\in \Q$ are  determined by $x_0,x_1$ via the formulas
\begin{align}\label{eq:XX'}
X' = \frac{x_0}{2}, \quad X'' =  \frac{2x_1 -  t' x_0}{2\si}.
\end{align}
In particular this implies that the ratios   $p_k/q_k,m_k/d_k$ converge $O(\la^{-2k})$ to the limits  $P/Q, M/D$.  Further, in the limit the Diophantine equations satisfied by $(d_k,m_k,p_k,q_k)$ simplify to
 \begin{align}\label{eq:DiophPQ}
  3D-M = P+Q,\qquad  D^2-M^2 = PQ.
 \end{align}
Note also that, if $t'\ge 3$ then $\si = (t')^2-4$ is not a perfect square,  and it follows from \cite[Cor.3.1.5]{MM} 
that the limits $z_\infty = P/Q, b_\infty = M/D$ are   irrational provided only that the sequences   $p_k/q_k, m_k/d_k$ are not constant. This condition holds for all the  sequences considered here: indeed  $p_k/q_k$ is always assumed to be nonconstant, while the adjacency condition \eqref{eq:adjac1} that we impose on the initial terms  $\bE_0, \bE_1$ implies that  $m_k/d_k$ is nonconstant.
 \MS
 
 \NI (ii)
 Note that for the staircase at $b=1/3$ there are three separate strands, which cannot be combined into one recursion. Each of these strands has $t \in \{1,2\}$ constant and has alternating values for $\eps$.  Thus the definitions of adjacency and $t$-compatibility do not apply, so that this staircase is of rather different nature from the others. Furthermore,  
 for these strands $d,m$  do not satisfy homogeneous recursions, and hence, \eqref{eq:recurx} does not hold. Its properties are discussed in Section~\ref{ss:13stair}.  See also \cite[Ex.2.3.7]{MM}. 
  \hfill$\er$
\end{rmk}

 \begin{rmk}\rm\label{rmk:adjac} {\bf (The adjacency condition)} (i)
In the language of \cite{MM}, the condition
  \begin{align*}
 (3m'-d') d = (m'-q')p + m'q
 \end{align*}
 in \eqref{eq:adjac1} 
says that $\bE$ satisfies the lower (i.e. the one for ascending stairs) linear relation given by $\bE'$.
 Notice that by \cite[Lem.3.3.5(ii)]{MM} if $p/q < p'/q'$ then $\bE$ satisfies the lower linear relation given by $\bE'$ iff 
 $\bE'$ satisfies the upper linear relation given by $\bE$.
\MS

 \NI (ii)  We saw in \cite[Prop.52]{ICERM} that if all the classes $\bE_k = (d_k,m_k,p_k,q_k, t_k, \eps)$ in an ascending pre-staircase are adjacent to $\bE' = (d',m', p',q', t',\eps)$ and if $p_k/q_k<p'/q' $ for all $k$
then the limit point 
$(b_\infty, z_\infty) = (\lim m_k/d_k, \lim p_k/q_k) = (M/D, P/Q)$ has the property that $z_\infty = \acc(b_\infty)$ is   the lower endpoint $\p^-(I_{\bE'})$ of the $z$-interval blocked by $\bE'$.  
  Similarly, if $\bE''_k$ is a descending pre-staircase with $p_k/q_k>p'/q' $ for all $k$ that
 consists of classes that are adjacent to $\bE'$, then its  accumulation point $P''/Q''$  is the  upper endpoint  $\p^+(I_{\bE'})$ of the  $\bE'$-blocked $z$-interval. 
 In each case, we call $\bE'$ the {\bf associated blocking  class} of the pre-staircase.  
 \MS

\NI (iii) We will see in \S\ref{ss:arith} that the adjacency condition has quite different consequences as well.
In fact, if we look only at perfect classes whose centers $p/q$ are not integral (i.e. if $q>1)$, then every pair of adjacent classes  $\bE, \bE'$ is orthogonal with respect to the intersection  pairing, i.e. $\bE\cdot \bE' = 0$; see Proposition~\ref{prop:adjac}.  This gives us information about the continued fraction expansions of the centers of the steps.
\MS

\NI (iv)  Finally, notice that the adjacency relation for the classes $\bE, \bE'$ was phrased so that the order of the pair is irrelevant, though the condition itself depends on the relation of the centers $p/q, \ p'/q'$.  However, although the  $t''$-compatibility condition does not depend on the order of the inputs,
we are interested in the staircase classes generated by $t''$-recursion with the (ordered!) seeds $\bE, \bE'$. Note that the ordering here is numerical, depending only on the size of entries in $\bx, \bx'$, not on their ratios. Thus  the  staircase classes themselves may ascend or descend. \hfill$\er$
\end{rmk}
 
 \begin{rmk}\label{rmk:misc}\rm
(i)  
Since the basic inequalities involve the $z$-variable, while the function $b\mapsto \acc(b)$ can reverse orientation, we will often  work below with the $z$-interval $I_{\bE'}$ blocked by $\bE'$ rather than the possibly more natural blocked $b$-interval $J_{\bE'}$.  Note also that the center $p'/q'$ of $\bE'$ always lies in $I_{\bE'}$, while $m'/d'$ never lies in $J_{\bE'}$.  In fact,  \cite[Lem.2.2.11]{MM} shows that $\acc(m'/d') > \p^+(I_{\bE'})$ in all cases.   
 
 \MS
 
 \NI (ii)  If  $x_k,y_k, k\ge 0,$ are  recursively defined as in \eqref{eq:recurx} then the difference $x_{k+1}y_k - x_ky_{k+1}$ is constant. Thus one can tell if the ratios $x_k/y_k$ increase or decrease with $k$ by looking at the first two terms.    \hfill$\er$
  \end{rmk}

Here is our main definition.

\begin{definition}\label{def:gentr} The quasi-perfect classes $$
\bE_\la  = (d_\la,m_\la,p_\la,q_\la,t_\la,\eps),  \bE_\mu  = (d_\mu ,m_\mu ,p_\mu ,q_\mu ,t_\mu ,\eps),
\bE_\rho =(d_\rho ,m_\rho ,p_\rho ,q_\rho ,t_\rho ,\eps),
$$
 with $p_\la/q_\la<p_\mu /q_\mu  < p_\rho /q_\rho $ and $t_\la, t_\mu, t_\rho \ge 3$, are said to form a {\bf generating triple} $\Tt$ if 
 \begin{itemlist}\item[{\rm (a)}]  $\bE_\la, \bE_\rho$ are adjacent,
 \item[{\rm (b)}]  $\bE_\la, \bE_\mu $ are  adjacent and $t_\rho$-compatible, i.e. $t_\rho  = q_\la p_\mu  - p_\la q_\mu $;
 \item[{\rm (c)}] $\bE_\rho , \bE_\mu $ are  adjacent and $t_\la $-compatible, i.e.  $t_\la  = q_\mu p_\rho -p_\mu q_\rho $;
 \item[{\rm (d)}]
 $t_\la t_\rho  - t_\mu  =  q_\la p_\rho -p_\la q_\rho$,
  \item[{\rm (e)}]  $\acc(m_\rho /d_\rho ),\ \acc(m_\la /d_\la ) > \acc(m_\mu /d_\mu )$.
  \end{itemlist}
 \end{definition}
 
Here the letters $\la,\mu, \rho$ stand for \lq left', \lq middle', and \lq right'.

\begin{example}\label{ex:SsU} \rm The basic examples of generating triples are the triples $$
\Tt^n_*: = 
(\bB^U_n,  \bE^1_n, \bB^U_{n+1}), \quad n\ge 0,
$$
 where 
$\bB^U_n = (n+3,n+2,2n+6,1,2n+3), \quad \eps: = 1,
$
 are the blocking classes for the $\Ss^U$ family and, for each $n$, $\bE^1_n$ denotes the shared first step of the staircases\footnote
 {
 The convention is that staircases with label $\ell = $ `lower' are ascending and converge to the lower endpoint of the corresponing blocked interval, while those 
 label $u = $ `upper' are descending.  In the current paper the steps are indexed by $k$ (while \cite{MM} used both $k$ and $\ka$ for reasons explained in \cite[\S3.1]{MM}.)  The other decorations in the label 
 of $\Ss^U_{\ell,n}$ for example describe the family (namely $U$) and the blocking class, in this case $\bB^U_n$.}
$\Ss^U_{\ell, n+1}$ and $\Ss^U_{u,n}$.  As we see from \eqref{eq:SsUsteps}, $\bE^1_n$ has center $p^1_n/q^1_n = [2n+7; 2n+4]$ with parameters 
\begin{align}\label{eq:SsUparam}
&(d_n^1,m_n^1, p^1_n,q^1_n, t^1_n): \\  \notag
&\qquad = \bigl(2n^2 + 11n + 14, \ 2n^2 + 9n + 9,\ 4n^2 + 22n + 29, 2n+4, \ 4n^2 + 16n + 13\bigr);
\end{align}
we sometimes denote it by  $\bE_{[2n+7; 2n+4]}$.
Further both  the ascending staircase $\Ss^U_{\ell, n+1}$ and descending staircase $\Ss^U_{u, n}$ have decreasing ratios $m_k/d_k$, and, by Remark~\ref{rmk:adjac}~(ii) these ratios converge to the appropriate endpoint of the blocked $b$-intervals. 
  In particular, the ratios for the  ascending staircase $\Ss^U_{\ell, n+1}$ satisfy   $m_n^0/d_n^0> m_n^1/d_n^1 > \p^-(J_{\bB^U_{n+1}})$.  Therefore because $b\mapsto \acc(b)$ preserves orientation  for $b>1/3$ the required inequalities
  $$
  \acc(m_{n+1}^0/d_{n+1}^0), \acc(m_n^0/d_n^0) > \acc(m_n^1/d_n^1)
$$
hold.  For later reference, the centers of the steps of staircases $\Ss^U_{\ell, n}$ and $\Ss^U_{u, n}$ and their accumulation points  are
\begin{align}\label{eq:SsUsteps}
p^U_{\ell,n,i}/q^U_{\ell,n,i} & =  [\{2n+5,2n+1\}^i,\eend_n], \;  \ i\ge 0,\\ \notag
& \qquad \mbox{ where } \  \eend_n = 2n+4\;\;\mbox{ or } (2n+5,2n+2),\\  \notag
p^U_{u,n,i}/q^U_{u,n,i} & = [2n+7;\{2n+5,2n+1\}^i,\eend_n], \;  \ i\ge 0,\\ \notag  \mbox{ with limits } 
z^U_{\ell,n,\infty} &= [\{2n+5,2n+1\}^\infty], \qquad z^U_{u,n,\infty} = [2n+7; \{2n+5,2n+1\}^\infty].
\end{align}
Note that here we use $i$ as the indexing label for the staircase steps because, for ease of writing,  we have written each of the two strands\footnote
{
As remarked earlier, we use the word \lq staircase' rather loosely to refer to any sequence of classes, infinitely many of which are live for some limiting $b$-value.  A subset of these classes whose continued fraction can be described by a single formula is called a  {\bf strand}; see \cite[\S3.1]{MM} for further discussion. 
In the case of the staircases $\Ss^U_{\ell,n}, \Ss^U_{u,n}$, these strands are distinguished by the ending of the continued fractions of their centers $p/q$.}
 of the staircase separately.  These strands are intertwined; thus the steps in  
$\Ss^U_{\ell, 1}$ when written in ascending order have centers at
$$
[6],\ [7;4], \ [7; 3,6],\ [7;3,7,4],\ [7; 3,7,3,6],\dots.
$$ 
(For an explanation of how to order continued fractions see Appendix~\ref{app:arith}.)   \hfill$\er$
\end{example}

\begin{rmk}\label{rmk:triple0}\rm (i)  An important fact about quasi-perfect classes $(d,m,p,q)$ with $m/d>1/3$ is that
$p/q < \acc(m/d)$; see \cite[Lem.2.2.11]{MM}.   This is relevant in a variety of contexts, for example in Lemma~\ref{lem:md1} that analyzes the relations between two different perfect classes.  Notice also that the definition of a generating triple requires control over the relative positions  of the points $\acc(m_\la, d_\la),  \acc(m_\mu/d_\mu)$ and $\acc(m_\rho/d_\rho)$.  One consequence of this hypothesis is that the ratios $m/d$ decrease  for all our staircases with $b>1/3$, which, as pointed out in Proposition~\ref{prop:live},  is important in the proof that these pre-staircases are live.

 We also showed
 in \cite[Lem.2.3.5]{MM}  that for the $\Ss^U$ 
family we have
\begin{align}\label{eq:accn}
\acc(m_n^0/d_n^0) < p^0_{n+1}/q^0_{n+1} = 2n+8,\qquad \forall \ n\ge0.
\end{align}
Because $m_n^1/d_1^1<m_n^0/d_n^0$ we also have $\acc(m_n^1/d_1^1)< p^0_{n+1}/q^0_{n+1}$.
In fact, it follows from Lemma~\ref{lem:accmd} below (applied to the quasi-triples\footnote
{See Remark~\ref{rmk:quasitrip}.}
 $(\bE^U_{\ell,seed}, \bB^U_n, \bB^U_{n+1})$ and their images under the symmetries) that all the generating triples that we encounter have the property that
 $\acc(m_\mu/d_\mu)<p_\rho/q_\rho$.  
 We did not put  this property in the definition for the sake of simplicity.

\MS

\NI (ii)  The  triples $\Tt^n_*$ all have  $t_\la<t_\rho$ and $\eps = 1$.    However this inequality is reversed when we consider triples with $b<1/3$.  For example, the reflection $R$ defined in \S\ref{ss:symm} below takes the $\Ss^U$ staircase family (with $b>1/3, \ \eps = 1$)  to the $\Ss^L$ family (with $b<1/3, \ \eps= -1$).  It fixes $t$ but reverses the $z$-orientation.  We will prove in Proposition~\ref{prop:symm} that the action of $R$ preserves generating triples.  For example, the three classes
\begin{align*}
\bB^L_2 = R^\sharp(\bB^U_2)& =  \bigl(10,1,25,4,7,-1\bigr), \quad R^\sharp(\bE^1_2) = \bigl(48,5,  120,19,33, -1
\bigr),\\
&  \bB^L_1 = R^\sharp(\bB^U_1) =  \bigl(5,0,13,2, 5,-1\bigr)
\end{align*}
also form a generating triple.  Note that the image by $R$ of the basic triple $(\bB^U_0, \bE^1_1, \bB^U_1)$
whose classes have centers $(6, [7;4], 8)$ requires special treatment because the point $R(6)$ is undefined; see Remark~\ref{rmk:symm0}.  
\hfill$\er$
\end{rmk}

  We show below that any generating triple $\Tt$ has two associated {\bf principal pre-staircases}, a descending pre-staircase $\Ss^\Tt_u$ with seeds (i.e. first two steps)  $\bE_\rho, \bE_\mu$, and an ascending pre-staircase $\Ss^\Tt_\ell$ 
with seeds $\bE_\la , \bE_\mu$. In order to prove
 this we will use our  main result, which states that if $\Tt: = (\bE_\la ,\bE_\mu,\bE_\rho)$ is a generating triple, there are two associated generating triples $x\Tt: = (\bE_\la , \bE_{x\mu}, \bE_\mu)$  and $y\Tt: = (\bE_\mu,  \bE_{y\mu}, \bE_\rho)$.  Here the new class $ \bE_{x\mu}$ in the first triple $x\Tt$
is what will be the third step of the descending pre-staircase $\Ss^\Tt_u$  while the middle entry in $y\Tt$ is the third step $ \bE_{y\mu}$  of 
$\Ss^\Tt_\ell$.
We sometimes call these the {\bf left (resp. right) derived triples}, with $x$ denoting a move to the left and $y$ a move to the right.  We will also call $x\Tt$ (resp. $ y\Tt$) the left (resp. right) {\bf mutation} of $\Tt$.   See Figure~\ref{fig:decimal} to visualize the left versus right move.

\begin{prop}\label{prop:gener}  Suppose that  $\bE_\la  : = (d_\la ,m_\la ,p_\la ,q_\la ,t_\la , 1),\ 
 \bE_\mu  = (d_\mu ,m_\mu ,p_\mu ,q_\mu ,t_\mu ,1)$ and $\bE_\rho  : =(d_\rho,m_\rho,p_\rho  ,q_\rho  ,t_\rho  ,1)$ form a generating triple $\Tt$.
   Then  
    \begin{itemlist}
 \item[{\rm (i)}]
If $\bE_{x\mu} = \bigl(d_{x\mu},m_{x\mu},p_{x\mu},q_{x\mu},t_{x\mu},1\bigr)$ is the class obtained by one $t_\la $-iteration from $\bE_\rho  , \bE_\mu $ then
the classes  $(\bE_\la ,  \bE_{x\mu}, \bE_\mu )$ form a generating triple $x\Tt$.
 \item[{\rm (ii)}]   Similarly, if 
 $\bE_{y\mu} =  \bigl(d_{y\mu},m_{y\mu},p_{y\mu},q_{y\mu},t_{y\mu},1\bigr)$ is the class obtained by one $t_\rho$-iteration from $\bE_\la , \bE_\mu $,
 then $(\bE_\mu , \bE_{y\mu}, \bE_\rho)$ form a generating triple $y\Tt$.
  \end{itemlist}
 \end{prop}
\begin{proof}  To prove (i), we must check that conditions (a) through (e) hold for $x\Tt$.  Condition (a)  states that $E_\la, E_\mu$ are adjacent, which holds by hypothesis.

 The first step in the proof of (b) is to check that  the classes $\bE_\la , \bE_{x\mu},  \bE_\mu $ satisfy the order condition $p_\la/q_\la < p_{x\mu}/q_{x\mu} < p_\mu/q_\mu$.  Now  $p_{x\mu}/q_{x\mu} < p_\mu/q_\mu$
 because $p_\mu/q_\mu< p_\rho/q_\rho$; see Remark~\ref{rmk:misc}~(ii).
Further $p_\la/q_\la < p_{x\mu}/q_{x\mu}$  if $p_\la(t_\la q_\mu - q_\rho)  < (t_\la p_\mu - p_\rho)q_\la$, that is, if 
$$
q_\la p_\rho  - p_\la q_\rho < t_\la\bigl(q_\la p_\mu  - p_\la q_\mu\bigr) = t_\la t_\rho,
$$
where the equality holds by  condition (b) in Definition~\ref{def:gentr}.
But $$
q_\la p_\rho  - p_\la q_\rho  = t_\la t_\rho - t_\mu
$$
 by (d), and $t_\mu \ge 3$ by hypothesis.
Therefore this order condition holds.    

It follows that the new class $\bE_{x\mu}$ is adjacent to $\bE_\la$, because it is a linear combination of 
the two classes $\bE_\rho, \bE_\mu$ which are both adjacent to $\bE_\la$, and the condition in \eqref{eq:adjac} for classes with centers $> p_\la/q_\la$ to be adjacent to $\bE_\la$ is linear in the variables $p'/q'$.
Further $\bE_\la, \bE_{x\mu}$ are $t_\mu$-compatible because 
\begin{align*}
q_\la p_{x\mu} - p_\la q_{x\mu} & = q_\la(t_\la p_\mu-p_\rho) - p_\la (t_\la q_\mu-q_\rho) \\
& = t_\la t_\rho - (t_\la t_\rho - t_\mu) = t_\mu
\end{align*}
by conditions (b) and (d) for $\Tt$.
Therefore (b) holds for $x\Tt$.
\MS

Condition (c) requires that  $\bE_{x\mu},  \bE_\mu $ be adjacent and $t_\la$-compatible, which holds by
 Lemma~\ref{lem:recur0}~(i).  
 
 Condition (d) requires that $t_\la t_\mu - t_{x\mu} = q_\la p_\mu - p_\la q_\mu $.  But
 \begin{align*}
 t_\la t_\mu - t_{x\mu} = t_\la t_\mu - (t_\la t_\mu - t_\rho) = t_\rho =  q_\la p_\mu - p_\la q_\mu,
 \end{align*}
 where the last equality holds by condition (c) for the initial triple $\Tt$.
 
Finally we must check the inequalities for $\acc(m_{x\mu}/d_{x\mu})$.   Lemma~\ref{lem:accmd} below shows that $\acc(m_{x\mu}/d_{x\mu}) < p_\mu /q_\mu $. (Note that $p_{x\mu} > 2p_\mu$ because $p_{x\mu} = t p_\mu - p_\rho$ for some $t\ge 3$.)
Because $\bE_\mu$ is a quasi-perfect class, \cite[Lem.2.2.11]{MM} implies that $p_\mu /q_\mu  < \acc(m_\mu /d_\mu )$. Further  $\acc(m_\mu /d_\mu)< \acc(m_\la /d_\la )$ by assumption (e) for $\Tt$.  
Therefore we  have $\acc(m_{x\mu}/d_{x\mu}) < \acc(m_\la /d_\la ), \acc(m_\mu /d_\mu )$.  This completes the proof of (i).
\MS

The proof of (ii) is very similar.  Again condition (a) is automatic, while to prove (b) we must first check that the centers of the classes are correctly ordered, i.e. we need
$p_\mu/q_\mu < p_{y\mu}/q_{y\mu} < p_\rho/q_\rho$.  The first inequality holds by construction, while the second
holds if
$ (t_\rho p_\mu-p_\la) q_\rho < (t_\rho q_\mu-q_\la) p_\rho,$ that is if
$$
q_\la p_\rho - p_\la q_\rho < t_\rho (q_\mu p_\rho - p_\mu q_\rho)  = t_\rho t_\la.
$$
But this holds because $q_\la p_\rho - p_\la q_\rho  = t_\la t_\rho - t_\mu$ by hypothesis.

The rest of the proof follows as before; again we use Lemma~\ref{lem:accmd} to check (e).
\end{proof}

 The next lemma applies to any pair of classes $\bE, \bE'$ that are adjacent and $t''$-compatible for some $t''\ge 2$ with decreasing centers $p/q>p'/q'$ and with $p<p'$. Hence, in particular, it
 applies to any pair of adjacent steps in a descending pre-staircase. 
  
 \begin{lemma}\label{lem:accmd}   
 Let $\bE = (d,m,p,q,t,1), \bE'= (d',m',p',q',t',1)$ be quasi-perfect classes with $3+2\sqrt{2} <p'/q' < p/q$ that are adjacent and $t''$-compatible for some $t''\ge 2$.  Assume further that $\sqrt{2}p'>p$.  Then  
 $$
w': = \acc\Bigl(\frac{m'}{d'}\Bigr) < \frac {p}{q}.
$$
\end{lemma}
\begin{proof}
To prove this it suffices to show that $$
w' + \frac 1{w'}< \frac{p}{q} +\frac{q}{p}.
$$
Since  $(d,m,p,q)$ is quasi-perfect,  $w': = \acc(m'/d')$ satisfies the  equation 
\begin{align*}
w' + \frac 1{w'} &= \frac{(3-m'/d')^2}{1-(m'/d')^2} - 2\\
& = \frac{(3d'-m')^2 - 2((d')^2-(m')^2)}{(d')^2-(m')^2} = \frac{(p')^2+(q')^2+2}{p'q'-1}.
\end{align*}
Thus we must show that 
\begin{equation}\label{eq:p'q'<pq}
  \frac{(p')^2+(q')^2+2}{p'q'-1} < \frac{p}{q} + \frac{q}{p},
\end{equation}
or equivalently
\[
\frac{(p+q)^2}{pp'qq'} <  \frac{p^2 + q^2}{pq}- \frac{(p')^2 + (q')^2}{p'q'}.
\]
Consider the function $h(z) = z + 1/z$.  This is an increasing function for $z>1$ whose derivative increases to $1$ as $z\to \infty$.  Therefore, for $z>z'$ we have  $$
h(z) - h(z') > h'(z) (z-z'),
$$
 so that it suffices to show that
$h'(p'/q')(p/q - p'/q') >  \frac{(p+q)^2}{pp'qq'}$, or equivalently
\begin{align*}
\Bigl(1 - \bigl(\frac {q'}{p'}\bigr)^2\Bigr)(pq'-p'q)  >  \frac{(p+q)^2}{pp'} = \Bigl(1 + \frac {q}{p}\Bigr)^2 \ \frac{p}{p'}.
\end{align*}
But $pq'-p'q = t''\ge 2$, so that  because $p/q, p'/q' > 3+2\sqrt{2}$ it suffices to check that 
\[
2\left(1-\frac{1}{(3+2\sqrt{2})^2}\right)>\left(1+\frac{1}{3+2\sqrt{2}}\right)^2\frac{p}{p'}\;\Leftrightarrow\;\sqrt{2}>\frac{p}{p'},
\] which holds by hypothesis.
\end{proof}

\begin{cor}\label{cor:accmd} In any  generating triple $(\bE_\la, \bE_\mu , \bE_\rho  )$ 
that is derived by mutation from one of the basic  triples $(\bB^U_n, \bE_n^1, \bB^U_{n+1})$ in Example~\ref{ex:SsU}, we
 have $\acc(\frac{m_\mu }{d_\mu }) < \frac{p_\rho  }{q_\rho  }$.
\end{cor}
\begin{proof}  We observed in  Remark~\ref{rmk:triple0}~(i) that this holds for the triples $(\bB^U_n, \bE_n^1, \bB^U_{n+1})$.  We can also deduce this from Lemma~\ref{lem:accmd} because $\bB^U_{n+1}, \bE_n^1$ are adjacent steps in the descending staircase $\Ss^U_{\ell,n}$.  It holds for all derived triples  by induction.   In the case of 
$(\bE_\la, \bE_{x\mu}, \bE_\mu )$ this is immediate because $\bE_\mu ,\bE_{x\mu}$ are adjacent steps in a descending pre-staircase. 
 In the case of $(\bE_\mu , \bE_{y\mu}, \bE_\rho  )$ it holds because $\acc(m_\mu /d_\mu )< p_\rho  /q_\rho  $ by the inductive hypothesis, and $\bE_\mu , \bE_{y\mu}$ are steps in an ascending pre-staircase 
with decreasing ratios $m_k/d_k$ so that $\acc(m_{y\mu}/d_{y\mu})<\acc(m_\mu /d_\mu )$.
\end{proof}

 \begin{lemma}\label{lem:2}  If $(\bE_\la , \bE_\mu ,\bE_\rho )$ form a generating triple $\Tt$ then:
  \begin{itemlist}
  \item[{\rm (i)}] The pre-staircase $\Ss^\Tt_\ell$  with recursion parameter $t_\rho$ and seeds $\bE_\la , \bE_\mu $ ascends and has blocking class $\bE_\rho  $.
   \item[{\rm (ii)}]  The pre-staircase $\Ss^\Tt_u$ with recursion parameter $t_\la$ and seeds $\bE_\rho, \bE_\mu $ descends and has blocking class $\bE_\la $.
      \item[{\rm (iii)}] In both cases the sequence $\acc(m_k/d_k)$ decreases and the limits $z_\infty = \lim p_k/q_k$ and $b_\infty = \lim m_k/d_k$ are irrational.
  \end{itemlist}
  \end{lemma}
  \begin{proof}  The pre-staircase $\Ss^\Tt_\ell$ has first two steps $\bE_\la , \bE_\mu$ with subsequent steps $\bE_{y^k \mu}, k\ge 1.$     It follows from Proposition~\ref{prop:gener}~(ii) that all these steps lie below $\bE_\rho$,
  and are adjacent to $\bE_\rho$.  Hence, as explained in Remark~\ref{rmk:adjac}~(ii), $\bE_\rho$ is the blocking class for this pre-staircase. 
   In particular, this implies that this pre-staircase 
   accumulates at $\bigl(b_\infty, z_\infty = \acc(b_\infty)\bigr)$ where $b_\infty = \p^- J_{\bE_\rho}$.
  Similarly, the subsequent steps $\bE_{x^k \mu}, k\ge 1,$ of the descending pre-staircase  
  $\Ss^\Tt_u$  lie above $\bE_\la$, and  are all adjacent to it.    Thus  this pre-staircase 
  accumulates at the upper endpoint $\p^+ J_{\bE_\la}$ of   the interval blocked by $\bE_\la$.  Finally to check  (iii) notice that the first two terms of these sequences decrease because
  we assume $\acc(m_\rho /d_\rho ), \acc( m_\la /d_\la ) > \acc(m_\mu /d_\mu )$.  Therefore the sequences decrease by Remark~\ref{rmk:misc}~(ii). Finally 
   the irrationality claims follow from  Remark~\ref{rmk:recur0}~(i). 
       \end{proof}

\begin{proof}[Proof of Theorem \ref{thm:main1} part (i)] This follows immediately from Example \ref{ex:SsU}, Proposition \ref{prop:gener}, and Lemma \ref{lem:2}.
\end{proof}

\begin{rmk}\rm (i)  In view of Corollary~\ref{cor:accmd} we could have sharpened  the last condition  
$\acc(m_\rho  /d_\rho  ), \acc(m_\la/d_\la) > \acc(m_\mu/d_\mu)$ in the definition of a generating triple by adding the requirement
$\acc(\frac{m_\mu}{d_\mu}) < \frac{p_\rho  }{q_\rho  }$.  We did not do this to keep the definition as simple as possible. 
Note however that, as we shall show later (for example in Lemma~\ref{lem:md}), the value of $m/d$ is relevant to the behavior of the corresponding obstruction $\mu_{\bE,b}$ as $b$ varies.  

\MS

\NI (ii)
It follows from Lemma~\ref{lem:2} that all the pre-staircases in the complete family $\Cc\Ss^U$ have decreasing ratios $m_k/d_k$.  This is a crucial ingredient in the proof that these pre-staircases are all live; see Proposition~\ref{prop:live}.
\MS

\NI (iii) The fact that all the pre-staircases in $\Cc\Ss^U$, 
whether ascending or descending, have ratios $m_k/d_k$ that decrease with $k$ is rather paradoxical, since one would naively expect that these ratios would increase for ascending pre-staircases. 
 Indeed, given two classes $\bE = (d,m, p,q,t)$ and $\bE' = (d',m', p',q',t')$ we have
\begin{align*}
\frac md<\frac{m'}{d'} & \ \Longleftrightarrow\ \frac{(p+q) + 3t}{3(p+q) + t} < \frac{(p'+q') + 3t'}{3(p'+q') + t'}\\
& \ \Longleftrightarrow\ t(p'+q') < t'(p+q)\  \ \Longleftrightarrow\ \frac t{p+q} < \frac {t'}{p'+q'}.
\end{align*} 
However, if we ignore the integer $8$ that appears in the definition of $t$ and denote $z: = p/q$, we have
$$
\frac {t^2}{(p+q)^2} \approx \frac{p^2 - 6pq + q^2}{(p+q)^2} = \frac {z^2 - 6z + 1}{(z+1)^2},
$$
which is an increasing function of $z$ for $z>1$.  Given $p/q<p'/q'$, this rough estimate suggests that we should also have $ \frac t{p+q} < \frac {t'}{p'+q'}$ and hence $m/d<m'/d'$; but this  is not the case.
  Thus, even though the number $8$ is very small compared to the eventual size of $p,q$ its influence in the relevant formulas cannot be ignored.
\hfill$\er$
\end{rmk}

\begin{rmk}\label{rmk:quasitrip} \rm {\bf(Quasi-triples)}\ 
As noted in \eqref{eq:seed} above, in \cite{MM} the staircase $\Ss^U_{\ell,n}$ was defined to have blocking class $\bB^U_n$ and seeds $\bE^U_{\ell,seed}=(1,1,1,1,2)$ and $\bB^U_{n-1}.$ Using the recursion parameter for $\bB^U_n,$ the next step in the staircase following $\bE^U_{\ell,seed},\bB^U_{n-1}$ is $\bE_{[2n+7,2n+4]}$ as expected from Example~\ref{ex:SsU}. For all $n$, this first seed $\bE^U_{\ell,seed}$ is independent of $n.$ Similarly, for $\Ss^U_{u,n},$ the seed 
$\bE^U_{u,seed}=(-2,0,-5,-1,2)$ plays a similar role. In \cite{MM}, these seeds were especially useful because they behave well under symmetries and also, because their entries are independent of $n,$ they made many computations simpler. 

Because the  two seeds $\bE^U_{\ell,seed}$ and $\bE^U_{u,seed}$ together with the blocking class $\bB^U_n$  generate recursive staircases as described above,
the triads
\begin{align*}
\Tt_{\ell,seed}^n:=& (\bE^U_{\ell,seed},\bB^U_{n},\bB^U_{n+1}),\qquad 
\Tt_{u,seed}^n:=  (\bB^U_{n},\bB^U_{n+1},\bE^U_{u,seed}), 
\end{align*}
satisfy the adjacency and $t$-compatibility conditions (a),(b),(c) for a generating triple. It is also easy to check that (d) holds. 
 Further if we define their $x$- and $y$-mutations  as in Proposition~\ref{prop:gener}, we have 
 \begin{align*}
y\Tt_{u,seed}^n=\Tt_{u,seed}^{n+1} \quad \text{and} \quad x\Tt^n_{\ell,seed}=\Tt_{\ell,seed}^{n-1}
\quad\mbox{ while}\\
x\Tt_{u,seed}^n=\Tt^n_* \quad \text{and} \quad y\Tt^n_{\ell,seed}=\Tt^{n}_*.
\end{align*}
We will say that triples  such as 
  $\Tt_{\ell,seed}^n, \Tt_{u,seed}^n$ that satisfy conditions (a), (b), (c), and (d) of Definition~\ref{def:gentr} are  {\bf quasi-triples}.   They do not satisfy condition  (e).  Also neither of the seeds are blocking classes; indeed the upper seed has negative entries, so it is not even a geometric class.

As explained in \cite[\S3.3]{MM}, both  $\bE^U_{\ell,seed}$ and $-\bE^U_{u,seed} = S^\sharp\bE^U_{\ell,seed}$ are  steps in the third strand of the staircase at $b=1/3$; indeed this strand has the single seed
$\bE^U_{\ell,seed}$ and has steps given the images of this seed  under the shifts $S^k, k\ge 1.$  Further, as we show in \S\ref{ss:13stair}, the classes in this strand obstruct the existence of  
 ascending staircases at the  special rational $b$.   The other two strands of this staircase  
are given  by the images by the shift $S^k, k\ge 1,$ of the classes $\bB^U_{-2} = (1,0,2,1,1,-1)$ and $
 \bB^U_{-1} = (2,1,4,1,1,1)$.  
 Moreover, by replacing the $t$ entries in $\bB^U_{-2}$ and $\bB^U_{-3}$ by $\varepsilon t$, we can think of $\bB^U_{-3}= (0,-1,0,1,3,-1)$ and $\bB^U_{-2}$ as the first two terms of the sequence $\bB^U_n, n\ge -3$ with recursion parameter $t=2$.    The results in \cite{M1} show that this iteration corresponds to the notion of $v$-mutation in the toric model.
 Thus the whole configuration of perfect classes found in this paper is generated by the staircase at $b=1/3$ together with mutations and symmetries.  \hfill$\er$
\end{rmk}

\begin{rmk}\label{rmk:fakestair}\rm {\bf (Fake pre-staircases and pseudo-triples):}  
To give perspective on the problem, we now explain some properties of a family of Diophantine classes that share many of the properties of the classes described above but which are \lq fake', i.e. they do NOT reduce properly under Cremona moves,  and are not live.\footnote{
 In fact, the very first pre-staircase that we found when working on \cite{ICERM} actually was fake, and is the image under $S$ of one of the pre-staircases described here. See \cite[Ex.28(iii)]{ICERM}.
Also see  \cite[\S4.1]{MM} for an explanation of Cremona reduction.}
These classes   have the formulas described in Example~\ref{ex:SsU} but with parameters in $ \frac 12\Z \less \Z$, so that only half of the them have integral entries and hence correspond to obstructive classes.  
 For example, the tuple $\bB_{1/2}^U = \bigl(\frac 72, \frac 52, 7,1,4\bigl)$ has center $7$,  and of course does not represent an obstructive class.  For clarity we will call such a tuple a {\bf pseudo-class}.  Similarly none of the tuples $\bB_{n-1/2}^U,n\in \N,$ 
 have integral values for $(d,m)$.  
 On the other hand, the tuples $ \bE_{n-1/2}^1$ are integral for all integers $n$, though  
one can check that none of them are  exceptional divisors; for example they do not 
Cremona reduce correctly and also have negative intersection with 
  exceptional classes such as $\bB^U_{0}$.  

Note that the triples $\Tt'_n: =  \Tt_{n-1/2}: = \bigl(\bB^U_{n-1/2}, \bE^1_{n-1/2}, \bB^U_{n+1/2}\bigr)$  satisfy all the numeric conditions  to be a 
 generating triple (except the requirement to be integral!), and half of the tuples in the ascending  \lq staircase' generated by $\bB^U_{n-1/2}, \bE^1_{n-1/2}$ by $t^0_{n+1/2}$-recursion are integral.  
   For example, when $n=0$, the class $$
 \bE_{-1/2}^1 = \bigl(9,5,19,3,6\bigr),\quad p/q = [6;3] 
 $$
is a step in an ascending fake pre-staircase with centers $[6;\{2,6\}^k, 3]$.
 These are  the odd-placed terms in the pre-staircase with seeds 
 $ \bB_{-1/2}^U,  \bE_{-1/2}^1$ and recursion parameter $\rho = t^0_{1/2} = 4$ and hence form one of the strands of this ascending  pseudo pre-staircase.  As explained in \cite[Lem.3.3.1]{MM},  we can consider this sequence to have recursion parameter $\nu = \rho^2 - 2 = 14$, and to be generated\footnote
 {
 In fact we can take the initial seed to be the class  $\bE^U_{\ell,seed} = (1,1,1,1,2)$ in Remark~\ref{rmk:quasitrip}.}
  by $ \bE_{-1/2}^1$ and the third term in this sequence which is
 $$ 
 \bE' = \bigl(125, 69,265,41,83\bigr),\quad p/q = [6;2,6,3]. 
 $$
Thus this fake pre-staircase 
has only one strand.  It satisfies the linear relation given by the pseudo-class $\bB^U_{-1/2}$.  
By Lemma~\ref{lem:pseudst} below, for sufficiently large $k$ the classes $\bE_k = (d_k,m_k,p_k,q_k,t_k)$ in this fake pre-staircase 
 are obstructive when $b=\lim m_k/d_k$.  However, they are not live because the obstruction from $\bB^U_0$ is larger.  Further, one can check that $\bB^U_0\cdot \bB^U_{-1/2} = -1$; in fact   
$\bB^U_0\cdot \bE<0$ for all elements $\bE$ in this sequence.

Similarly, only one strand of the corresponding descending sequence of tuples is integral, and it also consists of fake classes.  For example, when $n=1$  this strand of the  descending sequence generated by  $\bE^0_{3/2}$ and $\bE^1_{1/2}$ with recursion parameter $t^0_{1/2}$ and  pseudo-blocking class $\bE^0_{1/2}$ can be taken to have
seeds 
$\bE^U_{u,seed} = (-2,0,-5,-1,2)$ and $\bE^1_{1/2} = (20,14,41,5,22)$ with recursion parameter $(t^0_{1/2})^2-2 = 14$ and hence next step
 $$ 
 \bE = \bigl(282,196,579,71,306  \bigr),\quad p/q = [8;6,2,5]. 
 $$
Thus, the general  step has centers at $[8;\{6,2\}^k,5], k\ge 0$, and lies in the interval  blocked by $\bB^U_1$, though it satisfies the linear relation given by the pseudo-blocking class    $\bE^0_{1/2}$.
Again $\bB^U_1\cdot\bE^1_{1/2}<0$.
\MS

The reader can check that there is an analogous descending pseudo-staircase with centers in the interval blocked by 
 $\bB^U_0$.  This one satisfies the linear relation determined by $\bE^0_{-3/2}$ with $t=2$ and is generated numerically 
 by $\bE^U_{u,seed}$ and $\bE^1_{-1/2}$ with recursion parameter $t^2 -2= 2 = \nu$.
 This sequence consists of the tuples
 $$
\bE''_{\ell,k}: =  \bigl(-2 + 11k, 5k, -5+24k, -1+4k, 2+4k\bigr), \ k\ge 1,  p_k/q_k = [6; 4k-1],
 $$
and its continued fraction expansion is \lq degenerate' in the sense that it does not contain repeated digits.\footnote{
This happens because this sequence converges too slowly: indeed, $\nu=2$ is too small for the argument outlined in Remark~\ref{rmk:recur0} to apply.}
Further, one can check that at the limiting $b$-value $b_\infty = \acc_U^{-1}(6) = 5/11$ all these classes 
$\bE''_{\ell,k}$ fail to be obstructive at their centers since the requirement $|d_kb_\infty - m_k|< \sqrt{1-b_\infty^2}$
evaluates to  $\frac{10}{11}  < \sqrt{96}/{11}$, which is false.
\hfill$\er$
\end{rmk}

\begin{lemma}\label{lem:pseudst} Let $\bE_k = (d_k,m_k,p_k,q_k,t_k), k\ge 0,$ be a sequence of quasi-perfect classes with $\bE_0< \bE_1$ and
whose entries satisfy a recursion $x_{k+1} = \nu x_k- x_{k-1}$ with $\nu \ge 3$.  Then, if $b_\infty: = \lim m_k/d_k$, there is $k_0$ such that
$$
 \mu_{\bE_k, b_\infty}(p_k/q_k) > c_{H_{b_\infty}}(p_k/q_k), \quad k\ge k_0.
 $$
 Hence $\acc(b_\infty) = \lim_k p_k/q_k.
 $
\end{lemma}

\begin{proof}  By \cite[Lem.15(iii)]{ICERM},  a quasi-perfect class $\bE=(d,m,p,q)$ is obstructive at $p/q$ for some $b$ (that is, has
$\mu_{\bE,b}(p/q)> c_{H_b}(p/q)$) if and only if $|db-m|< \sqrt{1-b^2}$.

As we noted in Remark~\ref{rmk:recur0}, when $\nu\ge3$ there is $\la>1$ such that
 the ratios $m_k/d_k, p_k/q_k$ converge to $M/D, P/Q$ at the rate  $O(\la^{-2k})$.  Further,
$$
|d_k \tfrac MD - m_k| = \tfrac 1D |M\ov D - D\ov M| \la^{-2k} < \sqrt{1-\tfrac {M^2}{D^2}},\quad \mbox{ for suff large } k.
$$
Thus, for sufficiently large $k$ the function $ \mu_{\bE_k, b_\infty}$ is obstructive at $p_k/q_k$, which proves the first claim.

The second claim then follows from \cite[Thm.1.13]{AADT}. Notice that the statement of this theorem assumes the existence of a staircase at $b_\infty$, i.e. a sequence of live classes --- however the proof given there only uses the fact the classes are obstructive at $b_\infty$. 
\end{proof} 

\subsection{The pattern of derived classes}\label{ss:ternary}

We now describe the set $\Ee_{[6,8]}$ formed by all the blocking classes that are derived from the foundational triple 
$$
\Tt^0_*: 
= \bigl(\bB^U_0, \bE_{[7;4]}, \bB^U_1\bigr)
$$
defined in Example~\ref{ex:SsU}. The main result is Proposition~\ref{prop:dense} stating that the union of the corresponding blocked $z$-intervals intersects the interval $[6,8]$ in a dense, open subset.   
Finally, we assign an infinite ternary decimal label to each point $z\in [6,8]$ that is not blocked, and state in Proposition~\ref{prop:liveZ} a more detailed version
of the  claims in Theorem~\ref{thm:main}~(ii) about  the associated pre-staircases. 

 For simplicity, we will only discuss  in detail the case when $z\in [6,8]$.  However, the pattern formed by the classes with centers in the intervals $[2n+6, 2n+8], n>0$ is precisely the same; see Remark~\ref{rmk:n>0}.  In particular, the density result holds for all $n$.

We first introduce a ternary decimal label $\de$ for the  classes $E_\de \in \Ee_{[6,8]}$. 

 \begin{figure}[htbp] 
 \vspace{-1.1in}   \centering
\hspace{1.3in}    \includegraphics[width=7in]{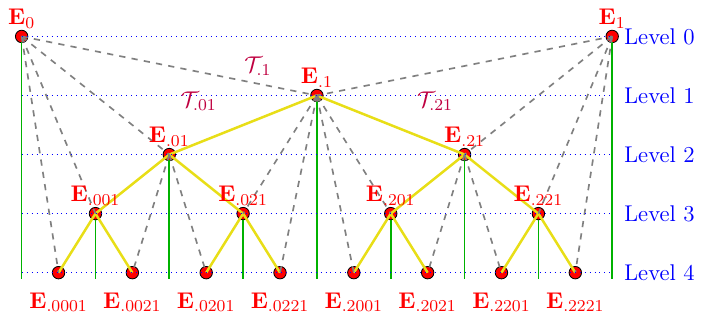} 
 \vspace{-6.3in}
  \caption{The edges of the main recursion are in  yellow (given by $\de(w) \to \de(xw),\ \de(yw)$), with the other edges of the triples in dotted grey; none of these edges  cross the vertical green lines that represent 
  the blocked intervals $I_{\bE}$.}
 \label{fig:decimal}
 \end{figure}

\begin{lemma} \label{lem:label} Each class in $\Ee_{[6,8]}$ can be given a ternary decimal label $\de$ 
such that the following holds:
\begin{itemlist}
\item[{\rm(i)}] We define $\bE_0: = \bB^U_0,  \bE_{.1}: =  \bE_{[7;4]},  \bE_1: = \bB^U_1$, and call the corresponding triple either $\Tt_{.1}$ or $\Tt_*^0$. 
All other classes have labels of the form $\de = .a_1\cdots a_{k-1} 1$ where $a_i\in \{0,2\}$.
\item[{\rm(ii)}] If $\de = .a_1\cdots a_{k-1} 1$ then $\bE_\de$ lies on the $k$th \textbf{level}, 
i.e. it is the middle class $\bE_\mu$ in a generating triple $\Tt_\de$ formed from $\Tt_{.1}$ by $k-1$ steps;
\item[{\rm(iii)}] the labelling respects order:  if $\de< \de'$ then the $z$-blocked interval  $I_{\bE_\de}$ lies to the left of $I_{\bE_{\de'}}$.
\end{itemlist}
\end{lemma}
\begin{proof}  The label is assigned inductively:  if the triple 
$\Tt_\de: = \bigl(\bE_\la,\bE_\mu = \bE_\de,\bE_\rho\bigr)$ with $\de: =  .a_1\cdots a_{k-1} 1$ is already labelled, then the middle classes $\bE_{x\mu}, \bE_{y\mu}$ of its two derived triples are labelled
$$
\bE_{x\mu} = \bE_{.a_1\cdots a_{k-1} 01}, \qquad \bE_{y\mu} = \bE_{.a_1\cdots a_{k-1} 21},
$$
so that the corresponding triples are
\begin{align}\label{eq:labelT}
x\Tt_{.a_1\cdots a_{k-1} 1}: = \Tt_{.a_1\cdots a_{k-1} 01},\qquad 
y\Tt_{.a_1\cdots a_{k-1} 1}: = \Tt_{.a_1\cdots a_{k-1} 21}.
\end{align}
It is evident that this labelling has the properties claimed: see Fig.\ref{fig:decimal}. Note that for $k\geq1$ there are $2^{k-1}$ classes at level $k$. 
\end{proof}
 
We next consider the properties of the graph whose vertices $\de = \de(w)$ consist of finite ternary decimals in $(0,1)$ with precisely one occurrence of $1$ at the end, together with the edges joining $\de(w)$ to $\de(xw)$ and 
$\de(yw)$.  
The edges of this graph are colored  yellow in Figure~\ref{fig:decimal}; the initial vertex is $E_{.1}$ at level $1$.   The {\bf augmented graph} is obtained from this one by adding the remaining edges in the triples; these are dotted grey in the figure.
 
 \begin{prop}\label{prop:label1}  This labelling has the following properties.
 \begin{itemlist} \item[{\rm (i)}] Each class $\bE_\de$ (where $\de\notin \Z$) is the middle entry in a unique triple $\Tt_\de$. 
  \item[{\rm (ii)}]  The entries in $\de$ record  in reverse order the sequence of moves $x,y$ needed to mutate from $\Tt_{.1}$ to $\Tt_\de$. 
 For example, $\Tt_{.0022201} = xyyyxx\Tt_{.1}$. 
 \item[{\rm (iii)}] Except in the initial case $\de =.1$,  the other two elements $\bE_{\de(w),\la}, 
 \bE_{\de(w),\rho}$ in $\Tt_\de$ lie on different levels with exactly one of them on the immediately preceding level.
Further, $\bE_{\de(w),\la}$ lies on a higher level than $\bE_{\de(w),\rho}$  if and only if $\de$ ends in $01$, so that the last mutation was via $x$.
\item[{\rm (iv)}] 
Let $w: = x^{m_1}y^{n_1}\cdots x^{m_j}y^{n_j}$ be a word in $x,y$ with all $m_i,n_i>0$ except possibly for $m_1, n_j$ which may be $0$, and let $\de(w): =  .2^{n_j} 0^{m_j}\cdots 2^{n_1}0^{m_1}1$ be the corresponding ternary decimal.  Then, assuming $w$ is nonempty,  
the other entries in the triple $w\Tt_{.1}: = \Tt_{\de(w)}$ are $\bE_{\de(w),\la}, \ 
\bE_{\de(w),\rho}$ where 
\begin{align}\label{eq:backlevel}
\bE_{\de(w),\la}&= \begin{cases} \bE_{\de(w'),\la}  & \mbox{ if } w = xw'\\
                             \bE_{\de(w')} =: \bE_{\de(w'),\mu}  & \mbox{ if } w = yw'\end{cases}\\ \notag
 \bE_{\de(w),\rho}&= \begin{cases} \bE_{\de(w')} = : \bE_{\de(w'),\mu}  & \mbox{ if } w = xw'\\
 \bE_{\de(w'),\rho}  & \mbox{ if } w = yw'.
                             \end{cases}
\end{align} 
\item[{\rm (v)}] In particular, in the triple $w\Tt_{.1}$ the words in $x,y$ that describe the other two elements consist of suitable initial segments of $w$.                         
More precisely, there is a triple with vertices $\de_\la,\de_\mu,\de_\rho$ iff one of the following conditions hold:\footnote
{Here we assume $w$ is not empty, but otherwise allow empty words as necessary.} 
\begin{itemize}\item[-] If $\de_\mu = \de(w)$ where $w= x^{m_1}\cdots y^{n_j}$ and $m_1>0$,  then 
$\de_\rho = \de(x^{m_1-1}\cdots y^{n_j})$ and $\de_\la =  \de(y^{n_1-1}\cdots y^{n_j})$. 
\item[-] If $\de_\mu = \de(w)$ where $w= y^{n_1}\cdots y^{n_j}$ and $n_1>0$,  then 
$\de_\la = \de(y^{n_1-1}\cdots y^{n_j})$ and $\de_\rho =  \de(x^{m_2-1}\cdots y^{n_j})$.
\end{itemize}
\item[{\rm (vi)}] The two  pre-staircases generated by the triple $\Tt_\de = \bigl(\bE_{\de,\la}, \bE_\de,\bE_{\de,\rho}\bigr)$, where $\de = \de(w)$ consist of the classes
$\bE_{\de(x^jw)}, j\ge 0$ (descending)  with recursion parameter $t_{\de,\la}$, and $\bE_{\de(y^jw)}, j\ge 0,$ (ascending) with recursion parameter $t_{\de,\rho}$. 
\item[{\rm (vii)}] Each {\bf edge} $\vareps$  in the augmented graph is a lower edge (i.e. one with an endpoint at the middle vertex $\bE_\mu$)   in a unique triple $\Tt_\vareps$, and can be associated to the unique vertex in $\Tt_\vareps$ that is not one of its endpoints; thus the two classes at the endpoints of $\vareps$  are part of a recursion whose recursion parameter is the $t$ component of   the third class in $\Tt_\vareps$. 
\end{itemlist}
\end{prop}
\begin{proof}  The proof is a straightforward induction that is left to the reader.\end{proof}

\begin{example} \label{ex:backlevel}
\rm Here are some examples of the recursive formula in \eqref{eq:backlevel}.

\begin{itemlist}\item[-] 
if $w = x^3yx$, then 
\begin{align*}& \bE_{\de(x^3yx),\la} = \bE_{\de(yx),\la} = \bE_{\de(x)}, \;\; i.e.\ \ \  
\bE_{.020001,\la} = \bE_{.01},\quad\mbox{  while }\\
& \bE_{\de(x^3yx),\rho} = \bE_{\de(x^{2}yx)},\;\;  i.e. \ \bE_{.020001,\rho}=\bE_{.02001}.
\end{align*}
\item[-]  if $w = x^3y$, then
\begin{align*}
&\bE_{\de( x^3y),\la} = \bE_{\de(y),\la} = \bE_{\de(\emptyset)}, \ \ \  i.e.  \ \ \ 
\bE_{.20001,\la} = \bE_{.1}, \quad\mbox{ while}\\
& \bE_{\de( x^3y),\rho} = \bE_{\de(x^{2}y)}, \ \ \ i.e.\  \bE_{.20001,\rho} = \bE_{.2001}.
\end{align*}
\end{itemlist}
\end{example}

The above notation allows us to recognize the generating triples and pre-staircases. 
Notice first that the set of triples $\Tt$ (or equivalently, the set of classes $\bE_\de$ that form their middle entries) is partially ordered, with order relation generated by the elementary steps $\bE_{\de(w)} < \bE_{\de(xw)} $ and $
\bE_{\de(w)} < \bE_{\de(yw)} $.  This order relation has the property that each element $\bE_{\de(w)}$ has precisely two successors, $\bE_{\de(xw)},\bE_{\de(yw)}$ 
that are the middle elements in  the triples $x\Tt, y\Tt$ respectively (where $\Tt: = \Tt_{\de(w)}$),  but (when $w\ne \emptyset$) exactly one predecessor $\pre(\bE_{\de(w)})$ that lies in $\Tt$.  Further,
$\pre(\bE_{\de(w)})$ is the middle entry of   a triple $\pre(\Tt)=:\Tt'$ that
has one other vertex, called $ \ppre(\bE_{\de(w)})$,  in common with $\Tt$.  Indeed, the construction implies that $\Tt$ equals either $x\Tt'$ or $y\Tt'$;
and the two triples $\Tt', x\Tt'$ share the vertices $\Tt'_\la, \Tt'_\mu$, while $\Tt', y\Tt'$ share the vertices $\Tt'_\rho, \Tt'_\mu$. Notice that in both  cases, 
the three classes $\bE, \pre(\bE),\ppre(\bE)$ form a triple (when appropriately ordered) with $\bE$ as the middle entry, and that these three classes lie on different levels, first  $\ppre(\bE)$, then $ \pre(\bE)$, and  then $\bE.$ Further, the classes $\pre(\bE)$ and $\bE$ belong to a pre-staircase with recursion parameter 
$t_{\ppre(\bE)}$. Note also that the two classes $\ppre(\bE), \pre(\pre(\bE))$ are usually different.

We then  inductively assign an integer  $\ell_{\Cc\Ss}(\bE_\de)$ (called its {\bf $\Cc\Ss$-length})   to each vertex $\bE_\de$ as follows:
\begin{itemize}\item $\ell_{\Cc\Ss}(E_0) = \ell_{\Cc\Ss}(E_1) = 1,\ \ell_{\Cc\Ss}(E_{.1}) = 2;$
\item if  $\Cc\Ss$-lengths are already assigned to all the vertices on  level $k$, then we assign  $\Cc\Ss$-length  to those on level $k+1$ as follows:
\begin{align}\label{eq:wtrecur0}
\ell_{\Cc\Ss}(\bE_{\de}) = \ell_{\Cc\Ss}(\pre(\bE_{\de})) + \ell_{\Cc\Ss} (\ppre(\bE_\de)).
\end{align}
\end{itemize}
Thus the classes  $\bE_{.01}, \bE_{.21}$ on level three have  $\Cc\Ss$-length $3$, while those on level four divide into two types:
we have
\begin{align*} 
\ell_{\Cc\Ss}(\bE_{.001})& = \ell_{\Cc\Ss}(\bE_{.01}) +\ell_{\Cc\Ss}(\bE_{0}) = 
\ell_{\Cc\Ss}(\bE_{.1}) + 2\ell_{\Cc\Ss}(\bE_{0}) = 4,\\
\ell_{\Cc\Ss}(\bE_{.021}) &  = 
 \ell_{\Cc\Ss}(\bE_{.01}) +\ell_{\Cc\Ss}(\bE_{.1}) = 
\ell_{\Cc\Ss}(\bE_{0}) + 2\ell_{\Cc\Ss}(\bE_{.1}) = 5.
\end{align*} 
Another way to understand the formula~\eqref{eq:wtrecur0} is to assign  $\Cc\Ss$-lengths recursively  to the edges $\vareps$ of the 
augmented graph giving them the same  $\Cc\Ss$-length as the corresponding vertex (see Proposition~\ref{prop:label1}~(vii)).  In this language, it reads
\begin{align}\label{eq:wtrecur}
\ell_{\Cc\Ss}(\bE_{\de}) = \ell_{\Cc\Ss}(\pre(\bE_{\de})) + \ell_{\Cc\Ss}(\vareps),
\end{align}
where $\vareps$ is the edge joining the two vertices $\pre(\bE_{\de}), \ \bE_{\de}$.
\MS

\begin{conjecture} \label{conj:1}
We conjecture that the following rules hold.\footnote{The notion of $CF$-length is defined in \eqref{eq:WCFlength}.} 

\begin{itemlist}\item[{\rm(i)}]   The continued fraction expansion of the center $p/q$ of the class $\bE_\de$ has $CF$-length $\ell_{\Cc\Ss}(\bE_\de)$.
\item[{\rm(ii)}] The steps of the two staircases associated to $\bE_\de$ have periodic continued fractions with periodic part of length $2(\ell_{\Cc\Ss}(\bE_\de))$.  Moreover,  these periodic parts have reverse cyclic order.
\end{itemlist}
\end{conjecture}

For example, the class $\bE_{.1}$ has $\Cc\Ss$-length  $2$ and center $[7;4]$.  Its two staircases are periodic of period $4$ with steps
\begin{align*}
&[\{7,5,3,1\}^k, \eend],\ k\ge 0, \quad\mbox{ where }\eend = 6,\mbox{ or } (7,5,2),\ \ \mbox{ ascending} \\
&[7,3, \{5,7,1,3\}^k, \eend], \ k\ge 0, \quad\mbox{ where }\eend = 6,\mbox{ or } (5,7,2),\ \ \mbox{ descending}
\end{align*}
Further, the reverse of the periodic part $7,5,3,1$ of the ascending staircase  is $1,3,5,7$, which agrees with  
the periodic part $5,7,1,3$ of the descending staircase, modulo a cyclic permutation. 
See Remark~\ref{rmk:Farey} for a different perspective on these results.

Proposition~\ref{prop:adjac} explains what we know about the continued fraction expansions of the centers of a pair of adjacent steps $\bE_\de,\ \bE_{x\de}$ or $\bE_\de, \ \bE_{y\de}$, while Proposition~\ref{prop:short} establishes what we can prove about the related notion of the weight length of the step centers; see \eqref{eq:WCFlength}.

\begin{rmk}\label{rmk:n>0}\rm  (i) To keep things simple, above we only considered the classes with centers in the interval $[6,8]$. However, we saw in Example~\ref{ex:SsU}
that for each $n$ the classes $(\bB^U_n, \bE_{[2n+7, 2n+4]}, \bB^U_{n+1}\bigr)$   form a generating triple.
It follows that, for each $n>0$, the set $\Ee_{[2n+6,2n+8]}$  of all quasi-perfect classes with centers in $[2n+6,2n+8]$ has precisely the same structure as $\Ee_{[6,8]}$. 
Indeed,
we may assign to the classes  in the initial triple  $\Tt^n_*: = (\bB^U_n, \bE_{[2n+7; 2n+4]}, \bB^U_{n+1})$ the labels $0^n,\ .1^n,\, 1^n$, and then label the derived classes  by
$\de^n$ where $\de$ is a ternary decimal as before.  Thus the center classes of the first two derived triples
$x \Tt^n_*$ and $y\Tt^n_*$ have labels $\de^n_{.01},\ \de^n_{.21}$, and so on.  Conjecture~\ref{conj:1} should hold for all $n$.\MS

\NI (ii) 
Finally we observe that the $\Cc\Ss$-length of a vertex can be given the following interpretation.    The coefficients of the classes are  polynomials in $n$:
by \eqref{eq:SsUsteps}, the classes $\bB^U_n$ are at level zero and have coefficients $p,t$ that are linear in $n$, while $q=1$ is constant.  
Further the $p,t$ coefficients of the classes at level $1$ are quadratic in $n$.  Thus in these initial cases 
the degree of $p,t$ as a function of $n$ is  just the $\Cc\Ss$-length assigned to this class.  One can then easily prove by induction that  the $\Cc\Ss$-length of a vertex is the degree of $p,t$ as a function of $n$.   The $\Cc\Ss$-length of an edge (see \eqref{eq:wtrecur}) can then be understood as the degree of the recursion variable for the pre-staircase 
with first two steps given by its endpoints.

This interpretation fits in with the above conjecture since in the cases we have calculated the periodic part of the ascending stairs has entries of the form $2n+i$ for $i\in \{1,3,5,7\}$, and the degree of the corresponding recursion is half the length of this periodic part.  For example, by
 \eqref{eq:SsUparam} the ascending staircase with blocking class $\bE^n_{.1} = [2n+7; 2n +4]$ has periodic part $(2n+7, 2n+5, 2n+3, 2n+1)$ with recursion variable $t = 4n^2 + 16n + 13$. \MS

\NI (iii)  For further comments on these matters see Remark~\ref{rmk:Farey}.  In the table given there, we write the staircase accumulation points in a  form that is slightly different (but equivalent) to that conjectured above, in order to emphasize the relation to the blocking class.
\hfill$\er$
\end{rmk}

We are now in a position to prove that the blocked $z$-intervals form a dense subset of $[6,\infty)$.
Recall that for each quasi-perfect class $\bE$, the interval $I_\bE: = \acc(J_\bE)$ is defined to be the set of $z$ such that
$\mu_{\bE, \acc_\eps^{-1} (z)} (z) > V_{\acc_\eps^{-1} (z)}(z)$,  where $\acc_\eps^{-1}$ is the appropriate inverse to the accumulation function.  (In the current situation we take the inverse with values in $(1/3,1)$ since all the classes $\bE$ of relevance here have $m/d>1/3$.) Note that the  claim in the following proposition about the length of the blocked interval $I_{\bB^U_n}$  was proved in \cite{ICERM} by direct calculation.  The current proof is computationally much easier.

\begin{prop}\label{prop:dense}  Let $\Ee_{[6,\infty)}$ be the set of all quasi-perfect classes with centers in $[6,\infty)$.
Then they block a dense subset of $[6,\infty)$; more precisely 
$$
{\it Block}_{[6,\infty)}: = \bigcup_{\bE\in \Ee_{[6,\infty)}}  I_\bE \cap [6,\infty) \quad\mbox{ is a dense subset of } \ [6,\infty).
$$
Moreover the length of the  interval blocked by the class $\bB^U_n$ converges to $2$ as $n\to \infty$.
\end{prop}

In fact all the above intervals $I_\bE$ lie entirely  in $[6,\infty)$ except for  $\bE = \bB^U_0$. 

\begin{proof}  Since it suffices to prove the first claim for each $n\ge 0$, we will begin with the case $n=0$.  Thus, consider the subset $\Ee_{[6,8], {\ell}}$ of classes in  $\Ee_{[6,8]}$ with level $\le {\ell}$.  We will show by induction that the set
$$
{\it Block}_{[6,8],{\ell}}: = \bigcup_{\bE\in \Ee_{[6,8]},{\ell}}  I_\bE \cap [6,8],\quad {\ell}\ge 0,
$$ 
is $2^{-({\ell}+1)}$-dense\footnote
{This constant could definitely be improved;  but it is all we need in the current context.}
    in $[6,8]$.  When ${\ell}=0,1$ this is clear since $7\frac 14 = [7;4]$ is the center of $\bE_{[7;4]}$  and all the points  $<7$ or $> 7\frac12$ are blocked by either $\bE_6$ or ${\bE_8}$.
We will prove the following: 
\begin{itemlist} \item Any recursively defined parameter at least doubles as one moves from one level to the next.  In particular, the $q$ coefficient of all classes at level ${\ell}\ge 1$ is at least $2^{{\ell}+1}$.
\item  if $\Tt$ is any triple with $\bE_\mu$ at level ${\ell}$, then in each of the two associated pre-staircases 
$\Ss^\Tt_\ell, \Ss^\Tt_u$ 
(see Lemma~\ref{lem:2})
\begin{itemize}
\item[-]  the distance between the centers of the second\footnote
{
This step is $\bE_\mu$ itself.}  and third steps  is at most $1/2^{{\ell}+1}$;
\item[-]  the distance from the center of the third step to the pre-staircase limit is $<1/2^{{\ell}+1}$. 
\end{itemize}
\end{itemlist}
Since the pre-staircase 
limit is in the closure of ${\it Block}_{[6,8],{\ell}}$,  it follows that for any point in the interval between the centers of $\bE_\la$ and $\bE_\rho$ the distance between that point and
${\it Block}_{[6,8],{\ell}}$ is at most $1/2^{\ell+1}$.

We argue by induction.  The first point is clear since all recursion parameters are at least $3$ (which is the $t$ coefficient for $\bB^U_0$).  Since
the $q$-coefficient of $\bE_{.1} = \bE_{[7;4]}$ at level one is $4$, it must at least double as one goes from level to level. 
To prove the second point, notice that if $p/q$ and $p'/q'$ are the centers of successive steps in a pre-staircase 
with recursion parameter $t''$, then because these steps are adjacent and $t''$-compatible we have
$|p/q-p'/q'| = t''/(qq')$  by Lemma~\ref{lem:recur0}~(iii).    Both pre-staircases $\Ss^\Tt_\ell, \Ss^\Tt_u$ 
have second step $\bE_{\mu}$, which lies on level ${\ell}$, and so by hypothesis has $q\ge 2^{{\ell}+1}$.  
Further, because the third step lies on a deeper level than the second we  can estimate
$$
\frac{t''}{qq'} = \frac{t''}{q (t'' q-q_0)}< \frac 1 q \le \frac 1{2^{{\ell}+1}},
$$
where $q_0$ comes from the previous step and we have used the obvious fact that $t''q - q_0> t''$, or equivalently 
$t''(q-1) > q_0$.   

Finally we need to check the distance from the center $p_3/q_3: = p'/q'$ of the third step to the limit $a_\infty$. This distance decomposes as a sum that we estimate as follows:
\begin{align}\label{eq:distell}
|p_3/q_3 - a_\infty| &= \sum_{k\ge 3} |p_{k+1}/q_{k+1} - p_k/q_k| 
\le \sum_{k\ge 3}  \frac{t''}{q_k q_{k+1}}\\ \notag
&\le   \frac {t''}{q_3^2} \Bigl(\frac 12 + \frac 1{2^2} +  \frac 1{2^3} + \dots\Bigr) \mbox{ since } q_{k+1}>2q_k\\ \notag
& \le \frac {t''}{q_3 (t'' q_2-q_1)} \le \frac 1{q_3} \le \frac 1{2^{{\ell}+1}}.
\end{align}
This completes the proof when $n=0$.

The same argument works for general $n$, except that now we can get better estimates since the recursion parameter is at least $n+3$, so that   at each stage
each recursively defined parameter increases by a factor of at least $n+2$.  
In particular, because the $q$-coefficient of the level one class $\bE^n_{.1}$ is $2n+4$, 
the $q$-coefficient of each class at level $\ell$ in $\Ee_{[2n+6, 2n+8]}$ is at least $2(2+n)^\ell$.
Hence the distance $|p_3/q_3 - a_\infty|$ estimated in \eqref{eq:distell} has upper bound 
$\frac 1{2(2+n)^{{\ell}}}$ and hence converges to $0$  for each fixed $\ell\ge 1$   as $n\to \infty$.
In particular, when $\ell = 1$ this says that the distance between the centers of the third steps $\bE^n_{.01}, 
\bE^n_{.21}$ and the limits $\p^+ I_{[2n+6]}, \ \p^-I_{[2n+8]}$ converges to zero.  
But the joint second step has center at $[2n+7; 2n+4]$ which converges to $2n+7$ as $n\to \infty$.  
Further the distance between this 
joint second and the third step of each pre-staircase also converges to zero.  Hence the 
distances $\p^+ I_{[2n+6]}- 2n+7$ and $\p^- I_{[2n+8]}- 2n + 7$ both converge to zero, which implies that
 the length of each interval $I_{[2n+6]}$ must converge to $2$.
 \end{proof}

We now define the rest of the staircases mentioned in Theorem~\ref{thm:main}, using the notation for symmetries introduced in \S\ref{ss:symm}; see in particular \eqref{eq:SR}, \eqref{eq:acceps}.
For $\eps \in \{\pm1\}$ and $\de\in \{0,1\}$,  write
\begin{align}\label{eq:whZ}
& \widehat{Z}: = \widehat{Z}_{+1} \cup \widehat{Z}_{-1},\qquad \widehat{Z}_\eps: = \bigcup_{\{(i,\de)\,:\,  (-1)^{i+\de}\, =\, \eps\}}\ \whZ_{i,\de},\\ \notag
&\whZ_{i,\de}: =  \bigl\{ S^iR^\de(z)\ \big| \ z\in [6, \infty), \acc_{+1}^{-1}(z)\notin {\it Block} \bigr\},
\end{align}
and define $Z_\eps\subset \widehat{Z}_\eps$ to consist of all points that are not endpoints of blocked intervals.  Thus $\widehat{Z}$ is a countable union of Cantor sets, while $Z$ is the complement of the set of endpoints.
It is convenient to assign ternary decimal labels to the points in $\whZ$ in the following way, where as before we simplify notation by describing it for the steps in $[6,8]$, with the understanding that $n$ is added if $z\in [2n+6,2n+8]$, and extra decorations $(i,\de)\in \N^+\times \{0,1\}$ are added for the points in $\whZ\cap \bigl(S^iR^\de([6,\infty))\bigr)$.

Consider the triple $\Tt_\de$ with $\bE_\mu = \bE_\de$, where 
$\de = .a_1a_2\dots a_i 1$.
The ascending principal pre-staircase 
$\Ss^\Tt_\ell$ has steps  $\bigl(\bE_{\de^+_k}\bigr)_{k\ge 0},$ where 
$\de_{\ell,k} = 0.a_1a_2\dots a_i 2^{\times k}1$ and hence has limit with infinite decimal label
$\de_{\ell,\infty} = 0.a_1a_2\dots a_i 2^{\times \infty}$.  Similarly, the limit of the descending 
principal pre-staircase 
can be given the  decimal label $\de_{u,\infty} = 0.a_1a_2\dots a_i 0^{\times \infty}$.
More generally if $\al_\infty = 0.a_1a_2\dots a_i \cdots$ is any infinite decimal with entries from $\{0,2\}$ 
 then we claim that there is a corresponding sequence of classes  $\Ss_{\al_\infty} = \bigl(
 \bE_{\al_k}\bigr)_{k\ge 1}, $ where $\al_k = 0.a_1a_2\dots a_k1$.  Note that, unless $\al_\infty$ ends in repeated $0$s or $2$s, the corresponding sequence of step centers contains both increasing and decreasing subsequences because  $\al_k < \al_\infty$ (resp. $\al_k > \al_\infty$)  for all $k$ such that $a_{k+1} = 2$ 
 (resp. $a_{k+1} = 0$).   Further, there is a bijection between these decimals and sequences of 
 vertices in the graph illustrated in Fig.~\ref{fig:decimal} that go down one level at each step.

 The next definition defines the monotonic pre-staircases $\Ss^{\pm 1}_{\al_\infty}$ by picking out particular   subsequences of steps.  Note that there are many ways of doing this, all of which would work equally well.
\MS

\begin{definition}\label{def:newstair}
For each  $\al_\infty \in Z$ we define $b_{\al_\infty}$ to be the limit of the ratios $m_{\al_k}/d_{\al_k}$ of the degree coordinates of $ \bE_{\al_k}$.     Further we define the {\bf ascending pre-staircase $\Ss^+_{\al_\infty}$} to have steps at $\al_{n_k}$ where $(\al_{n_k})_{k\ge 0}$ is the maximal subsequence such that $a_{n_k} = 0, a_{n_k+1} = 2$.   Similarly, the 
{\bf descending pre-staircase $\Ss^-_{\al_\infty}$} has steps at $\al_{n_k}$ where $(\al_{n_k})_{k\ge 0}$ is the maximal subsequence such that $a_{n_k} = 2, a_{n_k+1} = 0$. 
\end{definition}

In this definition (as in Definition~\ref{def:prestair}) we use the word {\bf pre-staircase}  to refer to this sequence of perfect classes whose center points $p_k/q_k$ converge to a limit $z_\infty$, since 
 the word staircase should be reserved for a sequence of classes that are live at the limiting $b$-value.  
Thus this notion of pre-staircase does not assume that when $b = \acc^{-1}(z_\infty)$ the corresponding functions $\mu_{\bE_k, b}$ are even obstructive at the center points $p_k/q_k$, let alone live.   
Note also that it is considerably more general than that in \cite[\S1.2]{MM}, since the steps are no longer required to be recursively defined.

Here is our main result about these pre-staircases.  

\begin{prop}\label{prop:liveZ}  \begin{itemlist} \item[{\rm {(i)}}]
For each ${\al_\infty}\in Z$, all the  steps in the ascending pre-staircase 
$\Ss^+_{\al_\infty}$ are live at $b=b_{\al_\infty}$.   
\item[{\rm {(ii)}}]  For each staircase family $(S^iR^{\de})^\#\bigl(\Cc\Ss^U\cap[2n+6,2n+8]\bigr)$, there is a constant $D_0$ such that each  step in each
descending pre-staircase $\Ss^-_{\al_\infty}$ of degree $\ge D_0$ is 
 live at $b=b_{\al_\infty}$. 
\end{itemlist}
\end{prop}

Here is an immediate consequence.

\begin{cor}\label{cor:liveZ}  Every $b$ that is not in the special rational sequence  is either blocked or has a  staircase.
Moreover, unless $b$ is in the closure of a blocked interval, it admits both an ascending and a descending staircase.
\end{cor} 
The special rational  $b$ are described  in \eqref{eq:specialb}.
The proofs are given in \S\ref{ss:uncount}.

 \subsection{Effect of symmetries}\label{ss:symm}
 
We explain how two symmetry operations from \cite{MM} act on perfect classes and hence on triples, and prove several key facts about this action, including Theorem~\ref{thm:main1}~(ii) in Corollary~\ref{cor:main1iipf}.
The action of these symmetry operations $S^\sharp,R^\sharp$ 
 is illustrated in Fig.~\ref{fig:symm}.   Since this is rather complicated, we begin with a brief recap of some results in \cite{MM}.

  The symmetries act on the $(p,q,t)$ coordinates by  $\bx: = (p,q,t)^T\mapsto \bx': = (p',q',t')^T$ via the matrices $S,R$ given by
\begin{align}\label{eq:SR}
 S: = \left(\begin{array}{ccc} 6&-1&0\\1&0&0\\0&0&1\end{array}\right), \quad R: = \left(\begin{array}{ccc} 
 6&-35& 0\\1&-6&0\\0&0&1\end{array}\right).
\end{align}
Thus $t$ is fixed, while $S,R$ induce the following action on the $z$-coordinate:
\begin{align*}
 &S: (a_{\min},\infty) \to (a_{\min},6), \quad p/q \mapsto (6p-q)/p,\\
& R: (6, \infty)\to (6,\infty), \quad p/q\mapsto (6p-35 q)/(p-6q).
\end{align*}
Thus $S$ is a shift to the left, while $R$ is a partially defined reflection that fixes the point $7$.
If we define the $y_i$ and $v_i$ recursively  by 
\begin{align}\label{eq:symmvi}
v_i  = y_i/y_{i-1}, \quad y_0=0, \ y_1 = 1, \ y_{i+1} = 6y_i-y_{i-1}, \ i\ge 1
\end{align}
we have $S(v_i) = v_{i+1}$; note that  $v_i\to a_{\min} = 3 + 2\sqrt2$.   
It is also useful to consider the points $w_i: = \frac{y_{i+1}+y_{i}}{y_i+y_{i-1}} = S(w_{i-1}), i\ge 1,$ given by the sequence $7, 41/7, \dots$.  
 \begin{figure}[htbp] 
 \centering
\vspace{-1 in}
\centerline{\includegraphics[width=6in]{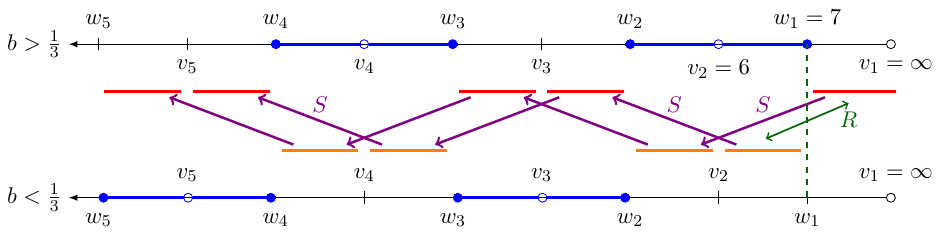}}
\vspace{-5.8 in}
 \caption{The complete family $\Cc\Ss^U$ consists of classes with centers in $(w_1,\infty)$; the reflection $R$ takes this to a complete family with centers in $(v_2,w_1)$ and $b<1/3$, and $S^i$ shifts every such family $i$ units to the left, changing the sign of $\eps$ by $(-1)^i$.  Note that $R$ reverses $z$-order, while $S$ preserves it.  The blue intervals $[w_{i+2}, w_{i+1}]$ are blocked by  the blocking classes $(S^i)^{\sharp}(\bB^U_0)$ for the appropriate $b$-value.  
 }
\label{fig:symm}
\end{figure}

    By \cite[Lem.2.1.1]{MM} we have
\begin{align}\label{eq:acceps}
b_{\eps,i+1}: =     \acc_\eps^{-1}\left(\frac{y_{i+1}}{y_i}\right)=\frac{y_{i+1}+y_i+3\eps}{3y_{i+1}+3y_i+\eps}, 
\end{align}
where $\eps = \pm1$, and we write $ \acc_\eps^{-1}$ for the appropriate branch of the inverse; thus  
$\acc_{\eps}^{-1} $ has image in $(1/3,1)$ if $\eps=1$ and in $(0,1/3)$ if $\eps = -1$.
It is shown in \cite[Rmk.~1.2.5]{MM} that for all $i\ge 0$, when $b>1/3$, the class $(S^{2i})^{\sharp} (\bB^U_0)$ blocks $v_{2i+2}$, while when $b<1/3$,
$(S^{2i+1})^{\sharp} (\bB^U_0)$ blocks $v_{2i+3}$ (we define the action of a symmetry on a quasi-perfect class below). Thus $b_{\eps,i+1}$ is blocked when either  $\eps = 1$ and $i$ is odd or when  $\eps = -1$ and $i$ is even.  However,  the following $b$-values are not blocked, because they are in the closure of $Stair$ by \cite[Thm.1.1.1]{MM}: 
\begin{align}\label{eq:specialb}
1/5, \ 59/179,\dots, b_{-1,2i}, \dots  \qquad  11/31, \ 349/1045, \dots, b_{1,2i+1}, \dots
\end{align}
These values of $b$ with $\eps=(-1)^i$ will be called the {\bf special rational $b$}, and denoted by
\begin{equation}\label{eq:bi}
b_i:=b_{(-1)^{i-1},i} \mbox{ so that } b_2, b_3, b_4, \dots = 1/5, 11/31, 59/179, \dots
\end{equation}
 We will see that each such $b$ is the limit point of both ascending and descending sequences of $b$-values with staircases and hence is unobstructed; see Remark~\ref{rmk:symm0}~(iii).  However, none of the steps of these staircases remain live at one of these special $b$-values; indeed, according to Conjecture~1.2 in \cite{AADT},  these special $b$-values should not admit any staircases at all. 

\MS

 We extend the action of $S,R$  to quasi-perfect classes $\bE = (d,m,p,q,t, \eps)$ via the formulas in \eqref{eq:formdm} that define $d,m$ as functions of $p,q,t, \eps$, noting that both $S$ and $R$ preserve $t$ and act on $\eps$ by $\eps\mapsto -\eps$.    We denote this image by $(S^iR^\de)^\sharp(\bE)$.   
 The main result of \cite{MM} is that for each $i\ge 0, \de\in \{0,1\}$ the transformation $S^iR^\de$ takes the seeds and blocking classes in the staircase family $\Ss^U$ to the seeds and blocking classes of a staircase family that we denote 
 $(S^iR^\de)^\sharp(\Ss^U)$.  (See Remark~\ref{rmk:symm0} for a discussion of the slightly anomalous case $R^{\sharp}(\Ss^U) =: \Ss^L$.)   Moreover, for all such $i,\de$ the image of a perfect tuple by $(S^iR^\de)^\sharp$  is perfect.   

We now extend this result by showing that these symmetries take each generating triple with centers in $(7,\infty)$ to another generating triple.  

\begin{lemma}\label{lem:Tgen}  Let $T = S^iR^\de$ for some $i\ge 0, \de\in \{0,1\}$, and let
$$
\bE_0: = (d_0,m_0,p_0,q_0,t_0, \eps),\quad  \bE_1: = (d_1,m_1,p_1,q_1,t_1,\eps).
$$

\begin{itemlist} \item[{\rm(i)}]   If $\bE_0 , \bE_1 $ are adjacent, then so are $T^\sharp(\bE_0), \ T^\sharp(\bE_1)$.

 \item[{\rm(ii)}]  If $\bE_0, \bE_1$ are $t$-compatible, then so are $T^\sharp(\bE_0), \ T^\sharp(\bE_1)$.
 \end{itemlist}
\end{lemma}

\begin{proof}  Suppose first that $T = S$ and that $p_0/q_0<p_1/q_1$.  
Write $S^{\sharp}(\bE_i): = (d_i',m_i',p_i',q_i',t_i')$, $i=0,1$.  Since $S$ preserves order, to prove (i)   we must show that
$$
(p_0+q_0)(p_1+q_1) - t_0t_1 = 8p_0q_1\Longrightarrow (p_0'+q_0')(p_1'+q_1') - t_0't_1' = 8p_0'q_1'.
$$
But $ t'_0t'_1= t_0t_1$, and 
\begin{align*}
(p_0'+q_0')(p_1'+q_1')  - 8p_0'q_1' & = (7p_0-q_0)(7p_1-q_1)  - 8(6p_0-q_0)p_1\\
& = p_0p_1 + q_0p_1 - 7 p_0q_1 + q_0q_1\\
& = (p_0+q_0)(p_1+q_1) - 8 p_0q_1
\end{align*}
as required.  Thus the action of $S$, and hence of $S^i$, preserves adjacency.

Since $R$ reverses order, we must show that, with $R^{\sharp}(\bE_i): = (d_i',m_i',p_i',q_i',t_i')$, we have
$$
(p_0+q_0)(p_1+q_1) - t_0t_1 = 8p_0q_1\Longrightarrow (p_0'+q_0')(p_1'+q_1') - t_0't_1' = 8p_1'q_0'.
$$
This holds because
\begin{align*}
(p_0'+q_0')(p_1'+q_1')  - 8p_1'q_0' & = (7p_0-41q_0)(7p_1-41q_1)  - 8(6p_1-35q_1)(p_0-6q_0)\\
& = (p_0+q_0)(p_1+q_1) - 8 p_0q_1.
\end{align*}

Now consider (ii).  This holds because $S^TAS = R^TAR = A$ for $A$ defined in \eqref{eq:matrix}.  In other words, this holds because $S,R$ preserve $t^2$.
\end{proof}

\begin{prop}\label{prop:symm} Let $\bigl(\bE_{\la}, \bE_{\mu},\bE_{\rho}\bigr)$ be a generating triple whose classes have centers in $(7,\infty)$.
Then for any $i\ge 1$, the classes
$$
\bigl((S^i)^\sharp(\bE_{\la}), (S^i)^\sharp(\bE_{\mu}),(S^i)^\sharp((\bE_{\rho})\bigr),\quad\mbox{and} \quad
\bigl((S^iR)^\sharp(\bE_{\rho}), (S^iR)^\sharp(\bE_{\mu}),(S^iR)^\sharp(\bE_{\la})\bigr)
$$
also form generating triples.
\end{prop}

Notice here that because the action of $R$ on the $z$-axis is orientation reversing, the order of the three entries in the second triple above is reversed. 

\begin{proof}    
By Lemma~\ref{lem:Tgen} the action of $T = S^iR^\de$ preserves $t$ and hence the required adjacency and $t$-compatibility conditions.  Since the $t$ values are unchanged by these symmetries, the equalities among them are preserved.   Next observe that $\det S = 1 = -\det R$.  Hence, expressions such as $ q_{\mu}p_{\rho}-p_{\mu}q_{\rho}$ are unchanged by $S$.  Thus $(S^i)^\sharp$ preserves the three identities expressing $t$ as a function of the relevant $p,q$.  
The reader can check that $R^\sharp$ also preserves these identities since the minus sign in the determinant is compensated for by the fact that $R^\sharp$ interchanges the first and last elements of the triple.
\MS

It remains to check that the image triple $
\bigl(\underline{\bE}_{\la},\underline{\bE}_{\mu},\underline{\bE}_{\rho}\bigr)$   satisfies the last condition, namely
\begin{align}\label{eq:order}
 \acc(\ovm_{\rho}/\ovd_{\rho}), \acc(\ovm_{\la}/\ovd_{\la}) > \acc(\ovm_{\mu}/\ovd_{\mu}).
\end{align}
 Notice that 
\begin{align*}
 &\frac md = \frac{p+q+3\eps t}{3(p+q) + \eps t} >  \frac {m'}{d'} = \frac{p'+q'+3\eps t'}{3(p'+q') + \eps t'}\\
 &\qquad \Longleftrightarrow \eps t (p'+q') > \eps t'(p+q).
 \end{align*}  
 Therefore, because the function $b\mapsto \acc(b)$ reverses orientation precisely if $\eps = -1$, we may conclude that
\begin{align}
 \acc\Bigl(\frac md \Bigr)> \acc\Bigl(\frac {m'}{d'}\Bigr) \Longleftrightarrow  t (p'+q') >  t'(p+q).
\end{align}
In particular, in the initial triple  $(\bE_{\la}, \bE_{\mu},\bE_{\rho})$ we have $\acc(m_{\la}/d_{\la}),\acc(m_{\rho}/d_{\rho}) > \acc(m_{\mu}/d_{\mu})$ so that
 $$
 \frac{t_{\la}}{t_{\mu}}>\frac{p_{\la}+q_{\la}}{p_{\mu}+q_{\mu}} , \quad 
    \frac{t_{\rho}}{t_{\mu}}>\frac{p_{\rho}+q_{\rho}}{p_{\mu}+q_{\mu}} ,\quad p_{\la}/q_{\la} < p_{\mu}/q_{\mu} < p_{\rho}/q_{\rho}.
 $$
 
  It follows from the formulas in \cite{MM}  that if $(\ovp, \ovq)^T = S^iR^\de(p,q)^T$, then $\ovp + \ovq = rp-sq$ for some integers $r,s>0$.  Therefore, because the symmetries preserve the $t$-coordinate, 
 $$
\frac {p_{\la}+q_{\la}}{p_{\mu}+q_{\mu}} > \frac {\ovp_{\la}+\ovq_{\la}}{\ovp_{\mu}+\ovq_{\mu}} \Longrightarrow \ \acc(\frac {\ovm_{\la}}{\ovd_{\la}} )> \acc\bigl(\frac {\ovm_{\mu}}{\ovd_{\mu}}\bigr),
$$
But 
$$
  \frac {\ovp_{\la}+\ovq_{\la}}{\ovp_{\mu}+\ovq_{\mu}} =  \frac {r p_{\la}-s q_{\la}}{r p_{\mu}-s q_{\mu}}\  < \ \frac {p_{\la}+q_{\la}}{p_{\mu}+q_{\mu}}
  $$
  if  $s(p_{\la}q_{\mu} - q_{\la} p_{\mu} ) < r(q_{\la}p_{\mu} - p_{\la} q_{\mu})$ which always holds because 
  $q_{\la} p_{\mu} - p_{\la} q_{\mu}>0$ by hypothesis.
  A similar argument proves that $ \acc(\ovm_{\rho}/\ovd_{\rho}) > \acc(\ovm_{\mu}/\ovd_{\mu})$.
This completes the proof.
 \end{proof}

\begin{cor} \label{cor:symm}  
 Let $\Ss$ be a pre-staircase for $b = b_\infty$ that is the image by some symmetry $S^iR^\de$ of a pre-staircase in the complete family $\Cc\Ss^U$.  If $b_\infty>1/3$ then the ratios $m_k/d_k$ decrease, while if 
 $b_\infty<1/3$ then the ratios $m_k/d_k$ increase.  In all cases, $\acc(m_k/d_k)$ decreases.
 \end{cor}
 \begin{proof}  If $\Ss$ is a principal pre-staircase in $\Cc\Ss^U$, then this holds by Lemma~\ref{lem:2}~(iii).
   It therefore holds for all principal pre-staircases by \eqref{eq:order} and the fact that 
   the function $\acc: [0,1)\to [a_{\min},\infty)$ preserves order on the interval  $(1/3,1)$ and reverses it on $(0,1/3)$.  We can interpret these order relations as saying that if the levels (in the sense of Lemma~\ref{lem:label})  of the classes $\bE_k$ strictly increase, then the corresponding ratios $m_k/d_k$ decrease when $m_k/d_k> 1/3$  (resp. increase when $m_k/d_k< 1/3$).   This suffices to prove the lemma,  since the levels of the classes in any of the new pre-staircases in 
 Definition~\ref{def:newstair} strictly increase.
  \end{proof}

\begin{cor} \label{cor:main1iipf} Theorem~\ref{thm:main1}~(ii) holds.
Moreover, for every class $\bE \in \Cc\Ss^U$, and every symmetry $T = S^iR^\de$ we have $I_{T^\sharp(\bE)} = T(I_{\bE})$.
\end{cor}
\begin{proof} In light of Proposition \ref{prop:symm}, the first claim will follow once we prove that $S$ commutes with mutation and $R$ switches $x$-mutation to $y$-mutation and vice versa. This is straightforward, e.g. to show $S^\sharp(\bE_{x\mu})$ is the middle entry of $xS^\sharp(\Tt)$, we compute
\begin{align*}
S^\sharp(t_\la p_\mu-p_\rho,t_\la q_\mu-q_\rho)&=(6(t_\la p_\mu-p_\rho)-(t_\la q_\mu-q_\rho),t_\la p_\mu-p_\rho)
\\&=(t_\la(6p_\mu-q_\mu)-(6p_\rho-q_\rho),t_\la p_\mu-p_\rho).
\end{align*}
The remaining cases, including the analogous facts for the left and right entries, are left to the reader.

To prove the second claim, suppose that $\bE$ is the middle entry $\bE_\mu$ in the triple $\Tt$.  Then the  end points of the blocked $z$-interval $I_{\bE}$ are the accumulation points of the principal pre-staircases $\Ss^{x\Tt}_\ell$ and  $\Ss^{y\Tt}_u$.   The symmetry $T$ takes the steps, and hence accumulation points, of these pre-staircases  
to the steps and  accumulation points of the pre-staircases 
with blocking class $T^{\sharp}(\bE)$.
Hence it takes the end points of  $I_{\bE}$ to those of $I_{T^\sharp(\bE)}$.\footnote
{
For further discussion of the action of the symmetries on ${\it Block}$ see the end of \cite[\S1.2]{MM}.}  
\end{proof}

\begin{cor} \label{cor:qtrip}  The image of each quasi-triple $\Tt^n_{\ell,seed}, T^n_{u,seed}$ under a symmetry $S^iR^\de$ is another quasi-triple.
\end{cor}
\begin{proof}
By the definition in  Remark~\ref{rmk:quasitrip}, a quasi-triple satisfies all the condition of a triple except for (e).  Therefore, this holds by the first paragraph of the proof of Proposition~\ref{prop:symm}.
\end{proof}

\begin{rmk}\rm\label{rmk:symm0}  (i)  The complete family $R^{\sharp}(\Cc\Ss^U)$ is somewhat anomalous because the class $R^{\sharp}(\bB^U_0)$  is not  geometric, since the reflection $R: z\mapsto (6z-35)/(z-6)$ is not defined at $z=6$.  On the other hand, $R(6/1) = 1/0$, and if we define $d,m$ via \eqref{eq:formdm} we obtain the tuple $(d,m,p,q,t,\eps) = (0,-1,1,0,3,-1)$.  
As pointed out  in \cite[Rmk~2.3.4(ii)]{MM}, 
this is a numerically meaningful replacement for the class $R^{\sharp}(\bB^U_0)$.  For example, 
there is a descending staircase with blocking class $$
R^{\sharp}(\bB^U_1) = R^{\sharp}(\bE_{[8]}) = (5,0,13,2, 5, -1)
$$
(and hence recursion parameter $5$) that is generated by the tuples $$
R^{\sharp}(\bB^U_0) = R^{\sharp}(\bE_{[6]}) = (0,-1,1,0,3,-1), \qquad R^{\sharp}(\bE_{[7;4]}) =(13, 0, 34,5,13, -1).
$$
Its third step has center $\frac{169}{25}$ and degree variables $(d,m) = (65,1)$.
This is the descending staircase $\Ss^L_{u,1}$ associated to the triple $R^{\sharp}(\Tt^0_*)$. 
The corresponding ascending staircase $\Ss^L_{\ell,0}$ with first two steps $R^{\sharp}(\bB^U_1), R^{\sharp}(\bE_{[7;4]})$ is the Fibonacci staircase.  All the steps of the latter staircase have $m=0$, and there is no geometric blocking class.  However, because by 
Proposition~\ref{prop:symm} the three classes in $R^{\sharp}(\Tt^0_*)$ satisfy all the numeric conditions to be a generating triple, we may define the 
 complete family $R^{\sharp}(\Cc\Ss^U)$ as before.  Note that all its elements except for the steps in the
 Fibonacci staircase have
 $m>0$.  Further, with the sole exception of  
$R^{\sharp}(\bB^U_0)$, they all have have nonnegative entries with $p/q> a_{\min}$ and so are blocking classes. 

However, it is important to note that $R^\sharp(\bB^U_0)$ is not geometrically meaningful, and so certain results do not apply to it. For example, the analysis of the steepness of a staircase following Proposition \ref{prop:live} (used to rule out potentially overshadowing classes) cannot apply to this class because it does not form an actual step for any $c_{H_b}$. In particular, any technique relying on $m/d$ cannot be used with this class, because $m/d=-1/0$.

\MS

\NI (ii)  Fig.~\ref{fig:symm} illustrates the fact that the intervals $[w_{2i+2}, w_{2i+1}], i\ge 0,$ are blocked when $b>1/3$ while their complements $[w_{2i+1}, w_{2i}], i\ge 1,$ 
are blocked when $b<1/3$.  
To see that $\bB^U_0$ blocks $[w_{2}, w_{1}]$  one only needs to observe that the accumulation point $z_\infty$ of the descending staircase  $\Ss^U_{u,0} $ is $ [7;\{1,5\}^\infty] > 7 = w_1$, and that the reflection $SR^{\sharp}$ fixes the class $\bB^U_0$.    It follows that $w_2 = SR(w_1)$ is greater than the limit point of the ascending staircase $(SR)^{\sharp}(\Ss^U_{u,0})$, so that $[w_2,w_1]$ is blocked by $\bB^U_0$.  Then the rest of this claim follows by applying the symmetries.  

\MS

\NI (iii)   The fact that the centers of the blocking classes $\bB^U_n$ converge to $v_1: = \infty$ as $n\to \infty$ implies that the images of these classes under $R$ and $S$ have centers converging to  $v_2: = 6$. 
These classes have ratios $m/d< 1/3$ that converge to $1/5$ as $n\to \infty$. 
However, they do not form staircases at $b =1/5$ because (as is easy to check),\footnote
{
Formulas for the $(p,q,t)$ entries for the blocking classes $\bB^L_n=R^\sharp(\bB^U_n)$ are given in \eqref{eq:baL}; they have $(d,m) =(5n,n-1)$. Using \eqref{eq:pqt} allows us to similarly compute the $d,m$ coordinates of $S^\sharp(\bB^U_n)$.} the step corners are not visible at $b=1/5$ since they fail the test  $|bd-m|< \sqrt{1-b^2}$ at $b = 1/5$; see Lemma~\ref{lem:basic}.   Similarly, 
none of the classes in the complete families 
$S^\sharp(\Cc\Ss^U), R^\sharp(\Cc\Ss^U)$ are live at $b=1/5$,  since for each such class both $d$ and the quantity $|b-m/d|$ 
are larger than they are for the appropriate blocking class.
However, these classes do organize into 
  staircases whose accumulation points $(b_\infty, z_\infty)$ converge  to $(1/5,v_2)$, both from the left (in the family
$S^{\sharp}(\Cc\Ss^U)$) and from the right  (in the family
$R^{\sharp}(\Cc\Ss^U)$). Because the shifts $S^i$ act on these staircases, a similar claim holds for the points $v_{2i}$ with $b< 1/3$ and for $v_{2i+1}$ for $b>1/3$.  These are the only rational points on the $z$-axis that might be accumulation points of staircases, since all the others are blocked by Corollary~\ref{cor:short2}. In Section~\ref{ss:13stair}, we show that  these rational points $v_{i}, i\ge 3,$ do not have ascending staircases, but it remains unknown if they have descending staircases.  
Similarly, we do not know if there is a descending staircase at $b=1/3$.
\hfill$\er$
\end{rmk}

We will show in Proposition~\ref{prop:adjac} that for any pair of  adjacent classes  $\bE, \bE'$ we have
 $\bE\cdot \bE' = 0$.  Since the symmetries preserve adjacency, this implies that any pair of successive steps in any principal pre-staircase 
 have zero intersection. In fact, as seen in the following lemma, regardless of adjacency, the symmetries $T:=S^iR^\de$ preserve intersection number. 

\begin{lemma} \label{lem:inters}
\begin{itemlist}\item[{\rm (i)}] For quasi-perfect classes $\bE_1$ and $\bE_2$ with center greater than $5$, we have $S^\sharp(\bE_1) \cdot S^\sharp(\bE_2)=\bE_1 \cdot \bE_2$.
\item[{\rm (ii)}] For classes $\bE_1$ and $\bE_2$ with centers greater than $7$, $R^\sharp(\bE_1) \cdot R^\sharp(\bE_2)=\bE_1 \cdot \bE_2.$
\end{itemlist}
\end{lemma}
\begin{proof}
For (i), consider
\begin{align*}
S^\sharp(\bE_1) &=S^\sharp(d_1,m_1,p_1,q_1)=(D_1,M_1,P_1,Q_1)\\
S^\sharp(\bE_2)&=S^\sharp(d_2,m_2,p_2,q_2)=(D_2,M_2,P_2,Q_2).
\end{align*}
 Then, by \cite[Lemma 2.1.5]{MM}, for $i=1,2,$ 
we have\footnote
 {
 See \eqref{eq:Wpq} for the definition of the
 weight decomposition $W(p,q):= W(p/q)$.
 }
$$
W(P_i,Q_i)=W(6p_i-q_i,p_i)=(p_i^{\times 5},p_i-q_i) \sqcup W(p_i-q_i,q_i),
$$
 where by definition $W(p_i-q_i,q_i)$ is $W(p_i,q_i)$ with the first entry $q_i$ removed.
As
\begin{align*}
\bE_1 \cdot \bE_2=d_1d_2-m_1m_2-q_1q_2-W(p_1-q_1,q_1)\cdot W(p_2-q_2,q_2),
\end{align*} 
we want to show
\[ d_1d_2-m_1m_2-q_1q_2=D_1D_2-M_1M_2-5p_1p_2-(p_1-q_1)(p_2-q_2).\]
To check this, we use 
 the formulas for $d_i,m_i,D_i,M_i$ in terms of $p_i$ and $q_i$, except that 
 rather than the usual notation of including $\eps$ explicitly, we account for $\eps$  by letting $t_1,t_2$ be either positive or negative.
 Thus, we have $S: t_1 \mapsto -t_1$ and $S: t_2 \mapsto -t_2$. 
We then obtain: 
\begin{align*} 
d_1d_2-m_1m_2-q_1q_2&=\tfrac18\bigl((p_1+q_1)(p_2+q_2)-t_1t_2\bigr)-q_1q_2 \\
&=\tfrac18(p_1p_2-7q_1q_2+q_1p_2+q_2p_1-t_1t_2) \\
&=\tfrac18\bigl((7p_1-q_1)(7p_2-q_2)-t_1t_2\bigr)-5p_1p_2-(p_1-q_1)(p_2-q_2) \\
&=D_1D_2-M_1M_2-5p_1p_2-(p_1-q_1)(p_2-q_2).
\end{align*}
This completes (i).  The proof of (ii) is similar and left to the reader. 
\end{proof}

 \begin{rmk}\label{rmk:fakesymm}\rm 
 (i) Let $X_\Z = \{(p,q,t) \in \Z^+ \ | \ p^2 - 6pq + q^2 + 8 = t^2\}$.  
 It is not known whether there is an element $A\in {\rm GL}(3, \Z)$ that induces a map $X_\Z\to X_\Z$ that does not fix $t$.  
 Since there are integral triples $(p,q,t)\in X_\Z$ for which neither pair of numbers $(d,m)$ defined by \eqref{eq:pqt} are integral, one
might also require that $A$ preserve the integrality of the appropriate pair $(d,m)$. 
The symmetries considered above are the only elements of this group that fix $t$ (or equivalently fix the quadratic form $p^2-6pq + q^2$); see \cite[Lem.2.1.3]{MM}.  Further, by \cite[Lem.2.2.4]{MM}  they also preserve integrality. 

 \MS

 \NI (ii)  We now discuss the properties of a \lq symmetry' $A$ that is not in ${\rm GL}(3, \Z)$ but has some interesting features.

   Denote by $\Cc$ the set of integral tuples $(d,m,p,q,t)$ with $t>0$ (but possibly with some  negative entries) that satisfy the numeric conditions to be a quasi-perfect class with $\eps = 1$, i.e. 
 $p^2-6pq + q^2 +8= t^2$ and $d,m$ satisfy \eqref{eq:formdm}.  
   The subset with all entries nonnegative  is denoted $\Cc^+$.\footnote
   {
   Although $\Cc^+\subset X_\Z$, the two sets are different since  $X_\Z$ contains elements for which $d,m$ are not integers.}
Consider the
 transformation
\begin{align*}
 A^\sharp:&\ \Cc^+\to \Cc,\quad  (d,m,p,q,t) \to \bigl(m+3Q,d+Q,p-q+5Q,Q, p+q\bigr),\\
 &\hspace{1.8in} \quad\mbox{ where }\  Q: = \tfrac 12 (p-q-t).
\end{align*}
It is not hard to check that $p+q+t$ is even for all $\bE\in \Ee'$ so that $A(\bE)$ is always integral. One can check that  it also  satisfies the requirements to be in $\Cc$. 
Notice that
\begin{align*}
&A^\sharp(1,1,1,1,2) = (-2,0,-5,  -1,2),\quad \mbox{ i.e. } \ A^\sharp(\bE^U_{\ell,seed}) = \bE^U_{u,seed},\\
&A^\sharp(n+3,n+2, 2n+6,1,2n+3) = (n+5,n+4, 2n+10,1,2n+7),\;\\
&\hspace{2in}\mbox{ i.e. } A^\sharp(\bB^U_n)= \bB^U_{n+2}.
\end{align*}
Thus $A^\sharp$ acts on the seeds and blocking classes of the staircase family $\Ss^U$.  Because it is linear it therefore takes all steps in the ascending staircase $\Ss^U_{\ell,n}$ (that has blocking class $\bB^U_n$ and seeds
$\bE^U_{\ell,seed}, \bB^U_{n-1}$) to the corresponding descending staircase $\Ss^U_{u,n}$ with blocking class $\bB^U_n$ and seeds $\bE^U_{u,seed}, \bB^U_{n+1}$.    In particular, for all $n$ it takes the step
$ \bE_{[2n+7,2n+4]}$  that
$\Ss^U_{\ell,n}$ shares with $\Ss^U_{u,n+1}$ to the corresponding shared step 
$ \bE_{[2n+9,2n+6]}$  of the staircases
$\Ss^U_{\ell,n+1}$ and $\Ss^U_{u,n+2}$. 
However, although  $A^\sharp$  takes each of the three classes in the generating triple  $(\bB^U_n,  \bE_{[2n+7,2n+4]}, \bB^U_{n+1})$ to another perfect class, the image classes  
$\bB^U_{n+2},  \bE_{[2n+9,2n+6]}, \bB^U_{n+3}$ do  not form a generating triple.   Correspondingly there are many classes descended from $(\bB^U_n,  \bE_{[2n+7,2n+4]}, \bB^U_{n+1})$  whose image under $A^\sharp$  is not perfect.   
For example, consider the step $\bE_{[7;5,2]}$ after $\bE_{[7;4]}$ in the descending staircase $\Ss^U_{u,0}$.
Since $\bE_{[7;5,2]} = \bigl(38,24,79,11,34\bigr)$ we have $$
A^\sharp\bigl(38,24,79,11,34\bigr) = \bigl(75, 55, 153, 17, 90\bigr),
$$
with center $153/17 = 9$ and hence corresponding quasi-perfect class
$\bE = 75 L - 55 E_0 - 17 E_{1\dots9}$.  This class is not perfect because $\bE \cdot \bB^U_1 =
\bE\cdot (4L-3E_0 - E_{1\dots 8}) = -1.$  See also Remark~\ref{rmk:fakestair}. 
   \hfill$\er$
 \end{rmk}

  \section{The pre-staircases are live} \label{sec:main}
  In \S\ref{ss:live} we develop a criterion for a quasi-perfect class to be perfect and then apply it, together with the results from \cite{M1} to prove in Corollary~\ref{cor:perfATF} that all the classes in the complete family $\Cc\Ss^U$, as well as their images under the symmetries  are perfect. 
   Proposition~\ref{prop:live} then explains why the principal pre-staircases are live. The proof here needs a new slope estimate that is proved in Lemmas~\ref{lem:slopeest1} and~\ref{lem:slopeest3}. We extend these arguments in \S\ref{ss:uncount} to show that there are both ascending and descending staircases at every unblocked $b$, except possibly for the special rational $b$.

 \subsection{The principal pre-staircases are live}\label{ss:live}

We first establish some useful   sufficient conditions for a quasi-perfect class to be an exceptional class and hence perfect.
The following result is extracted from the proof of \cite[Prop.42]{ICERM}.
 
 \begin{lemma}\label{lem:live1}  Let $\bE = (d,m,p,q)$ be a quasi-perfect class  such that $\mu_{\bE, b}(p/q)$ is live for all $b$ in an open set $J\subset [0,1)$.  Then $\bE$ is perfect.
 \end{lemma}

\begin{proof}  Because, by \cite[Lem.15(ii)]{ICERM}, $c_{H_b}(p/q)$ is the maximum 
of the obstructions given by exceptional  classes, for each $b\in J$ there must be a exceptional  class $ \bE_b'$ such that $\mu_{\bE, b}(p/q)=\mu_{\bE'_b, b}(p/q)$. 
We saw in
\cite[Lem.15(i)]{ICERM} that if  the degree coordinates of $\bE'_b$ are $(d',m')$ 
where $|bd' - m'|<\sqrt{1-b^2}$,
then
$$
 V_b(z) < \mu_{\bE'_b,b}(z) \le V_b(z) \sqrt{1 + \frac{1}{(d')^2 - (m')^2}},
$$
One can get from this an upper bound for $d'$ in terms of the ratio $
\mu_{\bE, b}(p/q)/V_b(p/q).
$
Therefore, there can only be finitely many such classes $\bE'_b$, which implies that there is  an exceptional class $\bE' = (d', m';\bbm')$
and an open subset  $J_b$ of $J_\bE$  such that  
\begin{align*}
\mu_{\bE',b}(p/q) = \mu_{\bE,b}(p/q), \quad  \forall\ b\in J_b. 
\end{align*}
Then we may write
  $\bbm' = \la \bbm + \bn$ where $\bbm = W(p/q)$ is the weight expansion of $p/q$ (i.e. the coordinates of $\bE$) 
  and $\bbm \cdot  \bn=0$.
Since
$$
\frac{\bbm' \cdot \bw(p/q)}{d' - m'b} = 
\frac{\la p}{d' - m'b}  =
\frac{p}{d - mb} 
 $$
 for all $b\in J_b$, we must have  $d'=\la d,\ m'=\la m$ for some $\la > 0$.
 The identities 
\begin{align*}
& \bbm'\cdot \bbm' -1= (d')^2 - (m')^2 = \la^2 (d^2-m^2) = \la^2 (pq - 1),\\
& \bbm'\cdot \bbm'  -1 = \la^2 \bbm\cdot\bbm + \|\bn\|^2 - 1 = \la^2 pq + \|\bn\|^2 - 1
\end{align*}
then imply  that $ \|\bn\|^2=1-\la^2$.  Therefore, unless $\bn' = 0$ so that $\bE' = \bE$ we must have  $0<\la < 1$.
Further 
$$
\bE'\cdot \bE = d'd - m' m - \bbm'\cdot \bbm = \la (d^2-m^2 -\bbm\cdot \bbm) =  - \la.
$$
But $\bE'\cdot \bE$ is  an integer.    It follows that $\bE' = \bE$, so that $\bE$ is perfect as claimed.
\end{proof}

We next show that every quasi-perfect class that intersects nonnegatively with every exceptional class is perfect.  Although this follows from the general theory of exceptional curves in blowups of $\CP^2$, the proof below is  self-contained, using only the positivity of intersections of distinct exceptional classes. 
It is based on the following version of \cite[Prop.21(i)]{ICERM}.  
It shows that, for $b$ less than and sufficiently close to $m/d$,  the obstruction  $\mu_{\bE, b}(p/q)$ from a quasi-perfect class $\bE$ at its center  is  larger than 
any other that is defined by an exceptional class with which it intersects nonnegatively.

\begin{lemma}\label{lem:pos}  Suppose that the quasi-perfect class $\bE = (d,m,p,q)$  has nonnegative intersection with the exceptional class $\bE'$.  Then $\mu_{\bE, b} (p/q) > \mu_{\bE', b} (p/q)$ for all $b \in \bigl(\frac{m^2-1}{md} , \frac md\bigr]$.
\end{lemma}
\begin{proof}
Because $\bE\cdot \bE' \ge 0$ we have
\begin{align}\label{eq:posit}
dd' - m m' - \bbm\cdot \bbm' \ge 0.
\end{align}
Therefore if $b\le \frac{m}d$ we have
\begin{align*}
\mu_{\bE', b}(a) &= \frac{\bbm'\cdot \bw(p/q)}{d'-m'b}\;\; \le\;\; \frac{\bbm' \cdot \bbm}{q(d'-m'\frac{m}{d})}  \;\;\mbox{ since } b\le \frac{m}d\\
&\le \frac {d (dd'-mm')}{q(dd'-mm')} = \frac dq.
\end{align*}
On the other hand,  because $\bw(a)\cdot \bw(a) = \frac pq$ 
$$
\mu_{\bE, b}(p/q) = \frac {\bbm\cdot\bw(a)}{d-mb}= \frac p{d-mb} >  \frac dq \;\;\mbox{ if } \;\;  pq > d^2-dm b.
$$
Since $pq= d^2 - m^2 +1$ this will hold if also
$$
dm b > d^2-pq = m^2-1, 
$$
i.e. $  b > \frac{m^2-1}{dm}.$  Hence when  $\frac{m^2-1}{dm} \le b <  \frac{m}{d}$  
we have $\mu_{\bE, b}(p/q)>\mu_{\bE', b} (p/q)$.
\end{proof}

\begin{cor}\label{cor:pos}  A quasi-perfect class $\bE$ such that $\bE\cdot \bE'\ge 0$ for all $\bE'\in \Ee\less \bE$ is perfect.
\end{cor}
\begin{proof}   By Lemma~\ref{lem:pos} $\mu_{\bE, b}$ is live  at $p/q$ for all $b$ in an open set. Hence this follows from Lemma~\ref{lem:live1}.\end{proof}

The next result is key to the proof that all the classes considered here are perfect.

\begin{lemma}\label{lem:live2}   Let $\bE= (d,m,p,q,t)$ be a quasi-perfect class such that one of the endpoints of the associated blocked $b$-interval $J_\bE$ is unobstructed and irrational.  Then $\bE$ is perfect.
\end{lemma}
\begin{proof}  We suppose that $m/d>1/3$, leaving the case $m/d<1/3$ to the reader.  

Suppose first that the lower endpoint $b_\infty \in \p J_\bE$ is unobstructed and irrational, and let $z_\infty: = \acc(b_\infty)$.  
Thus $c_{H_{b_\infty}}(z_\infty) = V_{b_\infty}(z_\infty)$ because $z_\infty$ is unobstructed, while 
$V_{b_\infty}(z_\infty) = \mu_{b_\infty}(z_\infty) = \frac {qz_\infty}{d-mb_\infty}$ 
where the second equality holds by \eqref{eq:muEb}, and the first holds by continuity since
$z_\infty$ is the lower endpoint of $I_{\bE} = \acc(J_{\bE})$.
 Therefore, by the scaling property\footnote
{
This says that for any target $X$ and $\la>1$, $c_X(\la z) \le \la c_X(z)$; see \cite[eq;(1.1.1)]{ICERM}.}
we must have
\begin{align}\label{eq:live2}
c_{H_{b_\infty}}(z) = \frac {qz}{d-mb_\infty}=\mu_{\bE,b_\infty}(z),\qquad z\in [z_\infty,p/q].
\end{align}
Thus $\bE$ is live on $(z_\infty,p/q]$ for $b = b_\infty$.  By Lemma~\ref{lem:live1}, it suffices to show that 
$\bE$ is live at $a=p/q$ on some interval $[b_\infty, b_\infty+\eps)$.

If this is not true, there is an exceptional class $\bE' = (d',m',\bbm')$ such that $\mu_{\bE', b_\infty}(p/q) = 
\mu_{\bE, b_\infty}(p/q)$, while $\mu_{\bE', b}(p/q) >  \mu_{\bE, b}(p/q)$ for $b\in 
(b_\infty, b_\infty+\eps)$.  By \cite[eq.(2.1.6)]{ICERM} and \cite[Prop.2.3.2]{ball} a general obstruction function has the form
$$
 \mu_{\bE', b}(z) = \frac{A +Cz}{d'-m'b},
 $$
on the closure of any interval  consisting of points $z$ such that $\ell_{wt}(z)> \ell_{wt}(\bbm')$.\footnote{
 Because in this paper we use two notions of the length of a continued fraction, we here write $\ell_{wt}(p/q)$ (rather than   $\ell(p/q)$) to denote
$\sum_{i=0}^n\ell_i$, where $p/q=[\ell_0;\dots,\ell_n]$, and call it the {\bf weight length} of $p/q$.
Further, $\ell_{wt}(\bbm')$ is simply the number of elements in the tuple $\bbm'$; see Appendix~\ref{app:arith}.}
 Since $\bE'$ is obstructive at $p/q$, we know that $\ell_{wt}(\bbm') \le \ell_{wt}(p/q)$ by 
  \cite[Lem.14.]{ICERM}, while the fact that $\bE$ is obstructive on $[z_\infty, p/q]$ implies that
  $ \ell_{wt}(p/q)\le \ell_{wt}(z)$ for all $z\in [z_\infty, p/q]$.  Therefore the above formula describes $ \mu_{\bE', b}(z)$  on the whole interval $[z_\infty, p/q]$.  
 
 Next notice
 that the slope of the function  $\mu_{\bE', b_\infty}(z) $ for $z\in (p/q-\eps,p/q)$ must agree with that of $\mu_{\bE, b_\infty}(p/q)$: if it were smaller $\mu_{\bE', b_\infty}(z) $ would overwhelm 
$\mu_{\bE, b_\infty}(z)$ for $z< p/q$, while if it were larger the scaling property would be violated when $z>b_\infty$.  Therefore, the constant $A$ above must vanish.
 
 Thus if we now fix $z = p/q$ and $b=b_\infty$ we have
 $$
  \mu_{\bE', b_\infty}(p/q) = \frac{Cp}{q(d'-m'b_\infty)} =   \mu_{\bE, b_\infty}(p/q) = \frac{p}{d-mb_\infty},
  $$
which implies that
 $$
 C(d-mb_\infty) = q(d'-m'b_\infty).
 $$
 Since $b_\infty$ is irrational by Remark~\ref{rmk:recur0}~(i), this is impossible unless $Cm =qm', \ Cd = qd'$.  
 Hence 
 $$
  \mu_{\bE', b}(p/q) = \frac{Cp}{q(d'-m'b)} =    \mu_{\bE, b}(p/q)=  \frac{p}{d-mb},\quad \mbox{ for all }\ b,
  $$
  and in particular for $b\approx b_\infty$.  
  Therefore   $ \mu_{\bE, b}(z) $ is live at $p/q$ for $b\in [b_\infty, b_\infty + \eps)$, so that  $\bE$ is perfect by Lemma~\ref{lem:live1}.
\MS

This proves the lemma when the lower endpoint is unobstructed.  The proof when $(b_\infty, z_\infty)$ is the higher endpoint   is essentially the same.  Again, we first argue that $\mu_{\bE, b_\infty}(z)$ is live on $[p/q,z_\infty)$ and then show that if $\mu_{\bE', b}(p/q)$  is live for  $b\in (b_\infty-\eps, b_\infty]$ the function $z\mapsto 
\mu_{\bE', b_\infty}(z)$ must be constant on $[p/q,z_\infty)$, and hence (by the irrationality of $b_\infty$)  be given by the same formula as $\mu_{\bE, b_\infty}(z)$. It follows that the two obstructions must be equal as $b$ varies.
\end{proof}

\begin{cor}\label{cor:perfATF}  Let $\bE$ be a quasi-perfect class in the complete family $\Cc\Ss^U$ or in one of its images under a symmetry $S^iR^\de$.  Then $\bE$ is perfect.
\end{cor}
\begin{proof}   
The results of \cite{M1} show that 
for every class $\bE$ in  $\Cc\Ss^U$ 
the lower endpoint $b_\bE$ of the corrsponding blocked $b$-interval $J_\bE$ is unobstructed.    Therefore $\bE$ is perfect by Lemma~\ref{lem:live1}.   The
proof of the first claim is completed by
 Lemmas~4.1.2 and~4.1.3 in \cite{MM} that show that  the image by the shift $S$ (resp. $R$)   of a perfect class whose coefficients $(d,m,p,q)$ satisfy \eqref{eq:pqt}  is another perfect class whose coefficients $(d,m,p,q)$ satisfy \eqref{eq:pqt}.  Note that in the case of $R$ we restrict consideration to classes with $p/q>7$; see the discussion in Remark~\ref{rmk:symm0}~(i).
\end{proof}

\begin{cor}\label{cor:perfATF2}  Let $\bE$ be a perfect class  that occurs as a step in a principal pre-staircase.  Then both endpoints of the corresponding blocked $b$-interval $J_\bE$ are unobstructed.
\end{cor}
\begin{proof}  
Our assumptions imply that
for each  endpoint $b_\infty$ of $J_{\bE}$ there is a sequence of perfect classes $\bE_k$ with $p_k/q_k \to z_\infty,\ m_k/d_k\to b_\infty$, where 
$z_\infty = \acc(b_\infty)$.   Moreover, because these classes form a staircase rather than a pre-staircase, the corresponding obstructions are live at the limiting $b$-value $b_\infty$. Then  
  \cite[Lem.27]{ICERM} implies
that  $c_{H_{b_\infty}}(z_\infty)$ is the limit of the obstructions $\mu_{\bE_k, m_k/d_k}(p_k/q_k)$.  Moreover, because $d_k\to \infty$  this limit is $V_{b_\infty}(z_\infty)$.
\end{proof}

We next turn to the proof that the principal pre-staircases are live. 
By \cite[Thm.51]{ICERM}, there are three reasons why a sequence of perfect classes whose steps $p_k/q_k$ converge to $z_\infty$ may not form a staircase at $b_\infty = \acc_\eps^{-1}(z_\infty)$.
\begin{itemize} \item [(i)]  The convergence $m_k/d_k\to b_\infty$ may be so slow that there is no $k_0$ such that the classes $\bE_k,k\ge k_0,$ are obstructive at their centers when $b=b_\infty$.
\item[(ii)] There may be a sequence  of obstructive classes each of which obscures 
a finite number of steps.
\item [(iii)] There may be an {\bf overshadowing class}, i.e. a class  $\bE'$ whose obstruction function $z\mapsto \mu_{\bE',b_\infty}$  goes through the accumulation point $\bigl(z_\infty, V_{b_\infty}(z_\infty)\bigr)$ with sufficiently steep slope to obscure the step corners at $\bigl(p_k/q_k,  p_k/(d_k-m_kb_\infty)\bigr)$ for all $k\ge k_0$.
\end{itemize}
Here we say that the class $\bE'$ {\bf obscures} the step at $p_k/q_k$ given by $\bE_k$  if there is $\eps>0$ such that $
\mu_{\bE', b_\infty}(z) > \mu_{\bE_k, b_\infty}(z) $ either for $z\in (p_k/q_k - \eps, p_k/q_k)$ or for $z\in (p_k/q_k, p_k/q_k+\eps)$. Thus if a prestaircase in $H_{b_\infty}$ is live, infinitely many of its step classes $\bE_k$  are live for $b = b_\infty$ and $z$ in some neighborhood of the step center $p_k/q_k$.

By \cite[Cor.4.2.3]{MM}, problems (i), (ii) never happen for a recursively defined pre-staircase
because there is an upper bound on the degree of a class that could 
obscure a step corner.  On the other hand there could be an overshadowing class.  In particular, 
 recall the following
identity (from \cite[(2.2.5)]{ICERM} or \cite[Lem.2.2.7]{MM})
\begin{align}\label{eq:accform}
V_b(\acc(b))=\frac{1+\acc(b)}{3-b}.
\end{align}
When $b\in (1/5,5/11)$ so that $\acc(b) < 6$,   the function  $z\mapsto\frac {1+z}{3-b}$  is the obstruction 
 from the special exceptional class $\bE'' = 3L - E_0 - 2E_0 - E_{1\dots 6}$  when $z<6$.  Therefore this obstruction always goes through the accumulation point $(\acc(b), V_b(\acc(b))$.
If $\Ss$ ascends, then this class causes no difficulties because $\frac{1+z}{3-b} < V_b(z)$ when $z< \acc(b)$.  However,   we do need to check that this slope of this function is not steep enough to overshadow the steps of a descending pre-staircase $\Ss$ with limit $z_\infty < 6$.
Now, the slope $s_k(\Ss)$ 
 of the line segment from the accumulation point $\bigl(\acc(b), V_b(\acc(b))\bigr)$ to the outer corner $\bigl(\frac{p_k}{q_k}, \frac{p_{k}}{d_{k}-m_{k}b}\bigr)$ of the $k$th step is
\[
s_k(\Ss): = \frac{\frac{p_{k}}{d_{k}-m_{k}b_\infty}-\frac{1+\acc(b_\infty)}{3-b_\infty}} {\frac{p_k}{q_k}-\acc(b_\infty)}.
\]
Therefore $\lim_k s_k(\Ss)> \frac{1}{3-b_\infty}$ exactly if
\begin{align}\label{eq:MDTineq0}
b_\infty \bigl(m_k(p_k+q_k) - p_kq_k\bigr) > d_k(p_k+q_k) - 3p_kq_k, \quad \mbox{ for } k\ge k_0.
\end{align}
If this holds and if the pre-staircase has a blocking class, then as explained in the next result, we can use an arithmetic argument from~\cite{MM} to rule out the existence of an overshadowing class.  The
case of nonrecursive $\Ss$  is more complicated and is treated in \S\ref{ss:uncount}.

\begin{prop}\label{prop:live}  Let $\Ss: = \bigl(\bE_k=(d_k,m_k,p_k,q_k,t_k)\bigr)_{k\ge0}$ be a sequence of 
recursively defined perfect classes such that $m_k/d_k$ decreases with irrational limit $b_\infty>\tfrac13$. Assume $\Ss$ is associated to a blocking class $\bB$ with $t_{\bB} \geq 3$, 
and that, if $\Ss$ descends, the inequality   \eqref{eq:MDTineq0} holds for some $k_0$.
  Then 
$\Ss$ is live, i.e. $H_{b_\infty}$ has a staircase with steps $\bigl(\bE_k\bigr)_{k\ge k_0}$.
The same result holds if the $m_k/d_k$ increase with irrational limit $b_\infty < 1/3$. \end{prop}
\begin{proof}
Since  $\bE_k$ is perfect, it is live at its center $p_k/q_k$ for $b=m_k/d_k$; see Lemma~\ref{lem:basic}.  To prove the claim about $\Ss$, we must show that for large $k,$ $\bE_k$ remains live at the limiting value  $b_\infty.$   As explained above, by \cite[Thm.4.2.1, Cor.4.2.3]{MM} 
it suffices to rule out the existence of an overshadowing class. 
  This obstruction   must be different from $z\mapsto \frac{1+z}{3-b}$; this holds for descending pre-staircases by assumption, and holds for ascending pre-staircases  because the line $z\mapsto \frac{1+z}{3-b}$ cuts through the volume curve from below and so is too steep.  Therefore,
  the three lines given by the obstructions from the overshadowing class, the blocking class, as well as the graph of $z\mapsto \frac{1+z}{3-b_\infty}$ are distinct and
 all go through the accumulation point $(b,\acc(b)) = (b_\infty, z_\infty)$.  But 
$b_\infty, z_\infty$ are both irrational by Remark~\ref{rmk:recur0}~(i).  In \cite[Prop.4.3.7]{MM}  an arithmetic argument  is used   to show  that such an overshadowing class 
cannot exist.
This argument also crucially uses the fact that the $m_k/d_k$ decrease when $b>1/3$ and increase when $b<1/3$, which holds by Corollary~\ref{cor:symm}.
\end{proof}

It remains to prove that the principal pre-staircases satisfy \eqref{eq:MDTineq0}.   Lemma~4.2.7 in \cite{MM} 
shows that this estimate holds for all descending pre-staircases 
associated to the base triples $(S^iR^\de)^{\#}(\Tt^n_*)$, except $\Tt^0_*$.  
Rather than extending that asymptotic argument to cover more cases, we will prove that in most (but not all cases)  the inequality in \eqref{eq:MDTineq0} holds with $k_0 = 0$, since that will be useful in \S\ref{ss:uncount} where we prove Proposition~\ref{prop:liveZ}. 
Although its main steps are given below, the proof also relies on some formulas and estimates that are  established in Appendix B. To simplify our formulas we will denote the sum $p_\bullet + q_\bullet$ by $r_\bullet$.

\begin{lemma}\label{lem:slopeest1} Let $\Ss$ be a  descending principal  pre-staircase with steps 
 $(d_k,m_k,p_k,q_k,t_k, \eps),$ $ k\ge 0$, and write $r_k: = p_k + q_k$.
Then,  the inequality  
 \eqref{eq:MDTineq0} holds for a given $k_0\ge 0$ if, for all $k\ge k_0$, one of the following equivalent conditions holds:
\begin{align}\label{eq:MDTineq1} 
 \frac{t_{k+1}^2 - 8}{t_{k+1}r_{k+1}}&  > \frac{t_{k}^2 - 8}{t_k r_k}, \quad\mbox { or } \\ 
 \label{eq:MDTineq2}
\frac{r_{k+1}}{t_k} - \frac {r_{k}}{t_{k+1}}& > \tfrac 18 \bigl( t_k r_{k+1} -  t_{k+1} r_{k}\bigr).
\end{align}
\end{lemma}
 \begin{proof}   Notice first that \eqref{eq:MDTineq0} holds exactly  if 
\begin{align}\label{eq:MDTineq1a} 
\bigl(m_k(p_k+q_k) - p_kq_k\bigr) b_\infty >  d_k(p_k+q_k) - 3p_kq_k.
\end{align}
If we substitute for $m_k, d_k$ in terms of $p_k, q_k$ using \eqref{eq:pqt} we obtain
$$
8d_k(p_k+q_k) - 24 p_kq_k = 3(t_k^2 - 8) + \eps t_k r_k, \quad
8m_k(p_k+q_k) - 8p_kq_k = (t_k^2 - 8) + 3\eps t_k r_k.
$$
Since, by Lemma~\ref{lem:prelim0}~(iii), $r_k\ge  3t_k\ge 3$ when $\eps=-1$, 
  these expressions are negative  exactly when $\eps= -1$.  
Thus the condition in  \eqref{eq:MDTineq1a} is equivalent to
\begin{align}\label{eq:binfty}
{b_\infty}& > \frac{3(t_k^2 - 8) +  t_k r_k} {(t_k^2 - 8) + 3 t_k r_k} = 3 -\frac{8t_kr_k}{(t_k^2 - 8) + 3 t_k r_k}\quad \mbox{ if } \eps = 1,\\  \notag
{b_\infty}& < \frac{t_k r_k  -3 (t_k^2 - 8)} {3 t_k r_k  -(t_k^2 - 8)} = 3 - 
\frac {8t_kr_k} {3 t_k r_k  -(t_k^2 - 8)} \quad \mbox{ if } \eps = -1.
\end{align}
As $k\to \infty$, the ratio $\frac{d_k(p_k+q_k) - 3p_kq_k}{m_k(p_k+q_k) - p_kq_k}$ converges  to
 $$
 \frac{D(P+Q) - 3PQ}{M(P+Q) - PQ} =  \frac{D(3D-M) - 3(D^2-M^2)}{M(3D-M) -  (D^2-M^2)}  = \frac MD = b_\infty,
 $$
 where we use the identities in \eqref{eq:DiophPQ}.
  Therefore, the result will hold if we prove that  for $k\ge k_0$ the sequence on RHS of \eqref{eq:binfty}  is increasing  when $\eps=1$ and decreasing when $\eps = -1$.

Thus, when $\eps=1$, we need
 the sequence $\frac{t_kr_k}{(t_k^2 - 8) + 3 t_k r_k}$ to decrease with $k$, or equivalently 
the sequence $\frac{t_k^2-8}{t_kr_k}$ to increase with $k$. 
Similarly, when $\eps=-1$ the condition still is that 
 $\frac{t_k^2-8}{t_kr_k}$ should increase with $k$. 
 This proves the claim about condition~\eqref{eq:MDTineq1}.  Finally, the inequality \eqref{eq:MDTineq2} 
is just a rearrangement of~\eqref{eq:MDTineq1}. 
  \end{proof}

 \begin{lemma}\label{lem:slopeest2} Let $\Tt = (\bE_\la, \bE_\mu, \bE_\rho)$ be any triple.
 Then
   \begin{itemlist}
     \item[{\rm (i)}] 
If  \eqref{eq:MDTineq2} holds for the first two terms $\bE_\rho, \bE_\mu$  in 
  the associated descending pre-staircase $\Ss^\Tt_u$, then it also holds for these terms 
in $\Ss^{x\Tt}_u$.
  \item[{\rm (ii)}]  
If  \eqref{eq:MDTineq1} holds for the  first two terms $\bE_\rho, \bE_\mu$ in 
  the associated descending pre-staircase $\Ss^\Tt_u$, then it also holds for these terms 
 in $\Ss^{y\Tt}_u$. 
    \end{itemlist}
\end{lemma}

\begin{proof}    We must show that if the inequality
\begin{align}\label{eq:MDTineq3}
 \frac{r_\mu}{t_\rho} -  \frac{r_\rho}{t_\mu} > \tfrac 18 (t_\rho r_\mu- t_\mu r_\rho)
\end{align}
holds for $\Tt$ then it also holds for $x\Tt$.  Under this mutation
the RHS remains the same.  Hence it suffices  to check that
 $$
 \frac{t_\la r_\mu - r_\rho}{t_\mu} -  \frac{r_\mu}{t_\la t_\mu - t_\rho} >  \frac{r_\mu}{t_\rho} -  \frac{r_\rho}{t_\mu}.
 $$
When we multiply throughout by  $t_\mu (t_\la t_\mu - t_\rho) t_\rho$, the terms that are products of three factors cancel, and after dividing the remaining terms by $t_\la$, we obtain the inequality
$$
t_\la t_\mu t_\rho  r_\mu - t_\mu t_\rho r_\rho - t_\rho^2 r_\mu > t_\mu^2 r_\mu- t_\mu t_\rho r_\rho.
$$
Now cancel the term $ t_\mu t_\rho r_\rho$ from both sides and divide by $r_\mu$ to obtain the inequality in Lemma~\ref{lem:prelim0}~(v).  This proves (i).

Now consider (ii).  We must show that 
the inequality 
$$
\frac{t_{\mu}^2 - 8}{t_{\mu}r_{\mu}}> \frac{t_{\rho}^2 - 8}{t_\rho r_\rho}
$$
persists under  a $y$-mutation.
Since the RHS remains  unchanged, it suffices to show that the LHS increases, i.e. 
$$
\frac{(t_\rho t_{\mu} - t_\la)^2 - 8}{(t_\rho t_{\mu} - t_\la)(t_\rho r_{\mu} - r_\la)} > \frac{t_{\mu}^2 - 8}{t_{\mu}r_{\mu}}.
$$
After simplifying, this reduces to the inequality
$$
t_\mu(t_\rho t_{\mu} - t_\la)(t_\la r_\mu - t_\mu r_\la) < 8\Bigl(t_\rho^2 t_\mu r_\mu - (t_\la r_\mu + t_\mu r_\la) t_\rho + t_\la r_\la - t_\mu r_\mu\Bigr).
$$
Simplify and increase the LHS of this inequality by using the fact that  $$
t_\la r_\mu - t_\mu r_\la = 8 q_\rho < \tfrac 43 r_\rho.
$$
Here the equality holds by Lemma~\ref{lem:prelim0}~(iv), while
 the inequality holds because 
 $p_\rho/q_\rho > a_{\min}> 5$ so that $q_\rho < r_\rho/6$
Next simplify and decrease the RHS by ignoring the term $+t_\la r_\la$ and replacing
$- (t_\la r_\mu + t_\mu r_\la) t_\rho$ by  $- 2t_\mu r_\mu t_\rho$, which is smaller because $r_\la < r_\mu, t_\la < t_\mu$  (see Lemma~\ref{lem:prelim0}~(ii)).   
These maneuvers show that, after cancelling the common factor of $t_\mu$,  
it suffices to prove
$$
\tfrac 16  r_\rho (t_\rho t_{\mu} - t_\la) <(t_\rho^2 - 2 t_\rho -1) r_\mu.
$$
We now show that 
this inequality holds even without the term $-\tfrac 16  r_\rho t_\la$ on LHS.  Indeed, after omitting this term and then rearranging, 
we find that it suffices to prove
$$
\tfrac 16 \frac{r_\rho}{t_\rho} \Bigl(1 - \frac 2{t_\rho} - \frac 1{t_\rho^2}\Bigr)^{-1} < \frac{r_\mu}{t_\mu}.
$$
But because $t_\rho\ge 3$,  $1 - \frac 2{t_\rho} - \frac 1{t_\rho^2}> \frac 29$ so that it suffices to check that
 $ \frac 34 \frac{r_\rho}{t_\rho}\le  \frac{r_\mu}{t_\mu}$.  But this holds by Corollary~\ref{cor:prelim0}. 
\end{proof}

Finally, we claim in Lemma~\ref{lem:slopeest3}  that the inequality \eqref{eq:MDTineq1} holds for the base cases.  Its statement is
complicated by the fact that  \eqref{eq:MDTineq1} does {\it not} hold for $\Tt_*^0$ or its image by $R$.  We   give the proof now in the \lq easy' cases that do not involve these triples; the proof is completed in Lemma~\ref{lem:basecases}.

Recall from
 \eqref{eq:SsUparam} that the entries $(p,q,t)$ in each of the  classes in the base triple $\Tt^n_*, n\ge 0,$ are
\begin{align}\label{eq:baU} 
  & \qquad  \qquad\qquad \bE_\la:\ (p,q,t) : =  (2n+6,1,2n+3), \\ \notag
 & \bE_\mu:\ (4n^2+22n+29,2n+4,4n^2+16n+13), \quad    \bE_\rho:\ 
    (2n+8,1,2n+5).
\end{align}
 Since $R(p/q) = (6p-35q)/(p-6q)$ and $R$ fixes $t$ while interchanging $\bE_\la, \bE_\rho$, the corresponding entries for the triples $R^{\sharp}(\Tt^n_*), n\ge 0,$ are
\begin{align}\label{eq:baL} 
  &\qquad\qquad  \qquad\qquad    \bE_\la:\ (p,q,t) : (12n+13, 2n+2,2n+5), \\ \notag
  &\bE_\mu:\  (24n^2+62n+34,4n^2+ 10 n+5,4n^2+16n+13),\   \bE_\rho:\ 
    (12n+1,2n,2n+3).
\end{align}

 \begin{lemma}\label{lem:slopeest3} \begin{itemlist}\item[{\rm (i)}]   Let $\Tt$ be any triple of the form $(S^iR^\de)^{\sharp}(\Tt^n_*), i\ge 0, \de\in\{0,1\},$  or one of the form $yx^k\Tt^0_*$ for all $k\ge 0$.  Then,
 provided that $\Tt \ne \Tt^0_*, R^{\sharp}(\Tt^0_*)$, 
the inequality \eqref{eq:MDTineq1} holds for the first two terms in the associated descending pre-staircase $\Ss^\Tt_u$.

\item[{\rm (ii)}] The inequality \eqref{eq:MDTineq1}  holds for the second and third terms in the descending pre-staircases associated to $y^kR^\sharp(\Tt^0_*)$, $k\geq0$.
 \end{itemlist}
\end{lemma} 

\begin{proof} Consider (i).   
When $\Tt = \Tt^n_*$ or $R^{\sharp}(\Tt^n_*), n> 0$, we can check that inequality  \eqref{eq:MDTineq1} holds directly from  the formulas in
\eqref{eq:baU},~\eqref{eq:baL} above.   
Next note that because 
 $S(p/q) = (6p-q)/p$, the entries in $S^{\sharp}(\Tt^0_*)$ and $(SR)^{\sharp}(\Tt^0_*)$ are
\begin{align}\label{eq:baseS0} 
\mbox{ for } S^{\sharp}(\Tt^0_*): \;\; &  \bE_\la= (35,6, 3), \quad \bE_\mu=(170, 29,13),\quad 
  \bE_\rho=(47,8,5);\\ \notag
    \mbox{ for } (SR)^{\sharp}(\Tt^0_*): \;\; &  \bE_\la= (76,13, 5), \quad \bE_\mu=(165,34,13),\quad 
  \bE_\rho=(6,1,3).
\end{align}
It is again easy to check  that 
\eqref{eq:MDTineq1} holds for $S^\sharp(\Tt^0_*)$ and $(SR)^\sharp(\Tt^0_*)$. 

We next claim that
if \eqref{eq:MDTineq1} holds for a triple $\Tt$ then it also holds for $S^{\sharp}(\Tt)$.  
Since  $S$ fixes the parameter $t$, by rearranging \eqref{eq:MDTineq1}  so that
the $t$ terms are on one side and the $r$ terms on the other, one finds that $S$ preserves \eqref{eq:MDTineq1}  provided that it decrease the ratio $r_\mu/r_\rho$.
 Thus, we need
$\frac{7p_\mu - q_\mu}{7p_\rho - q_\rho} < \frac{p_\mu + q_\mu}{p_\rho + q_\rho}$, which holds because
$\frac{p_\mu}{q_\mu} < \frac {p_\rho}{q_\rho}$. 

This completes the proof of (i)  except for the claims about the $yx^k$-mutations of the exceptional triple
$\Tt^0_*$.  For these details and the proof of (ii), see  Lemma~\ref{lem:basecases}. 
\end{proof}

\begin{cor}\label{cor:live1}  Every descending principal pre-staircase $\Ss$ except $\Ss^U_{u,0}$ and $\Ss^{y^kR^\sharp(\Tt^0_*)}_u$ satisfies 
the inequality \eqref{eq:MDTineq0} with $k_0=0$, and the pre-staircases $\Ss^{y^kR^\sharp(\Tt^0_*)}_u$ satisfy \eqref{eq:MDTineq0} with $k_0=1$. Hence, every principal pre-staircase is live, and hence is a staircase.
\end{cor}
\begin{proof}    We proved that $\Ss^U_{u,0}$ is live in \cite[Ex.70]{ICERM}.  All other descending pre-staircases are associated to some triple which is a mutation of one of the basic triples listed in Lemma~\ref{lem:slopeest3}. Therefore, it follows from that result together with Lemmas~\ref{lem:slopeest1} and ~\ref{lem:slopeest2}
that the inequality \eqref{eq:MDTineq0} holds (with $k_0=0$ for all triples except for $y^kR^\sharp(\Tt^0_*)$, where $k_0=1$).  This proves the first claim.  It now follows from Proposition~\ref{prop:live} that every principal pre-staircase (both ascending and descending) is live.
 \end{proof}

\subsection{Uncountably many staircases}\label{ss:uncount} 

We now prove Proposition \ref{prop:liveZ}. We first discuss the ascending pre-staircases, which turn out to be relatively easy to deal with.   As explained in Definition~\ref{def:newstair}, we denote by $\Ss^{\pm}_{\al_\infty}$ any (ascending or descending) pre-staircase  with limit point at $\al_\infty \in Z$, for simplicity omitting the decorations $n, i, \de$ that specify more precisely where it is.

In this case, the key to our argument is the following lemma, that explains the influence of the ratio $m/d$ on the behavior of the corresponding obstruction.  
This  result applies to any pair of obstructive classes $\bE, \bE'$. These have the form  $\bE: = dL - mE_0 - \sum m_i E_i$ (abbreviated as $(d,m,\bbm)$) where $$
3d - m - \sum_{i=1}^N m_i = 1, \qquad d^2-m^2 - \sum_{i=1}^N m_i^2 = -1.
$$
  The corresponding obstruction function $\mu_{\bE,b}$ is piecewise linear, with the form $z\mapsto \frac{A + Cz}{d-mb}$ in any interval consisting of points $z$ with $\ell_{wt}(z) > \sum_{i=1}^N m_i$; see \eqref{eq:WCFlength} and \cite[\S2.1]{ICERM}.   Moreover, if $z_0$ is fixed, there is a constant $A_0 = A(z_0)$ such that  as a function of $b$ the obstruction 
$\mu_{\bE,b}(z_0)$ has the form $b\mapsto  A_0/(d-mb)$.

\begin{lemma}\label{lem:md}  Let $\bE: =  (d,m,\bbm), \ \bE' = (d',m',\bbm')$ be obstructive classes as above.
Then:

\begin{itemlist}\item[{\rm (i)}]
 If  $\mu_{\bE,b_0}(z_0) = \mu_{\bE',b_0}(z_0) \ge V_{b_0}(z_0)$ for some $b_0, z_0$,   then
\begin{align}\label{eq:deriv}
m'/d' <m/d \;\Longleftrightarrow \frac{\p}{\p b}\Big|_{b=b_0} \mu_{\bE',b}(z_0) < \frac{\p}{\p b}\Big|_{b=b_0} \mu_{\bE,b}(z_0). 
\end{align}

\item[{\rm (ii)}]
If $m'/d' <m/d$ and $\mu_{\bE',b_0}(z_0) < \mu_{\bE,b_0}(z_0)$ for some $b_0$, then
$\mu_{\bE',b}(z_0) < \mu_{\bE,b}(z_0)$ for all $b> b_0$.
\end{itemlist}
\end{lemma}
\begin{proof}  As explained above, we may write
$$
\mu_{\bE,b}(z_0) = \frac{A_0}{d-mb},\qquad \mu_{\bE',b}(z_0) = \frac{A'_0}{d'-m'b}.
$$
But 
$$
\frac{\p}{\p b}\Big|_{b=b_0} \mu_{\bE,b}(z_0) = \frac{A_0}{d-mb_0} \ \frac m {d-mb_0} = \frac 1{d/m - b_0} \mu_{\bE,b_0}(z_0).
$$
Note here that $d/m>1$ while $b_0< 1$.  Therefore $\frac 1{d/m - b_0} > \frac 1{d'/m' - b_0} $ iff $d/m < d'/m'$, which happens iff $m'/d' < m/d$.  This proves (i).

The  calculation above also implies that if   $\mu_{\bE,b_0}(z_0)> \mu_{\bE',b_0}(z_0)$ and $m/d> m'/d'$ then 
$\mu_{\bE,b}(z_0)$ increases faster than $\mu_{\bE',b}(z_0)$ as $b$ increases.  Hence (ii) also holds.
\end{proof}

\begin{lemma}\label{lem:5T}  Let $\Tt$ be any triple 
with center class $\bE_\mu$.  Then: 
\begin{itemlist}\item [{\rm (i)}]  If $b>1/3$ the obstruction $\mu_{\bE_\mu, b}(p_\mu/q_\mu)$ is live for all $b \in [\p^-(J_{\bE_\mu}), m_\mu/d_\mu]$, and hence for all $b\in  [\p^-(J_{\bE_\mu}), \p^-(J_{\bE_\rho})]$. 
\item [{\rm (ii)}]  If $b<1/3$ the obstruction $\mu_{\bE_\mu, b}(p_\mu/q_\mu)$ is live for all $b \in  [m_\mu/d_\mu, \p^+(J_{\bE_\mu})]$, and hence for all $b\in  [\p^+(J_{\bE_\rho}), \p^+(J_{\bE_\mu})]$.
\end{itemlist}
\end{lemma}
\begin{proof}  First suppose that $b>1/3$. 
Because $c_{H_b}$ is unobstructed for $b\in \p J_{\bE_\mu}$ by Corollary~\ref{cor:perfATF}, we can apply \cite[Prop.~42]{ICERM} to show that the obstruction $\mu_{\bE_\mu, b}(p_\mu/q_\mu)$  is live
for $b$ in the blocked interval $J_{\bE_\mu}$.    It is also live at the lower endpoint $\p^- J_{\bE_\mu}$ by \eqref{eq:live2}. Next note that $\mu_{\bE_\mu, b}(p_\mu/q_\mu)$ is live 
 for $b = m_\mu/d_\mu$  by \cite[Prop.21]{ICERM}, where $m_\mu/d_\mu> \p^+(J_{\bE_\mu})$ by \cite[Lem.2.2.11]{MM}.  
If  it were not live at some $b_0\in [\p^+(J_{\bE_\mu}), m_\mu/d_\mu\bigr)$, there would some exceptional class $\bE'$ with degree coordinates $(d',m')$ such that $\mu_{\bE', b_0}(p_\mu/q_\mu) > \mu_{\bE_\mu, b_0}(p_\mu/q_\mu)$.   Therefore there would have to be $b_1<b_2$ with $\p^+(J_{\bE_\mu}) < b_1< b_0 < b_2 < m_\mu/d_\mu$ at which the two obstructions are equal, with $\mu_{\bE', b}(p_\mu/q_\mu)$ growing faster than $ \mu_{\bE_\mu, b}(p_\mu/q_\mu)$ at $b=b_0$ and slower at $b= b_2$.    But this contradicts 
Lemma~\ref{lem:md}.   Now note that $m_\mu/d_\mu> \p^-(J_{\bE_\rho})$ because $\bE_\mu$ is a step in a pre-staircase for $b = \p^-(J_{\bE_\rho})$ with decreasing $m/d$ values.  This completes the  proof.

A similar argument follows for $b<1/3$, where the order of $b$ is reversed.    In particular, the interval $J_{\bE_\rho}$ lies to the left of $J_{\bE_\mu}$, and the sequence  $m_k/d_k$ increases for all principal pre-staircases by Lemma~\ref{lem:2}~(iii) with limit  $\p^+(J_{\bE_\rho})$.  Therefore we always have
 $m_\mu/d_\mu < \p^+(J_{\bE_\rho})$. Further details are left to the reader.
\end{proof}  

\begin{cor}\label{cor:5T}  All the steps in each ascending pre-staircase  $\Ss^+_{\al_\infty}$ are live at their centers when $b$ equals
 the limiting value
$b_{\al_\infty}$.
\end{cor}
\begin{proof}   Let $\bE$ be a step in some  ascending pre-staircase $\Ss^+_{\al_\infty}$, and denote by $\Tt$ the unique triple with middle step $\bE_\mu = \bE$.  
By construction the $z$-limit point of $\Ss^+_{\al_\infty}$ is at most $\p^-(I_{\bE_\rho})$.  Therefore if $b>1/3$  the corresponding $b$-value  is $\le \p^-(J_{\bE_\rho})$, and $\mu_{\bE,b}(p_\mu/q_\mu)$ is live at $b$ by Lemma~\ref{lem:5T}~(i). On the other hand if $b<1/3$ then the $b$-value corresponding to the 
 $z$-limit point of $\Ss^+_{\al_\infty}$ is  $\ge  \p^+(I_{\bE_\rho})$, and the conclusion now follows from 
 Lemma~\ref{lem:5T}~(ii).
\end{proof}

The  descending pre-staircases $\Ss^-_{\al_\infty}$  present a  more complicated problem.
According to the discussion before Proposition~\ref{prop:live}, 
the first step in the proof is to show that the steps are obstructive at the limiting $b$-value
$b_{\al_\infty}$.  
 This is a consequence of the next lemma that shows that each class  $\bE_\mu$ is  obstructive on the whole of the $b$-interval    
\begin{align}\label{eq:JjT} 
J_\Tt: = \acc^{-1}_\eps(I_\Tt), \quad\mbox{ where }\ I_\Tt = \bigl(\p^+(I_{\bE_\la}), \ \p^-(I_{\bE_\rho})\bigr).
\end{align}
 that lies between $J_{\bE_\la}$ and 
$J_{\bE_\rho}$.

\begin{lemma}\label{lemcor:1T} Each step $\bE$ of each descending pre-staircase $\Ss^-_{\al_\infty}$ for ${\al_\infty}\in Z$ is obstructive at its center at the limiting $b$-value $
b_{\al_\infty}$.  Moreover, the obstruction $z\mapsto \frac{1+z}{3-b}$ does not overshadow  any step in $\Ss^-_{\al_\infty}$.
\end{lemma}
\begin{proof}  
Since $  \frac{1+z}{3-b}> V_b(z)$ for $z>\acc(b)$, the
 first claim follows immediately from the second. To show the second, first suppose that  $b>1/3$.  We know from Lemma~\ref{lem:slopeest3} that the inequality \eqref{eq:MDTineq0} holds for all steps (except possibly the first) of all descending principal staircases. Each step $\bE$ in $\Ss^-_{\al_\infty}$ is the center step $\bE_\mu$ of a unique triple $\Tt$.  By definition, the next step in $\Ss^-_{\al_\infty}$ may be written $\bE_{y^kx^i\mu}$, where $i,k\ge 1$.  
Thus its center lies strictly to the right of the lower endpoint $\p^+(I_{\bE_\la})$ of $I_\Tt$, as do the centers of all subsequent steps. Hence the inequality \eqref{eq:MDTineq0}, which holds for $b_\infty=\p^+(I_{\bE_\la})$, continues to hold for $b_{\al_\infty}>b_\infty$.

When $b<1/3$ both sides of \eqref{eq:MDTineq0} are negative, and a corresponding argument applies.
\end{proof}

The second step in the proof is to find, for each descending pre-staircase $\Ss^-_{\al_\infty}$,  a uniform bound for the degree of a class that could obstruct any one of its steps  $\bE_k, k\ge 0$. 
We will treat the case $b>1/3$ in detail; the changes needed for the case $b<1/3$ are discussed in Remark~\ref{rmk:b<13}.  In the following  we use the notion of the {\bf level} of a step that was defined in Lemma~\ref{lem:label}.

\begin{lemma}\label{lem:2T} \begin{itemlist}\item[{\rm (i)}]  Let
$\bE = (d,m,p,q)$ be a perfect class such that for some $0<x < b_\infty < m/d$ we have
$\frac md > x(1+\frac 1{d^2})$ and
\begin{align}\label{eq:rsb}
A(m,d,x): = \frac{m(m-xd) - 1}{d(m-xd) - x} < b_\infty < m/d.
\end{align}
Then if $\bE'$ is any other perfect class with degree $d'>  1/(b_\infty - x)$, we have 
$\mu_{\bE, b_\infty}(p/q) > \mu_{\bE', b_\infty}(p/q) $.

\item[{\rm (ii)}]    If $x_0<b_\infty < y': = m'/d'<y: = m/d< x_1$ then
\begin{align}\label{eq:rsb0}
A(m,d,x_0) < A(m',d',x_0)
\end{align}
provided that, with $f(y) = \frac{1-x_0y}{y-x_0}, y>x_0$, we have
\begin{align}\label{eq:rsb1}
& md'-m'd  < \Bigl(\frac{d'}{d}  - \frac{d}{d'} \Bigr) f(x_1), \quad\mbox{ and } \\ \label{eq:rsb2}
& \frac{d'}{d} \bigl(f(y') - f(y)\bigr) < \Bigl(\frac{d'}{d}  - \frac{d}{d'} \Bigr)\bigl(f(y') - f(x_1)\bigr) .
\end{align}

\item[{\rm (iii)}]  
Let $\Ss$ be   any descending principal  staircase in $(S^iR^\de)^{\sharp}\bigl(\Cc\Ss^U\cap [2n+6, 2n+8]\bigr)$, $ n\ge 0,$ with recursion parameter $t_\la>3$ and with $i+\de$  even so that $b>1/3$. Denote by
$b_{\min}, b_{\max}$ the infimum (resp. supremum) of the $b$-values for these staircases.  Then there are constants $x_0< b_{\min}< b_{\max} < x_1$ and a level $\ell$ such that 
 conditions  \eqref{eq:rsb1},  \eqref{eq:rsb2} hold whenever
 $(d,m), (d',m')$ are the degree coordinates of a pair of adjacent steps in $\Ss$ at level $\ge \ell$.
\end{itemlist}
\end{lemma}
\begin{proof} It is shown in \cite[Prop.21(iii)]{ICERM} if this inequality for $x= r/s$ is satisfied any class $\bE'$ such that
 $\mu_{\bE, b_\infty}(p/q) < \mu_{\bE', b_\infty}(p/q) $ must have $m'/d'< r/s$.  Since in this case 
$\mu_{\bE', b_\infty}$ is 
obstructive we must have $|b_\infty d'-m'| < 1$ by \cite[Lem.15]{ICERM}, which readily gives the bound 
 on $d'$.  This proves (i).
 
 The inequality \eqref{eq:rsb0} states that even though $m/d$ decreases to $b_\infty$ along the staircase, the quantity $A(m,d,x_0)$ 
 (which also has limit $b_\infty$) increases, a fact that is key to the argument in Corollary~\ref{cor:2T}.
Now, it is straightforward to check that
\eqref{eq:rsb0} is equivalent to the inequality
$$
md'-m'd < \frac{d'-x_0m'}{m-x_0d} - \frac{d-x_0m}{m'-x_0d'},\quad\mbox{ when } \ 0<md'-m'd.
$$ 
Since  $x_0 < b_\infty < y': = m'/d' <  y:=m/d < x_1$ we have
  \begin{align*}
  \frac{d'-x_0m'}{m-x_0d} - \frac{d-x_0m}{m'-x_0d'} & = \frac{d'}{d} \frac{1-x_0y'}{y-x_0} - \frac{d}{d'} \frac{1-x_0y}{y'-x_0}\\
&  >  \frac{d'}{d} \frac{1-x_0y}{y-x_0}  -  \frac{d}{d'} \frac{1-x_0y'}{y'-x_0}.
\end{align*}
Therefore, since $f(y)$ is decreasing, we have
\begin{align*}
  \frac{d'-x_0m'}{m-x_0d} - \frac{d-x_0m}{m'-x_0d'}& > \Bigl(\frac{d'}{d}- \frac{d}{d'}\Bigr) f(y')   - \frac{d'}{d}\bigl(f(y') - f(y)\bigr)\\
& >  \Bigl(\frac{d'}{d}- \frac{d}{d'}\Bigr)f(x_1) \quad \mbox{by } \eqref{eq:rsb2}.
\end{align*}
Therefore, 
 \eqref{eq:rsb0} holds if, in addition, \eqref{eq:rsb1}  holds.
This proves (ii).

Now consider (iii).  For simplicity, we begin by considering the family
$\Cc\Ss^U\cap[6,8]$.  
 Since each  principal staircase is recursively defined, as was observed in \cite[Cor.4.2.3]{MM} there always is a constant $x< b_{\min}$ such that \eqref{eq:rsb} holds for all classes in that particular staircase.  
 As we shall see in Corollary~\ref{cor:2T}, the existence of such a constant $x$ is enough to show that the pre-staircase is live unless it is overshadowed by a class of low degree.\footnote
 {
Our earlier  proof that the principal pre-staircases are live used a  different argument.}  
The difficulty is that we want to find a single constant  that applies to  all  staircases in this family.  
It turns out that the descending staircase with blocking class $\bB^U_0$ 
(and recursion parameter $3$) is exceptional and that we get better estimates if we exclude it.    
Thus we will  find constants $x'_0, x_1, \ell$  such that
conditions~\eqref{eq:rsb1}, ~\eqref{eq:rsb2} hold for any pair of adjacent steps at level $\ge \ell$ in 
 $\Ss = \Ss^\Tt_u$  and  where $\Tt$ has
$p_\la/q_\la >7$, and then take $x_0= \max(x'_0,x''_0)$ where $x''_0$ is the lower constant for 
the exceptional staircase. 
Note that the value of  $ p_\la$  is relevant to the question at hand because, by Lemma~\ref{lem:prelim0}~(iii), the
 ratio $md'-m'd = m_\rho d_\mu -  m_\mu d_\rho > 0$ is fixed for all adjacent pairs of steps and equals $p_\la$.

We first claim that for any $x_1 > b_{\max}$ there is $\ell = \ell(x_1)$ such that \eqref{eq:rsb2} holds for all adjacent steps at levels $\ge \ell$.   This holds because 
\begin{itemlist}\item [-] the ratio $d'/d$ is $\ge t_\la - 1$, where $t_\la\ge 13$ is the recursion parameter of the staircase and $f$ decreases, so that it suffices to show $$
f(y')-f(y) < (1-(d/d')^2)\bigl(f(y') - f(x_1)\bigr) \le 143/144\bigl(f(b_\infty) - f(x_1)\bigr): = C
$$
 whenever $y, y'$ are $m/d$-values for two successive steps in the staircase;
\item[-]  for some $c_1,c_2>0$ we have $-c_1 \le f'(y)\le -c_2<0$ when $y\in [b_\infty, x_1]$ so that
$|f(y')-f(y)| < c_1|y'-y|$;
\item[-] we saw in \eqref{eq:distell} that adjacent steps $p/q$ at level $\ge \ell$ are less than a distance $\frac1{2^\ell}$ apart; a similar argument applies to the ratios $m/d$, where we use the formula in Lemma~\ref{lem:prelim0}~(iii) instead of the adjacency relation $|p/q-p'/q'|= t''/qq'$.
\end{itemlist}
Further details are left to the reader.

Next observe that by Lemma~\ref{lem:prelim2} 
$$ 
\frac {d_\mu}{d_\rho} - \frac {d_\rho}{d_\mu} > t_\la - 1.
$$ 
Therefore, by (ii),  
\eqref{eq:rsb0} will hold for a given $x_0<b_\infty$ if we also choose $x_1$ so that 
\begin{align}\label{eq:dtineq1}
p_\la/(t_\la-1) < \frac{1-x_0x_1}{x_1 - x_0}.
\end{align}
Now, for every class $\bE_\la$ under present consideration we have $7< p_\la/q_\la < 8$, so that $$
t_\la^2/p_\la^2 = 1 - 6 q_\la/p_\la + (q_\la/p_\la)^2 + 8/p_\la^2 \ge  1/7.
$$
Therefore, because we also have $t_\la \ge 5$,   
$$
p_\la/(t_\la - 1) = \frac{p_\la}{t_\la} (1- 1/t_\la)^{-1}   < \sqrt 7 \ \frac 54 < 4. 
$$  
On the other hand we know by \eqref{eq:accn} that $1/2 <\p^+(J_{\bB^U_0}) \le  b_\infty < 2/3$, where $2/3 = m/d$ for $\bB^U_0$. 
Moreover, all step  classes except $\bB^U_1$ have $m/d< 2/3$ by 
Lemma~\ref{lem:accmd}.  But if $x_0 = 1/2, x_1 = 2/3$ we have 
$\frac{1-x_0x_1}{x_1 - x_0} = 4$.  Therefore in this case  \eqref{eq:rsb1} holds with $x_0 = 1/2, x_1 = 2/3$.

 To establish (iii) in the general case (still with $b>1/3$), 
  we first need to choose suitable upper and lower bounds  $x_0,x_1$ 
for $b_\infty$,
 which is done in Corollary~\ref{cor:qtrip2}. Next notice that for any such family we always have $d'/d > t_\la$, and again we can assume $t_\la\ge 13$ by omitting the staircase with smallest $z$-accumulation point.  
 Hence 
  for given $x_0,x_1$ there always is a level $\ell$ such that  \eqref{eq:rsb2} holds for all staircases in the family.  We then show in Lemma~\ref{lem:rsforS} 
  that  
 \eqref{eq:dtineq1} holds for $x_0,x_1$ as chosen above.
  The argument given above then extends to complete the proof of (iii).
\end{proof}

\begin{rmk}\label{rmk:b<13}\rm  Lemma~\ref{lem:2T} extends to the  case $b< 1/3$ as follows.
If a staircase has $b<1/3$ then the ratios $m/d$ increase to $b_\infty$ and we have the following analogs to the claims in this lemma.
\begin{itemlist}\item[{\rm (i$^\prime$)}] If
$\bE = (d,m,p,q)$ is a perfect class such that for some $0<m/d < b_\infty <x$ we have
\begin{align*}
m/d < b_\infty < \frac{m(xd-m) + 1}{d(xd-m) + x} = : A'(m,d,x),
\end{align*}
then for any other perfect class $\bE'$ with degree $d'>  1/(x-b_\infty)$, we have 
$\mu_{\bE, b_\infty}(p/q) > \mu_{\bE', b_\infty}(p/q) $.
\item[{\rm (ii$^\prime$)}] If
 $x_0<y: = m/d<y': = m'/d' < b_\infty < x_1$ then $A(m',d', x_0) < A(m,d,x_0)$ 
provided that, with $f(y) = \frac{1-x_0y}{y-x_0}$ where $y>x_0$ as before,
\begin{align}\label{eq:md<13}
m'd-md' & < \Bigl(\frac{d'}{d}  - \frac{d}{d'} \Bigr) f(x_1),\\ \notag
\frac{d'}{d} \bigl(f(y) - f(y')\bigr)& < \Bigl(\frac{d'}{d}  - \frac{d}{d'} \Bigr)\bigl(f(y') - f(x_1)\bigr).
\end{align}
\item[{\rm (iii$^\prime$)}]
There are constants $x_0<b_{\min} < b_{\max}<x_1< 1/3$  (where $b_{\min},b_{\max}$ are the minimum, resp. maximum, of the $b$-values for these staircases) and a level $\ell$
such that, if $\Ss$ is  any descending principal  staircase in $(S^iR^\de)^{\sharp}\bigl(\Cc\Ss^U\cap [2n+6, 2n+8]\bigr)$, $ n\ge 0,$ where $t_\la > 5$ and $i+\de$ is odd so that $b<1/3$, 
the above
inequalities  hold when
 $(d,m), (d',m')$ are the degree coordinates of any two adjacent steps in $\Ss$ at level $\ge k$.
 \end{itemlist}
 The proofs of  (i$^\prime$) and (ii$^\prime$) are analogous to those in the case $b>1/3$ and are left to the reader. 
As for (iii$^\prime$), for fixed $x_0,x_1$ one can always choose $\ell$ so that 
the second inequality above holds.   Also, just as before, the inequality \eqref{eq:md<13} follows from
\eqref{eq:dtineq1}. 
Therefore, to complete the proof it remains to establish \eqref{eq:dtineq1}, which is accomplished in Lemma~\ref{lem:rsforS}. 
The fact that the values of $x_0,x_1$ in those lemmas are bounds for $b_\infty$ again follows from Corollary~\ref{cor:qtrip2}.
\hfill$\er$
\end{rmk}

\begin{cor}\label{cor:2T}  A descending pre-staircase $\Ss^-_{\al_\infty}$  is live unless it is overshadowed.
\end{cor}
\begin{proof}   First suppose that $m/d>1/3$, and consider a pre-staircase $\Ss^-_{\al_\infty}$ in the family $(S^iR^\de)^{\sharp}(\Cc\Ss^U)\cap [2n+6,2n+8]$ with steps $\bE_k$ and $i+\de$ even, and choose $x_0, x_1, \ell$ as in  Lemma~\ref{lem:2T}~(iii).     Then $x_0< b_{\min}$, where $b_{\min}$ is the minimum of the $b$-values for the staircases in $(S^iR^\de)^{\sharp}(\Cc\Ss^U)\cap [2n+6,2n+8]$.  Therefore, because $m_k/d_k> b_{\al_\infty}\ge b_{\min}>x_0$ there is a constant $d_{x_0}$
that depends only on $x_0$  such that 
$\frac{m_k}{d_k} > x(1+\frac 1{d_k^2})$ whenever $d_k\ge  d_{\min}$.
Further, because there are only finitely many  classes in this whole family that have level less than any fixed number $\ell$, we may suppose that $\bE_k$ has level $\ge \ell$.  
Since the sequence $m_k/d_k$ decreases with limit $b_{\al_\infty}$, it then follows from 
 Lemma~\ref{lem:2T}~(ii), (iii) and our choice of $x_0$ that the sequence
 $A(m_k,d_k,x_0)$  
 increases, and it is easy to check that its limit is also 
 $b _{\al_\infty}$.  Therefore the inequality ~\eqref{eq:rsb} with $b_\infty = b _{\al_\infty}$ holds for all steps with $d_k\ge d_{x}$ and level $\ge \ell$.    Hence  Lemma~\ref{lem:2T}~(i) implies that the degree $d'$ of any class $\bE'$ with $\mu_{\bE', b _{\al_\infty}}(p_k/q_k) \ge 
\mu_{\bE_k, b _{\al_\infty}}(p_k/q_k)$ must be bounded above  by $1/(b _{\al_\infty} - x) \le 
1/(b_{\min} - x)$. But  there are only finitely many exceptional classes of any given degree.  Therefore $\Ss^-_{\al_\infty}$ is live unless there is one class (whose obstruction  would have  to go through the accumulation point) that obsures infinitely many of its steps.  Such a class is an overshadowing class.

The proof when $b<1/3$ is very similar, with the statements in Remark~\ref{rmk:b<13} replacing those of Lemma~\ref{lem:2T}.  Further details are left to the reader.
\end{proof}

\begin{prop}\label{prop:6T}  Every descending pre-staircase $\Ss: = \Ss^-_{\al_\infty}$  is live. Moreover, if $\Ss$ belongs to the family $(S^iR^\de)^{\sharp}(\Cc\Ss^U)\cap [2n+6,2n+8]$ there is a constant $D_0$  that depends only on $n, i, \de$ such that any step in $\Ss$ of degree $> D_0$  is live.
\end{prop} 
\begin{proof}  
By
Lemma~\ref{lem:2T}~(iii), $\Ss$ is live unless it is overshadowed.
We will show that any class that overshadows $\Ss$ must be a blocking class, and hence cannot exist since there are unblocked values of $b$ on both sides of $b_{\al_\infty}$.
For clarity we will consider the cases $b>1/3, b< 1/3$ separately.  Hence 
let us first suppose that  ${\al_\infty}\in Z_{+1}$ so that $\Ss$ is a pre-staircase for some $b_{\al_\infty} > 1/3$.

Suppose that at $b= b_{\al_\infty}$ the obstruction $\mu_{\bE',b}(z) = \frac{A + Cz}{d'-m'b}$ from some class $\bE'$ goes through the accumulation point $\bigl(\al_\infty, V_{b_{\al_\infty}}(\al_\infty)\bigr)$.  If $\bE'$ overshadows 
$\Ss$, 
then we must have $C>A$ since, by Lemma~\ref{lemcor:1T},  the obstruction is steeper than the function $z\mapsto \frac{1+z}{3-b}$ mentioned in \eqref{eq:accform} above.  
Let  $w = w(b)$ be the $z$-coordinate of the point of intersection of the line $\mu_{\bE', b}(z) = \frac{A + Cz}{d'-m'b}$ with the line $ \frac{1+z}{3-b}$.
 Then $w = \frac{(d'-3A) - b(m'-A)}{(3C-d') - b(C-m')}$ so that
 $$
 \frac {\p w}{\p b} = \frac{-(C-A)(3m'-d')}{\bigl((d'-3C) - b(m'-C)\bigr)^2}.
$$
If $3m'-d' \ge 0$, then $\frac {\p w}{\p b}  \le 0$ for all $b$, and so is $< \frac {\p}{\p b} (\acc(b))$ which is  $>0$ when $b>1/3$.
Therefore, because $w(b) = \acc(b)$ when $b = b_{\al_\infty}$, we must have $w(b) < \acc(b)$ when  
$b > b_{\al_\infty}$.  But then  
the class  $\bE'$ blocks all $b$ in some nonempty interval $(b_{\al_\infty}, b_{\al_\infty}+\de)$,
 since the graph of $\mu_{\bE', b}$ crosses  the line $z\mapsto \frac {1+z}{3-b}$ before this line  crosses the volume curve, so that  $\mu_{\bE', b}(\acc(b)) > V_b(\acc(b))$.    
But this is impossible, since by hypothesis there are unobstructed points arbitrarily close to $b_{\al_\infty}$ and on both sides of it. 
Similarly,  if $\frac {\p w}{\p b} > \frac {\p}{\p b} (\acc(b))$, the class $\bE'$ will block $b$ in 
an  interval of the form $(b_{\al_\infty}-\de,b_{\al_\infty})$, which is again impossible.

Therefore, it remains to consider the case when 
$\frac {\p w}{\p b} = \frac {\p}{\p b} (\acc(b))$, which can happen only if  $3m'-d' < 0$. 
In this case $\bE'$ will block some $b$ near $b_{\al_\infty}$ unless 
$w(b) \ge \acc(b)$ for all $b$ near 
$b_{\al_\infty}$.    We show below that in fact when the first derivatives agree we always have
$$
\frac {\p^2 w}{\p b^2}\Big|_{b=b_{\al_\infty}} < \frac {\p^2 \acc(b)}{\p b^2}\Big|_{b=b_{\al_\infty}}.
$$
But this implies that  $w(b) < \acc(b)$ for $b\in (b_{\al_\infty}, b_{\al_\infty}+\de)$,
so that as above  such  $\bE'$ would have to be a blocking class, and hence cannot exist.

To begin the argument, notice that we can assume that $1/3 < b < 5/11 = \acc_+^{-1}(6)$.
Indeed,  otherwise  $b>0.61$
so that  the condition $|d'b - m'|<1$ implies that $d'\le 3$, and there are no potential overshadowing classes of such low degree.  Next observe that,  by
Lemma~\ref{lem:basic},   because $\bE'$ is obstructive at $b_{\al_\infty}$
we have
 $$
0<3(d' b_{\al_\infty} - m') =  d'(3 b_{\al_\infty} - 1)+d'-3m'  < 3.
$$
Therefore, because $d' -3m'>0$ by assumption, we must have $d' -3m' = 1$ or $2$, and if we write   $\eps: = d' -3m'$ we have
$$
w(b) = \frac{(3-b)(m'-A) + \eps}{(3-b)(C-m') - \eps}.
$$
Since $m'-A, C-m'$ are integers and $w(b)> 5$ the terms $(3-b)(m'-A), (3-b)(C-m')$, if nonzero,  dominate $\eps$ and hence must have the same sign.   Further, 
because  $0\le A< C$ and  $w(b) > a_{\min}> 5$, 
 we cannot have $m' = A$.  Indeed,
if $m'=A$,  then because $b\le 5/11$ and  $\eps=1,2$, we have
$w(b) = \frac{\eps}{(3-b)(C-A) - \eps}  \le \frac{\eps}{3-b - \eps} < 5$.

Next observe that, since $\eps/(C-m') \le 2$, we have
\begin{align*}
 \frac {\p w}{\p b} &=\frac{(C-A)\eps}{\bigl((3-b)(C-m') - \eps\bigr)^2}, \quad \mbox{ and} \\
   \frac{\p^2 w}{\p b^2}& =  \frac {\p w}{\p b}\  \frac {2(C-m')}{(3-b)(C-m') - \eps} \leq  \frac {\p w}{\p b}\ \frac {2}{(3-b) - \eps/(C-m')} < \frac {\p w}{\p b}\ \frac {2}{1-b}.
\end{align*}

 On the other hand, if $z(b): = \acc(b)$, then differentiating the equation $z+1/z = \frac{(3-b)^2}{1-b^2}-2$ we obtain 
\begin{align*}  \bigl(1-\frac 1{z^2}\bigr)\frac{\p z}{\p b} = \frac{2(3b-1)(3-b)}{(1-b^2)^2}&=: F(b),  \quad \mbox{ and}\\
 \frac 2{z^3}\Bigl(\frac{\p z}{\p b}\Bigr)^2 + \bigl(1-\frac 1{z^2}\bigr)\frac{\p^2 z}{\p b^2}& = F'(b).
\end{align*}
Now solve the second equation for $\frac{\p^2 z}{\p b^2}$, and simplify the term $(1-\frac 1{z^2})^{-1}F'(b)$ by using the identity  $\frac1{(1-b^2)(1-\frac 1{z^2})}  = \frac 1{2(3b-1)(3-b)}\frac{\p z}{\p b}$, to obtain
\begin{align*}
   \frac{\p^2 z}{\p b^2}&= \Bigl( \frac{2(1-b^2)(5-3b) + 4b(3-b)(3b-1)}{(1-b^2)^2(3-b)(3b-1)} -  \frac{\p z}{\p b}\frac 2{z(z^2-1)} \Bigr)\ \frac{\p z}{\p b}.
\end{align*}
Now suppose that $\frac{\p z}{\p b} =  \frac {\p w}{\p b}$ for some value of $b < 5/11$.  Then, because $z> a_{\min} $ one can check that $  \frac{\p^2 z}{\p^2 b} > 
 \frac{\p^2 w}{\p b^2}$ because 
 \begin{align*}
& \frac {2(5-3b)}{(1-b^2)(3-b)(3b-1)} > \frac{20}{(3-b)(1-b^2)}>\frac 2{1-b}, \quad \mbox{ and } \\
& \frac {4b}{(1-b^2)^2} > \frac{\p z}{\p b}\frac 2{z(z^2-1)} = \frac{4z(3b-1)(3-b)}{(z^2-1)^2(1-b^2)^2}.
\end{align*}
This completes the proof in the case $b>1/3$.

The argument when $b< 1/3$ is similar, except that now $b_{\al_\infty}$ is smaller than
the recursively defined  $b_\infty$ for the  staircase $\Ss^{\Tt_k}_u$ that contains $\bE_k$, while both sides of \eqref{eq:MDTineq0} are negative.  Therefore, as before, this inequality continues to hold at $b_{\al_\infty}$.  Further details are left to the reader.
\end{proof}

\begin{cor}\label{cor:6T}  Proposition~\ref{prop:liveZ}~(i), (ii) holds.
\end{cor}
 \begin{proof} Claim (i) is proved in Corollary~\ref{cor:5T}, while (ii) follows immediately from Lemma~\ref{prop:6T}.
  \end{proof}

\section{Proof of main theorems}
We first develop some arithmetic properties of continued fractions, as preparation for the proofs  of Theorems~\ref{thm:main} parts (i), (ii), (iii) and \ref{thm:main1} in \S\ref{ss:disj}.  After a short discussion 
 of stabilization in \S\ref{ss:stab}, Corollary~\ref{cor:13stair} gives the  proof of 
 Theorem~\ref{thm:main} part (iv).

\subsection{Arithmetic properties of perfect classes}\label{ss:arith}

Recall that from Lemma~\ref{lem:recur0}~(ii) that two quasi-perfect  classes $\bE = (d,m,p,q, t,\eps), \ \bE' = (d',m',p',q', t',\eps)$ (with the same $\eps$-value)
are said to be adjacent iff $dd'-mm' = \min(pq',p'q)$.   Our first aim in this section is to translate this condition into information on the continued fraction expansions of the centers $p/q, p'/q'$.  
We will use the notations and results from Appendix~\ref{app:arith}; in particular $W(p/q)$ denotes the 
weight decomposition~\eqref{eq:Wpq} of $p/q$.

\begin{lemma}\label{lem:WW} Let  $p/q, p'/q'>1$, and write
$
CF(p/q) = [s_0;\dots,s_n], \ CF(p'/q') = [s'_0;\dots,s'_{n'}].
$
\begin{itemlist} \item[{\rm(i)}]   If the inequality
 \begin{align}\label{eq:weight}
W(p/q)\cdot W(p'/q') 
\ge \min(pq', p'q),
\end{align}
holds for $p/q= [s_0;\dots,s_n]$ and $ p'/q'= [s'_0;\dots,s'_{n'}]$ then it also holds for 
 $$
P/Q: = (kp + q)/p =  [k; s_0,\dots,s_n],\;\;\mbox{and } 
P'/Q': = (kp'+q')/p' =  [k; s'_0,\dots,s'_{n'}].
$$
Moreover, there is equality in \eqref{eq:weight} for $p/q, p'/q'$ if and only if there is equality for $P/Q$ and $P'/Q'$.
\item[{\rm(ii)}] The inequality~\ref{eq:weight} holds  for all  pairs $p/q $ and $p'/q'$.
\item[{\rm(iii)}]   There is equality in \eqref{eq:weight} only if
  $s_\al = s_\al'$ for $0\le \alpha\le \min(n,n')-1$.
\end{itemlist}  
\end{lemma}

\begin{proof} 
If we write
$$
W([s_0,\dots,s_n]) = \bigl(q_0^{\times s_0} , \dots, q_n^{\times s_n}\bigr), \quad 
W([s'_0,\dots,s'_{n'}]) = \bigl((q_0')^{\times s_0'} , \dots, (q'_n)^{\times s'_n}\bigr)
$$
then we have
\begin{align*}
W([k, s_0,\dots,s_n])& = \bigl(p^{\times k}, q_0^{\times s_0} ,
\dots, q_n^{\times s_n}\bigr),\\
W([k, s'_0,\dots,s'_{n'}])& = \bigl((p')^{\times k}, (q_0')^{\times s_0'} , \dots, (q'_n)^{\times s'_n}\bigr).
\end{align*}
Let us suppose that $p/q < p'/q'$ so that 
$$
P/Q = (kp+q)/p > P'/Q' = (kp'+q')/(p').
$$
Then, assuming that \eqref{eq:weight} holds for $p/q, p'/q'$ we have
\begin{align*}
W([k, s_0,\dots,s_n])\cdot W([k, s'_0,\dots,s'_{n'}])& = kpp' + W([s_0,\dots,s_n])\cdot W([s'_0,\dots,s'_{n'}])\\
& \ge kpp' +pq' = p(kp'+q') = QP' = \min(QP',PQ').
\end{align*}
Thus \eqref{eq:weight} holds for $P/Q, P'/Q'$ and we either have equality in both cases or neither.   
This proves (i).

Now consider (ii).  By (i) it suffices to consider the case when $s_0 \ne s_0'$, 
and we may assume that $s_0 < s_0'$ so that 
$p_0/q_0<p'_0/q'_0$.  
Then  by \eqref{eq:Wpi}
\begin{align*}
W([s_0,\dots,s_n])\cdot W([s'_0,\dots,s'_n])&  = 
\bigl(q_0^{\times s_0} , \dots, q_n^{\times s_n}\bigr)\cdot \bigl((q_0')^{\times s_0'} , \dots, (q'_n)^{\times s'_n}\bigr)\\
& \ge s_0 q_0 q_0' + q_1 q_0' = (s_0 q_0  + q_1) q_0' = p_0 q_0'.
\end{align*}
This proves (ii).  

Now consider (iii).  Again, by (i) it suffices to consider the case when $s_0\ne s_0'$ and, by renaming if necessary, we may assume
$s_0 < s_0'$.  We must show that equality holds only if either $n=0$ or $n'=0$.
If $q_1=0$ then $p=s_0, q=1$ and $$
W([s_0])\cdot W([s'_0,\dots,]) = s_0 s_0' = pq' = \min(pq',qp').
$$ 
 Otherwise 
the weight decomposition $W([s_0,\dots,s_n])$ must have at least $s_0 + 2$ terms, which means that there is equality only  if  $s_0' = s_0 + 1$ and $s_1' = 0$.  Thus  $n'=0$, as required.
\end{proof}

\begin{cor}\label{cor:WW}  Suppose that there is equality in \eqref{eq:weight} and that 
$p/q<p'/q'$. Then $s_i = s_i'$ for $0\le i \le \min(n,n')-1$ and
\begin{itemlist}\item[{\rm (i)}] if $\min(n,n')$ is even, one of the following three possibilities occurs:
\begin{itemize}\item[-]
$n=n'$ and $s_n< s_n'$, or 
\item [-] $n<n'$ and $s_n \le s_n'$, or  
\item[-]  $n>n'$ and  $s_{n'} = s'_{n'} - 1$.
\end{itemize}
\item[{\rm (ii)}] if $\min(n,n')$ is odd, one of the following three possibilities occurs:
\begin{itemize}\item[-]
$n=n'$ and $s_{n}>s'_{n}$, or
\item[-]  $n<n'$ and $s'_{n} = s_{n} - 1$,
or
\item[-] $n>n'$ and $s_{n'}\ge s'_{n'}$. 
\end{itemize}
\end{itemlist} 
\end{cor}
\begin{proof}   Consider (i).  As in the proof of Lemma~\ref{lem:WW}, we may suppose that $\min(n,n')=0$.
Notice because we are forgetting an even number of terms, we still have $p/q<p'/q'$.
Therefore, if $n = 0\le n'$, then  $q = 1, p = s_0$ and
$$
W([s_0]) \cdot W([s_0',\dots ])  = s_0 q_0' = pq'.
$$
Thus
the first two cases are clear.     If $n> n' = 0$, then $s_0<s_0'$, $q' = 1$   and 
\begin{align*}
 W([s_0, s_1,\dots ]) \cdot W([s_0'])& = (q_0^{\times s_0}, q_1^{\times s_1},\dots])\cdot (1^{\times s_0'})\\
 & \ge  s_0 q_0 + q_1 = p = pq'.
\end{align*}
Moreover, by our conventions about continued fractions $W(p/q)$ must have at least two terms after the initial block  $q_0^{\times s_0}$, there is
 equality only if $s_0' = s_0 + 1$.  This proves (i).  The proof of (ii) is similar, and is left to the reader.
\end{proof}

\begin{prop} \label{prop:adjac} 
Let $\bE = (d,m,p,q, t,\eps), \ \bE' = (d',m',p',q', t',\eps)$ be perfect classes.
 If $\bE, \bE'$ are adjacent, then $W(p/q)\cdot W(p'/q')= \min(pq', p'q)$ and the conditions in
 Corollary~\ref{cor:WW} hold.  Further $\bE\cdot \bE' = 0$.
 \end{prop}
\begin{proof}  A quasi-perfect class cannot be adjacent to itself because each such  class satisfies $d^2-m^2 = pq-1 \ne pq$.  Therefore we have $\bE\ne \bE'$ and hence $\bE\cdot \bE' \ge 0$ because the classes are assumed to be perfect.   Thus
\begin{align*}
0 \le \bE\cdot \bE' = dd' - mm' - W(p/q)\cdot W(p'/q') = \min(pq',p'q)   - W(p/q)\cdot W(p'/q'),
\end{align*}
where the last equality  holds by Lemma~\ref{lem:recur0}~(ii).  By \eqref{eq:weight} this is possible only if 
$\min(pq',p'q)= W(p/q)\cdot W(p'/q')$ so that 
the conditions in Corollary~\ref{cor:WW} hold.
Further  $\bE\cdot \bE'=0$.\end{proof}

\begin{lemma}\label{lem:CFlength} Let $\Tt$ be a triple in $\Cc\Ss^U$, and let
$p_k/q_k = [s_{k,0};\dots,s_{k,n_k}], k\ge 0,$ be the step centers in one of the associated staircases $\Ss = \Ss^\Tt_\ell,$ or $ \Ss^\Tt_u$.  Write
$$
W(p_k/q_k) = [q_{k,0}^{\times s_{k,0}},\dots, q_{k,n_k}^{\times s_{k,n_k}}],\ \ k\ge 0.
$$
Then
\begin{itemlist} \item[{\rm(i)}]  The numbers $n_k: = \ell_{CF}(p_k/q_k)$ strictly increase.\footnote{Here, $\ell_{CF}$ is the length of the continued fraction defined in \eqref{eq:WCFlength}}
\item[{\rm(ii)}] 
For all $k\ge 0$ and $i< n_k$  we have $s_{k+j,i} = s_{k,i}$ for all  $j\ge 1$.  
Moreover for each $\ell$ and  $i\le n_{\ell} $ the weights $q_{k,i}$ satisfy the recursion
$$
q_{k+2,i} = t q_{k+1,i} - q_{k,i},\quad  k\ge \ell,
$$
where $t$ is the recursion parameter of $\Ss$.
\item[{\rm(iii)}] The accumulation point of $\Ss$ has infinite continued fraction $[s^\Ss_0; \dots, s^\Ss_i,\dots]$.  Further each step satisfies
\begin{align*}
CF(p_k/q_k) = [s^\Ss_0; \dots, s^\Ss_{n_k-1}, s^\Ss_{n_k} +\de]
\end{align*}
for some $\de \ne 0$, 
where $\de = +1$ if either $\Ss$ ascends and $n_k$ is odd, or $\Ss$ descends and $n_k$ is even.
\end{itemlist}
\end{lemma}
\begin{proof}  Consider a pair of steps $p_k/q_k, p_{k+1}/q_{k+1}$.  Since they are adjacent, Lemma~\ref{lem:WW}~(iii) shows that 
their continued fractions of lengths $n_k, n_{k+1}$ agree until the term  in the $n$th place, where $n = \min(n_k, n_{k+1})$.  We want to rule out the possibility that $n: = n_{k+1}\le n_k$.  
Because the steps are adjacent and $t$-compatible, the recursion parameter $t $ is given by $t= |p_{k+1}q_k - q_{k+1}p_k|$; and we may assume $t\ge 5$, since the only staircase in $\Cc\Ss^U$ with $t=3$ is
the descending stair with blocking class $\bB^U_0$, which has the property in question here. 
We will carry out the argument assuming that
 that the staircase ascends, so that  $p_k/q_k< p_{k+1}/q_{k+1}$, leaving the descending case to the reader.

We argue by contradiction.  Thus, suppose  that $n: = n_{k+1}\le n_k$, and  
write
\begin{align*}
W(p_k/q_k)&= (q_{k,0}^{\times s_0},\dots, q_{k,n-1}^{\times s_{n-1}}, q_{k,n}^{\times s_n},\dots), \\
W(p_{k+1}/q_{k+1}) &= \bigl((q'_{k+1,0})^{\times s_0},\dots, (q'_{k+1,n-1}=m)^{\times s_{n-1}}, 1^{\times m}\bigr).
\end{align*}
If $n$ is odd, then Corollary~\ref{cor:WW}~(ii) implies that $s_n\ge m$.  But then $q_{k,n-1}\ge m$ which implies that $q_{k,0} \ge  q'_{k,0}$, contrary to our hypothesis.
Hence $n$ is even.

Next observe that there is a $2\times 2$ matrix $A$ with $\det A = 1$ such that
\begin{align*}
\begin{pmatrix} p_{k+1}&p_k\\q_{k+1}&q_k\end{pmatrix} = A\begin{pmatrix} m&  q_{k,n-1}\\1& q_{k,n}\end{pmatrix}, \quad A = \prod_{i=0}^{n-1} \begin{pmatrix} s_i&1\\1&0\end{pmatrix}
= \begin{pmatrix} x&y\\z&w\end{pmatrix}.
\end{align*}
Therefore, because  $p_{k+1} = tp_k - p_{k-1} > (t-1)p_k$, we have 
$$
p_{k+1}= xm+y > (t-1) (xq_{k,n-1} + yq_{k,n}). 
$$
Since the entries of $A$ are nonnegative and $y< (t-1) y q_{k,n}$, we must have $$
xm > (t-1)x q_{k,n-1}, \quad \mbox{ i.e. } \quad m>(t-1)q_{k,n-1}. 
$$
 But also because $\det A=1$ and $t= p_{k+1}q_k - q_{k+1}p_k$, we know that 
$t = mq_{k,n} - q_{k,n-1}$. Thus, writing $q_{k,n-1}=a, q_{k,n}=b$ for simplicity, we have
$t = mb - a$ and 
$m> (t-1)a$  so that $a+t = mb > (t-1)ab$.  But $b\ge 1$, so that we need
$t>(t-2) a$.  Since $t\ge 5$ and $a\ge 2$ this cannot occur.

This completes the proof when the staircase ascends.  The case of a descending staircase is essentially the same, except that $n$ is now odd and $\det A = -1$.    This proves (i).                                                                       
\MS

The first claim in (ii) holds by (i) and Lemma~\ref{lem:WW}~(iii).
Then the
 second claim holds because, by \eqref{eq:Wpq}, if $p/q=[s_0,\dots,s_n]$  the $i$th weight $q_i$ of $p/q$ depends only on 
$p, q=q_0(p/q)$ and  $s_j$ for $j< i$.   
Thus, because the relevant $s_j$ do not depend on $k\ge k_0$,   the weights $q_{k,i}, i\le n_{k_0},$ satisfy the defining recursion of $\Ss$.  Thus (ii) holds.
\MS

Towards (iii), note first that as explained in Remark~\ref{rmk:recur0}~(i),
 the accumulation point of $\Ss$ is irrational and hence has infinite continued fraction.  Next,
consider two adjacent terms
$$
p_{k}/q_{k} = [s_{k,0}; \dots, s_{k,n_k}] , \quad
p_{k+1}/q_{k+1} = [s_{k+1,0}; \dots, s_{k+1,n_k}, s_{k+1,n_k+1}, \dots]
$$
Since $\bE_k, \bE_{k+1}$ are adjacent, there is equality in \eqref{eq:weight} so that we may apply Corollary~\ref{cor:WW}.  Therefore $s_{k+1,i} = s_{k,i}$ for $i< n_k$. 
Similarly, we have $$
s_{k+1, n_k} = s_{k+2, n_k} = s_{k+\ell, n_k} = s^\infty_{n_k}, \quad \ell\ge 1.
$$

Thus the general formula in (iii) holds, and we just have to check that $\de=1$ in the two given circumstances.  
Suppose first that  $\Ss$ ascends, and that  $n_k$ is odd.
Then, in the language of Corollary~\ref{cor:WW}, we have $p/q=p_k/q_k < p'/q' = p_{k+1}/q_{k+1}$ so that $n= n_k < n' = n_{k+1}$.  
Then by part (ii) we have $s_n =  s'_n+1$ which gives
$$
 s_{k,n_k} = s_{k+1,n_{k+1}} + 1 = s^\infty_{n} + 1
$$
as claimed.   Similarly, if $\Ss$ descends, and $n_k$ is even, we have
$p/q=p_{k+1}/q_{k+1}  < p'/q' = p_k/q_k$ and $n>n' = n_k$.  Hence by (i) we have
$$
s'_{n'} =  s_{n,k} = s_{n'}+ 1 = s_{n,k+1} + 1 = s^\infty_{n} + 1,
$$
as claimed.
\end{proof}

\begin{cor}\label{cor:adjac}    Let $\Ss = (\bE_k)_{k\ge 1}$ be any principal\footnote{Technically every staircase in $\Cc\Ss^U$ is principal by Definition~\ref{def:complete}, but we include \lq\lq principal'' here to highlight that these are not the staircases $\Ss_{\alpha_\infty}$ of Definitionn~\ref{def:newstair}, which for $\alpha_\infty\in[6,8]$ consist of perfect classes in $\Cc\Ss^U$.} staircase 
in the complete family $\Cc\Ss^U$ or one its images under a symmetry.  If $\Ss$ has recursion parameter  $t\ge 3$, then
\begin{align*}
&\bE_k\cdot \bE_{k+1} = 0,\qquad  \bE_k\cdot \bE_{k+2} = 1, \\
&\bE_k\cdot \bE_{k+j+1} = t \bE_k\cdot \bE_{k+j}  - \bE_k\cdot \bE_{k+j-1},\quad \forall j\ge 1,\ k\ge 0.
\end{align*}
\end{cor}
\begin{proof}   
Since the symmetries preserve adjacency by Lemma~\ref{lem:Tgen}, we know that every pair of classes
$\bE_k, \bE_{k+1}$ are adjacent by Proposition \ref{prop:gener}.
Therefore by Lemma~\ref{lem:CFlength}~(ii),  when $j\ge 1$ 
we have 
$$
q_{k+j+1,\al} = t  q_{k+j,\al}  - q_{k+j-1,\al}, \quad \al\le \ell_{CF}(p_k/q_k).
$$ 
Since these are the only terms in $W(p_{k+j}/q_{k+j})$ involved in computing $\bE_{k}\cdot\bE_{k+j}$, and $d_{k+j},m_{k+j}$ satisfy the same recursion, we always have
$$
\bE_k\cdot \bE_{k+j+1} = t \bE_k\cdot \bE_{k+j} - \bE_k\cdot \bE_{k+j-1},\quad j\ge 1.
$$
In particular, when $j=1$,
Proposition~\ref{prop:adjac} implies that 
$$
\bE_k\cdot \bE_{k+2} = t\ \bE_k\cdot \bE_{k+1} - \bE_k\cdot \bE_{k} = 0 - (-1) = 1.
$$
 This completes the proof.\end{proof}

Given a blocking class $\bE_\nu, \nu=\la,\rho$ we denote the corresponding blocked  $z$-interval by 
$(\p^-_\nu,\p^+_\nu):=I_{\bE_\nu}$.  We now show that in any derived triple $\Tt$ the  center point of the middle class $\bE_\mu$ has the minimal weight length among all points in $(p_\la/q_\la, p_\rho/q_\rho)$ that are not blocked by $\bE_\la$ or $\bE_\rho$. (We can expand the interval $I_\Tt$ in the conclusion of Proposition~\ref{prop:short} to $(p_\la/q_\la,p_\rho/q_\rho)$ by the fact that the center of a perfect class $\bE$ has shortest weight length amongst points in $I_\bE$ by the fact that $I_\bE$ is contained in the interval on which $\bE$ is nontrivial, and on this interval the center has shortest weight length by \cite[Lem.~2.28~(1)]{AADT}.)

\begin{prop}\label{prop:short}  Let $\Tt: = (\bE_\la, \bE_\mu, \bE_\rho)$ be any triple derived from one of the basic triples $(\bB^U_n, \bE_{[2n+7;2n+4]},\bB^U_{n+1})$.   Then the weight length of the center $p_\mu/q_\mu$ is strictly less than the weight length of any other point $p/q \in I_\Tt$.\footnote{The interval $I_\Tt=(\p^+(I_{\bE_\la}),\p^-(I_{\bE_\rho}))$ is defined in \eqref{eq:JjT}.}
 \end{prop}
 \begin{proof} 
  Let $p/q$ be any rational number lying between the centers of $\bE_\la$ and $\bE_\rho$.
  Write 
   $$
  CF(p/q) = [a_0,\dots,a_\ell], \quad CF(p_\mu/q_\mu) = [s_0,\dots, s_n].
  $$
 and assume that
  $$
  \ell_{wt}(p/q)  = \sum_{\al=1}^\ell a_\al\; < \; \ell_{wt}(p_\mu/q_\mu) = \sum_{\al=1}^n s_\al.
  $$
We aim to show that $p/q$ must either be  $< \p^+_\la$ or $>\p^-_\rho$.
Note that, if we define $\al_0: = \min\{\al \ | \ a_\al \ne s_\al\}$,
then  we must have $\al_0 \le n$.
  
  Suppose first that $\al_0 < n$.  Then we claim that $p/q$ is either $< \p^+_\la$ or $>\p^-_\rho$.  This holds because, by Lemma~\ref{lem:CFlength}~(iii) these limit points have infinite continued fractions with first $n_k-1$ terms equal to those of  $p_\mu/q_\mu$.  Thus, for example, if $\al_0$ is odd
  and $a_{\al_0} < s_{\al_0}$ then $p/q > \p^-_\rho$, while if 
  $\al_0$ is odd
  and $a_{\al_0} > s_{\al_0}$ then $p/q < \p^+_\la$.  The other cases are similar.  
  
Next suppose that $\al_0 = n$.  Then, because $  \ell_{wt}(p/q) <  \ell_{wt}(p_\mu/q_\mu) $ we must have 
  $a_n < s_n$.  We consider the cases $n$ even or odd separately.
  Suppose that $n$ is even.  Then $p/q < p_\mu/q_\mu$, and so we need to check that $p/q$ is smaller than the accumulation point $\p^+_\la$ of the descending staircase $\Ss=\Ss^\Tt_u$.  We saw in  
  Lemma~\ref{lem:CFlength}~(iii) that in this case the last entry $s_n$ in $CF(p_\mu/q_\mu)$ satisfies
  $s_n = s^\Ss_n + 1$.  Therefore, $a_n \le s^\Ss_n$.  But $a_n$ is the last element in  $CF(p/q)$ while 
  $CF(\p^+_\la)$ is infinite.  Therefore $p/q < \p^+_\la$ by   \eqref{eq:odd}.
  
 The argument when $n$ is odd is similar, except that now one compares with the ascending staircase. Therefore such $p/q$ cannot exist.
 \end{proof}

  \begin{example}  \rm  Here are some simple examples to illustrate the argument in Proposition~\ref{prop:short}.
 Consider the basic triple 
 $\Tt_*^0: = (\bB^U_0, \bE_{[7;4]},\bB^U_1)$, and take $p/q = [7;3]$.  
 Then $\p^+_\la = [7; \{5,1\}^\infty],\p^-_\rho = [\{7,3\}^\infty]$ by \eqref{eq:SsUsteps}.  
Since in this case  $n=1$ is odd we should have $p/q >  [\{7,3\}^\infty]$, which can be readily checked.
Similarly, in the triple $x\Tt_*^0 = (\bB^U_0, \bE_{[7;5,2]}, \bE_{[7;4]})$ one can calculate that
the limit of the ascending staircase is $[\{7,5,3,1\}^\infty]$ so that,
for example,  $[7;5]>  [\{7,5,3,1\}^\infty]$ is blocked by $\bE_{[7;4]}$.
 \hfill$\er$
 \end{example}

\begin{lemma}\label{lem:short}  Every rational point in $[6,\infty)$  lies in some $I_\bE$ for $\bE\in \Cc\Ss^U$.  In particular, no staircase with $b>1/3$ 
accumulates to a rational point in $[6,\infty)$.  
\end{lemma}

\begin{proof} 
Fix $n\ge 0$ and consider the classes $\bE_\de$ that belong to a triple with centers in $[2n+6, 2n+8]$.  Define
\begin{align*}
\om_k: = \min\{ \ell_{wt}(p/q) \ \big| \ p/q \mbox{ is the center of some }\bE_\de \mbox{ of level }  k\}.
\end{align*}
Thus $$
\om_1 = \ell_{wt}([2n+7;2n+4]) = 4n+11 \ge 4n+1. 
$$

Since each center $p/q$ of a class $\bE_\de$ at level $k+1$ is adjacent to a center at level $k$, it follows from Lemma~\ref{lem:CFlength} and Corollary~\ref{cor:WW} that $\om_{k+1}\ge \om_k  +1$.  Indeed,
Lemma~\ref{lem:CFlength}~(i) shows that  the CF-length of $p/q$ is at least one more than that of the centers at level $k$  and, even if $\de= -1$, the last entry of $CF(p/q)$ is $\ge 2$; thus the weight length increases by at least $1$. Hence $\om_k\ge 4n + k$. 
In other words the weight lengths of the centers at level $k$ (which are minimal among the unblocked weight lengths) are all at least $4n+k$, so that all rational numbers of weight length $< 4n+k$ are blocked by classes at level $<k$.   In particular, if $p/q$ is any rational number with $\ell_{wt}(p/q) = K$ then $p/q$ is blocked by some class in $\Cc\Ss^U$ at level $<K$.
\end{proof}

Recall the sequence $y_0=0, y_1 = 1, y_2 = 6, \dots, y_{i+1} = 6y_i - y_{i-1}$  and the associated points
\begin{align}\label{eq:vw1}
v_1 = \infty,\ v_2 = 6,\  v_3 = 35/6, \ \dots,\  v_k = y_{k}/y_{k-1} = S^{k-2}(v_2), \dots
\end{align}  
from \eqref{eq:symmvi}.   The points $$
w_k= (y_{k}+y_{k-1})/y_{k-1} = 7,\ 41/7,\ \dots
$$
 are also relevant; see Figure~\ref{fig:symm}.

\begin{cor}\label{cor:short2} \begin{itemlist}\item[{\rm (i)}] Every rational point in $(3+2\sqrt2, \infty)$  is blocked by a class with $m/d>1/3$,
except for the points $v_{2i+1}, i\ge 1$ which are unobstructed.
\item[{\rm (ii)}] Every rational point in $(3+2\sqrt2, \tau^4)$  is blocked by a class with $m/d<1/3$,
except for the points $v_{2i}, i\ge 1$  which are unobstructed.
\end{itemlist}
\end{cor}
\begin{proof}   The symmetry $R: p/q\mapsto (6p-35q)/(p-6)$ maps the interval $[7,\infty) = [w_1,v_1)$ to the interval $(6,7] = (v_2,w_1]$ and hence takes the complete family $\Cc\Ss^U$ to a complete family $R^{\sharp}(\Cc\Ss^U)$ whose blocking classes have $m/d< 1/3$ and block all rational points in $(6,7]\cap \acc([0,1/3))$.
Similarly, the symmetry $SR: p/q\mapsto (6p-35q)/(p-6)$ maps the interval $[7,\infty) = [w_1,v_1)$ to the interval $(v_3,w_2] = (35/6,41/7]$ and hence takes the complete family $\Cc\Ss^U$ to a complete family $(SR)^{\sharp}(\Cc\Ss^U)$ whose blocking classes have $m/d> 1/3$ and block all rational points in $(v_3,w_2]$.
We proved in \cite[Lem.3.4.6]{MM} that the  class $(S^i)^{\sharp}(\bB^U_0)$ blocks the interval $[w_{i+1},w_i]$ for $b<1/3$ when $i$ is odd and for $b>1/3$ when $i$ is even.  Hence the families
$(S^{2i})^{\sharp}(\Cc\Ss^U)$ and $(S^{2i+1}R)^{\sharp}(\Cc\Ss^U)$ for $i\ge 0$ together block all rational points 
for $b>1/3$ except for the points $v_3,v_5,$ and so on; see Figure~\ref{fig:symm}.  Notice that the points $v_{2i+1}$ are unobstructed since they are limits of accumulation points of staircases, which necessarily are unobstructed.

Similarly,  the families
$(S^{2i+1})^{\sharp}(\Cc\Ss^U)$ and $(S^{2i}R)^{\sharp}(\Cc\Ss^U)$ for $i\ge 0$ together block all rational points 
for $b<1/3$ except for the points $v_2,v_4,$ and so on.  Again, these points are unobstructed.
\end{proof}
\begin{figure}
    \centering
\includegraphics[width=.6\textwidth]{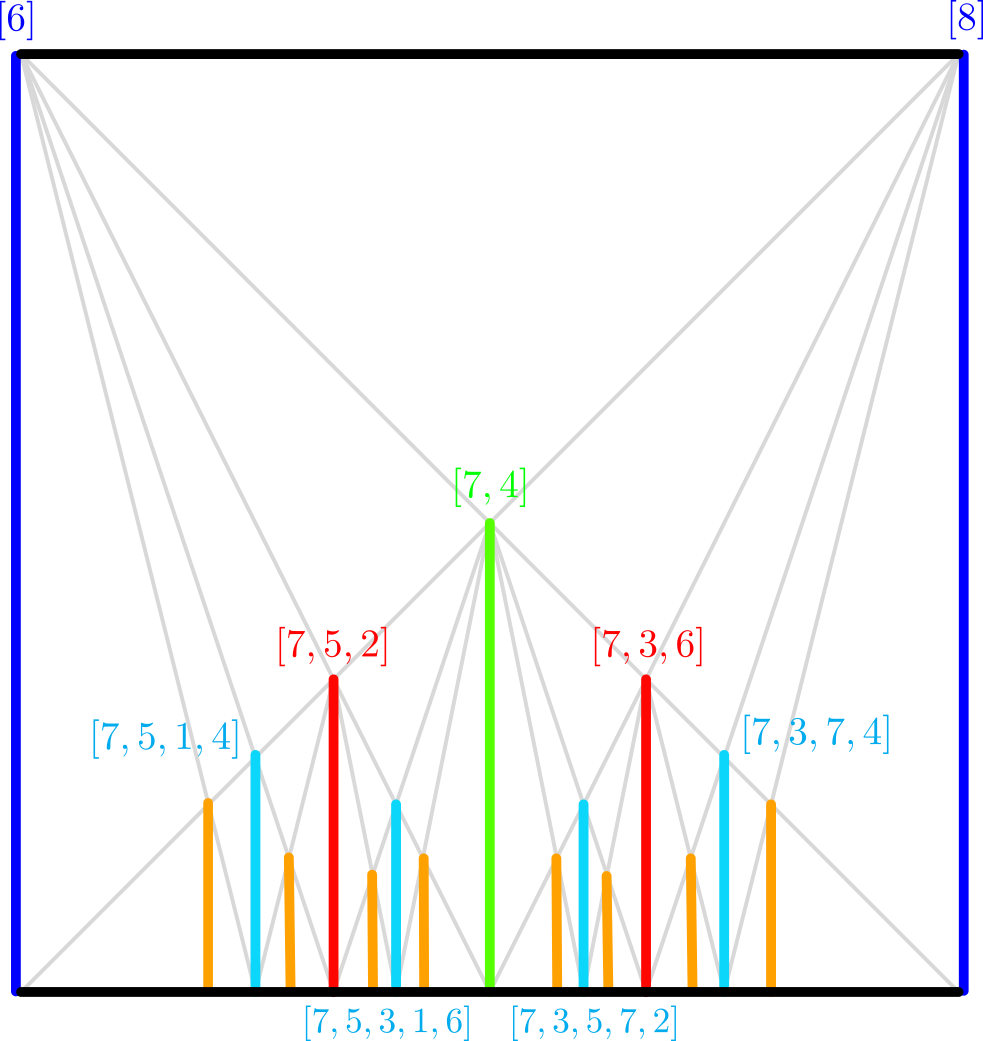}
    \caption{This diagram depicts the blocking classes between $6$ and $8$ through level five. The horizontal black lines represent the interval $[6,8]$. Each vertical line through level four is labeled with the continued fraction of a number between $6$ and $8$ which is the center of a blocking class. 
    To obtain the vertical line representing $E_\mu$ from that of $E_\lambda$ and $E_\rho$, draw a diagonal (gray) line from the top of the line representing $E_\mu$ to the bottom of the line representing $E_\rho$, and vice versa. The vertical line dropped from the intersection of these two diagonals represents $E_\mu$. 
    Any unbroken gray {``V\rq\rq} shape represents a generating triple. Each gray diagonal line represents a principal staircase 
     associated to the blocking class at its lower endpoint; the steps are the classes whose upper endpoints are on the given diagonal.    
    Classes on the same horizontal line have the same $CF$-length, while classes whose upper endpoints are on the same level are depicted in the same color.}\vspace{-.1in}
    \label{fig:Fareystairs}
\end{figure}

We end this subsection with an extended remark on the combinatorics and arithmetic properties of the family $\mathcal{CS}^U$ on the interval $[6,8]$, since they are so similar to those of Farey sequences.

\begin{rmk}\rm\label{rmk:Farey}  {\bf (Farey description of the blocking classes)}
Suppose that $p/q \in [p_1/q_1, p_2/q_2]$ has $q>q_1,q_2$.  
The first Farey sequence\footnote
{
 A subsequence of this sequence is relevant to the construction of  the weight decomposition of $p_\mu/q_\mu$; see the Appendix to \cite{ball}.} containing $p/q$   
consists of all rational numbers $p'/q'$ in the interval $ [p_1/q_1, p_2/q_2]$ with $q'\le q$, arranged in increasing order.  
For example, in the case of $[7;4]\in [7,8]$ the sequence is
is
\[
[7],\ [7;4],\ [7;3],\ [7;2], \ [7;1,2],\ [7;1,3],\ [8];
\]
similarly,  the first Farey sequence between $[7]$ and  
$[7;4]$ containing $[7;5,2]$ is
\[
[7], [7; 11], [7; 10], [7; 9], [7; 8], [7; 7], [7; 6], [7; 5, 2], [7; 5], [7; 4, 2], [7; 4].
\]
This sequence has the property that any three adjacent terms $p'/q', p/q, p''/q''$ satisfy the identity\footnote
{This expression is called the Farey sum of $p'/q', p''/q''$.}
$$
 \frac pq =  \frac{p'+p''}{q'+q''} =: \frac{p'}{q'} \oplus  \frac{p''}{q''}.
$$
Moreover, it turns out that the sequence can be constructed by repeatedly taking the Farey sum of two adjacent elements, discarding any with too large a denominator.

Proposition \ref{prop:short} implies that given any triple in $\Cc\Ss^U$ the nearest neighbors to 
$p_\mu/q_\mu$ in the first Farey sequence between $p_\la/q_\la$ and $p_\rho/q_\rho$ that contains $\rho_\mu/q_\mu$ lie in the blocked intervals $I_{\bE_\la}, I_{\bE_\rho}$.     Thus $p_\mu/q_\mu$ is the Farey sum of these numbers.
For example, 
$$
[7;4] = \frac{29}{4} =   \frac{7+22}{1+3} =:[7] \oplus [7;3], \quad  [7;5,2] =  [7;6] \oplus [7;5].
$$
and $[7], [7;6]\in I_{\bB^U_0}$ 
while $[7;3]\in I_{\bB^U_1}$, and $[7;5]\in I_{\bE_{[7;4]}}$. 
 \hfill$\er$
\end{rmk}

Thus  the class $E_\mu$ can be viewed as a type of Farey sum of $E_\lambda$ and $E_\rho$: given the blocked intervals $I_{\bE_\lambda}$ and $I_{\bE_\rho}$, there is a unique rational number between them with shortest weight decomposition, which will also be the center of a blocking class.
In this language, the first part of Conjecture \ref{conj:1} claims that the $\Cc\Ss$-length of the class $\bE_\de$ is precisely the continued fraction length $\ell_{CF}$  of its center. In particular, this would imply  that
\[
\ell_{CF}(\bE_\lambda)+\ell_{CF}(\bE_\rho)=\ell_{CF}(\bE_\mu),
\]
which experimentally seems to be true.  In fact there seems to be more internal structure here that comes from the relation of the adjacency condition to weight expansions explained in Lemma~\ref{lem:WW}.  
For example, if a principal staircase %
has blocking class $\bE_\mu$  with center $[s_0;s_1,\dots, s_n]$ then 
in (almost) all cases we have calculated \footnote{
The descending stairs associated to the blocking classes $\bB^U_n$ are the only exceptions we have found.}
the limit has the form  
$[s_0;s_1,\dots, s_{n-1}, P_\mu^{\infty}]$ where the periodic part $P_\mu$ has length $2(n+1)$ and is a combination in some order of the periodic parts $P_\la, P_\rho$ associated to $\bE_\lambda$ and $\bE_\rho$.
We give examples in the following table.\MS

\begin{tabular}{|l|l|l|}
    \hline center & $\p^-(I_\bE)$ & $\p^+(I_\bE)$
    \\\hline$\bE_0=[6]$ & $[\{5,1\}^\infty]$
     & $[7,\{5,1\}^\infty]$
    \\\hline$\bE_1=[8]$ & $[\{7,3\}^\infty]$ & $[9;\{7,3\}^\infty]$
    \\\hline$\bE_{.1}=[7;4]$ & $[7;\{5,3,1,7\}^\infty]$ & $[7;\{3,5,7,1\}^\infty]$
    \\\hline$\bE_{.01}=[7;5,2]$ & 
     $[7;5,\{1,3,5,1,7,5\}^\infty]$  & $[7;5,\{3,1,5,7,1,5\}^\infty]$
    \\\hline$\bE_{.21}=[7,3,6]$ & 
      $[7;3,\{5, 7, 3, 1,7,3\}^\infty]$   & $[7;3,\{7, 5, 3, 7, 1,3\}^\infty]$
    \\\hline
\end{tabular}
\MS

This pattern can be depicted just as the Farey diagram is recorded, and the two ways of depicting the Farey diagram (as in \cite[\S1.2]{hatcher}) emphasize two different features of the set of centers of our blocking classes. In the first diagram in Figure \ref{fig:decimal}, classes lying on the same horizontal line have equal level (see \S\ref{ss:ternary}). In Figure \ref{fig:Fareystairs}, however, classes lying on the same diagonal line lie in the same principal staircase, 
and classes lying on the same horizontal line have equal $\Cc\Ss$-length, thus equal ${CF}$-length (by Conjecture \ref{conj:1}).

\begin{rmk}\label{rmk:stairrat}\rm  If $b$ is rational, then the accumulation point formula forces $\acc(b)$ to be at worst a quadratic irrational, and hence to have periodic continued fraction. By Corollary \ref{cor:short2} no rational number greater than six can be an accumulation point of a staircase.
If we knew enough about the numerics of the principal staircases, then by extending the arguments in \S\ref{ss:uncount} one might be able to conclude that the only unblocked points $z$ with  periodic continued fraction are the endpoints of the blocked intervals.  Since none of these correspond to a rational $b$ by Remark~\ref{rmk:recur0}, it would follow that
  the only possible rational $b$ with staircases are $b=0,1/3$ or the special rational $b$ corresponding to the $z$ in \eqref{eq:specrat}.
\hfill$\er$
\end{rmk}

\subsection{Proof of Theorem~\ref{thm:main}}\label{ss:disj}

 Our arguments are based on the following result. 
 
 \begin{prop}\label{prop:disj}   
 Let  $\bE$ be a perfect class with $p/q> a_{\min}$ such that the endpoints $\p J_\bE$ are unobstructed, and let 
$\bE'\ne \bE$ be another perfect class with $p'/q'> a_{\min}$.  Then  $J_\bE\cap J_{\bE'} = \emptyset$.
 \end{prop}

 \begin{remark} \label{rmk:except}\rm 
    As noted in Remark~\ref{rmk:main0}~(i), the class $B=6L-3E_0-\sum_{i=1}^7 2E_i$ is an exceptional class since it Cremona reduces to $(0,-1,0,\hdots)$, but not a perfect class because the coefficients of the $E_i$ are twice those of the integer weight expansion of its center. This class illustrates the importance of the assumptions in Proposition~\ref{prop:disj} since, as we now show, $B$ is a blocking class such that $\emptyset\ne J_B \subset J_{\bB^U_0}.$ 
Note first that by \cite[Thm.1]{ICERM} we have
\begin{align}\label{eq:JBx}
J_{\bB^U_0} = \left(\frac{3-\sqrt 5}2, \frac {3(7+\sqrt 5)}{44}\right) \approx(0.382, 0.63),\qquad I_{\bB^U_0}= \bigl([\{5,1\}^\infty]], [7;\{5,1\}^\infty]\bigr).
\end{align}
Further, because the coefficients of $E_1,\dots,E_7$ in $B$ are a constant multiple of the weight expansion of $7/1$, the proof of \cite[Prop.~21(i)]{ICERM} (see also Lemma~\ref{lem:pos}) adapts to show that $B$ is live at $z=7$ for an interval of $b$ that includes $(4/9,10/19)\approx(0.44,0.526)$; in other words $c_{H_b}(7)=\mu_{B,b}(7)=\frac{14}{6-3b}$ for these $b$. We may also directly compute that $\mu_{B,\acc^{-1}_{+1}(7)}(7)\approx3.368>3.354\approx V_{H_{\acc^{-1}_{+1}(7)}}(7)$. Thus  $B$ blocks $\acc^{-1}_{+1}(7)\approx0.614$, which is contained in $J_{\bB^U_0}$ by \eqref{eq:JBx}. 
On the other hand, the class $\bB^U_0$ has $\mu_{\bB^U_0,b}(7) = \frac{6}{3-2b}$, so that
$\mu_{\bB^U_0,b}(7)=\mu_{B,b}(7)$ for $b=3/5$. By comparing the ratios of $m/d$, it follows from
 Lemma~\ref{lem:md}~(i) above that 
  the obstruction function for $\bB^U_0$ grows faster at $z=7$ than that of $B$ as $b$ increases from $3/5$ to $\p^+(J_{\bB^U_0})$.  Since both these obstruction functions are constant for $z\ge 7$ and $\mu_{\bB^U_0,b}(z) = V_{H_b}(z)$ at
  $z=\p^+(I_{\bB^U_0})>7$ and $b=\p^+(J_{\bB^U_0})>3/5$,  it follows that
   the class $B$ is no longer obstructive 
 at these values of $z,b$.

 The argument that $\p^-(J_{\bB^U_0}) < \p^-(J_{B})$ and that for $z<7$ we have $\mu_{B,b}(z)<\mu_{\bB^U_0,b}(z)$ is easier, since for $z< 7$ the obstruction $\mu_{B,b}$ lies on the line through the origin and $(7,\frac{14}{6-3b})$, while 
$ \mu_{\bB^U_0,b}$ is constant for $z\in [6,7]$ and for $z<6$ lies on the line from the origin to   $(6,\frac{6}{3-2b})$.

Another exceptional class which is not perfect yet is live and blocking is $73L-46E_0-22E_1-21E_{2,\dots,7}-4E_{8,\dots,12}-1E_{13,\dots,16}$, with center $151/21$ and whose blocked interval is nested inside that of $\bE_{[7;4]}$. Again because the coefficients of $E_1,\dots,E_{16}$ are closely related to the entries in the weight expansion of its center, we may use \cite[Prop.~21(i)]{ICERM} to find an interval on which this class is live and an argument similar to the one above to show that $J_{\bE_{[7;4]}}$ contains its blocked interval. We expect there are many such classes. This is in contrast to, for example, the class $5L-E_0-2E_{1,\dots,6}-E_7$, which is live for $b\approx1/5$ as explained in \cite[Ex.~34]{ICERM} but is not a blocking class by \cite[(2.2.6)]{ICERM}.
  \hfill$\er$
 \end{remark}

Since  in Proposition~\ref{prop:disj}  we assume that the endpoints of $J_{\bE}$ are unobstructed, the two intervals $J_{\bE}$ and $J_{\bE'}$ cannot overlap, and the key point of the proof is to show that they are not nested: in other words we cannot have   $J_{\bE'}\subset J_{\bE}$.  The arguments to prove this are somewhat delicate.  Hence we will begin the discussion
 by using  Proposition~\ref{prop:disj}  to deduce
 the main results stated in \S1.  We begin with a simple corollary of 
 Proposition~\ref{prop:disj}.

 \begin{cor}\label{cor:disj1} \begin{itemlist}\item[{\rm (i)}]  Every perfect class ${\bE'}$ with center in $(7,\infty)$ is derived by mutation from one of the basic generating triples 
 $\bigl(\bB^U_n, \bE_{[2n+7;2n+4]}, \bB^U_{n+1}\bigr),\ n\ge 0$ and so is a member of the complete family $\Cc\Ss^U$.

\item[{\rm (ii)}] Every other perfect class with center $> a_{\min}$ is the image of a perfect class in $\Cc\Ss^U$ by a symmetry  $S^iR^\de$, for some $i\ge 0,\ \de\in \{0,1\}$ with $i+\de>0$.
 \end{itemlist}
 \end{cor}
\begin{proof}  Consider the union $\Jj$ of the intervals $J_\bE$ that are blocked by some class $\bE$ in the complete staircase family $\Cc\Ss^U$.  
Since all these classes are perfect by Corollary~\ref{cor:perfATF}, Proposition~\ref{prop:live} implies that all the associated pre-staircases are live.  Hence, by Corollary~\ref{cor:perfATF2},  both
endpoints of the corresponding blocked $b$-interval $J_\Ee$ are unobstructed.\footnote
{
In the proof that the classes are perfect we used the fact that the lower endpoint of $J_\bE$ is unobstructed;  now we have a similar result for the upper endpoint.}
  We proved in Proposition~\ref{prop:dense} that $\Jj$ is an open dense set of $\acc_U^{-1}([6, \infty)) = [5/11, 1)$.   Therefore, the interval $J_{\bE'}$ must have nonempty intersection with    $\Jj$.  But this is possible only if
  $J_{\bE'}$ is contained in a component $J_\bE$ of $\Jj$.  Proposition~\ref{prop:disj} then shows that   $J_{\bE'} = J_{\bE}$, which implies that $I_{\bE'} = I_\bE$.  But then the two classes have the same breakpoint (since this is the point of shortest length in $I_{\bE}$, see Lemma~\ref{lem:pq}),  and hence coincide by the uniqueness result in \cite[Lem.2.2.1]{MM}.  This proves (i).
 
 To prove (ii), notice that, as illustrated in Figure~\ref{fig:symm}, the set of $z$-values blocked by
 perfect classes with $m/d>1/3$ and in some family $(S^iR^\de)^{\sharp}(\Cc\Ss^U)$ with $i+\de$ even is dense in the interval $(a_{\min},\infty)$. Hence the set of blocked $b$-values is also dense, and it follows as above that there are no other perfect classes with $m/d>1/3$.  A similar argument applies to perfect classes with $m/d< 1/3$.
\end{proof}

\begin{proof}[Proof of Theorem~\ref{thm:main1} (iii)] By Corollary~\ref{cor:disj1} (ii), it suffices to consider perfect classes in the complete family $\Cc\Ss^U$. Besides the classes $\bB^U_n$, each perfect class in $\Cc\Ss^U$ is the middle entry in a unique triple and so a step in both the ascending staircase $\Ss^\Tt_\ell$ (live by Proposition~\ref{prop:live}) and the descending staircase $\Ss^\Tt_u$ (live by Corollary~\ref{cor:live1}). For $n\geq1$, $\bB^U_n$ is a member of both the descending staircase $\Ss^{\Tt^{n-1}_*}_u$ and the ascending staircase $\Ss^{\Tt^n_*}_\ell$. The class $\bB^U_0$ is a member of the ascending staircase $\Ss^{\Tt^0_*}_\ell$ and the descending staircase $\Ss^{(SR)^\sharp\Tt^0_*}_u$.
\end{proof}

\begin{proof}[Proof of Proposition~\ref{prop:1}]  We must show that ${\it Block}\subset [0,1)$ is the disjoint union of the sets $J_\bE$ as $\bE$ ranges over all perfect classes with centers $> a_{\min}.$  By Corollary~\ref{cor:disj1}, the class $\bE$ must belong to one of the families $(S^iR^\de)^{\sharp}(\Cc\Ss^U)$, and hence its endpoints are unobstructed.  Therefore all these sets $J_\bE$ are disjoint by Proposition~\ref{prop:disj}. 
\end{proof}

The statements in the first three parts of Theorem~\ref{thm:main} are also now easy to prove. 
\MS

\begin{proof}[Proof of Theorem~\ref{thm:main} parts (i), (ii), (iii)]
We must   show that
\begin{itemlist}\item[{\rm (i)}]  {\it Block} is an open dense subset of $[0,1)$ that is invariant under the action of the symmetries. 
\MS

\item[{\rm (ii)}] There are staircases at each end of each connected component of {\it Block}.\MS

\item[{\rm (iii)}]   For $n\ge 0$ define 
\begin{align*}
{\it Block}_{[2n+6,2n+8]}: = \bigl\{ b \in {\it Block} \ \big| \ \acc(b) \in [2n+6,2n+8], \ b>1/3\bigr\}.
\end{align*}   
For each $n\ge 0$ there is a homeomorphism of ${\it Block}_{[2n+6,2n+8]}$ onto the complement $[-1,2]\less C$ of  the middle third Cantor set $C\subset [0,1]$.
\end{itemlist}

 By Propositions~\ref{prop:1} and ~\ref{prop:disj}, 
 $$
 {\it Block} = \bigcup\bigl\{ J_\bE \ \big| \ \bE\in (S^iR^\de)^\sharp(\Cc\Ss^U), i\ge 0, \de\in \{0,1\}\ \big\}.
 $$
 Since each class $\bE$ in $\Cc\Ss^U$ is the center of a generating triple, it has two associated staircases that converge to the end points of $\p J_\bE$.  
 Therefore (ii) holds for the family $\Cc\Ss^U$, and hence it also holds for the image of this family under any symmetry.    Next notice that (i) follows from Proposition~\ref{prop:dense} and the fact that the image of the interval $[6,\infty)$ under the  symmetries $\bigl\{S^iR^\de: i+\de\; \mbox{ is even}\bigr\}$ is dense in  
 $(a_{\min}, \infty) = \acc\bigl((1/3,1)\bigr)$, while the image under the  symmetries
$\bigl \{S^iR^\de: i+\de\; \mbox{ is odd}\bigr\}$  is dense in $(a_{\min}, \tau^4)=\acc\bigl((0,1/3)\bigr)$.  
Finally, we construct the homeomorphism in (iii) in the case $n=0$  so that it takes the half-open intervals 
 $I_{\bB^U_0}\cap [6,8]$ and $I_{\bB^U_1}\cap [6,8]$ onto $[-1,0)$ and $(1,2]$ respectively, and more generally, takes the elements of $I_{\bE_\de}$ onto the interior of the
interval consisting of all $x\in [0,1]$ whose ternary decimal expansion starts with the entries in the decimal $\de$.   Here we are using the decimal notation introduced in Lemma~\ref{lem:label}.  Thus, 
$\bE_{[7;4]} = \bE_{.1}$ has image $(1/3,2/3)$ while   
$\bE_{[7;5,2]} = \bE_{.01}$  has image $(1/9,2/9)$.  The definition for the cases $n>0$ is analogous.
\end{proof}

It remains to prove Proposition~\ref{prop:disj}.
Our first lemma describes useful properties of the elements $p'/q'\in I_{\bE} : = \acc(J_{\bE})$.

\begin{lemma}\label{lem:pq}  Let $\bE = (d,m,p,q)$.  Then for all  $p'/q'\in I_{\bE}$ we have
$p'\ge p$ and $ q'\ge q$. 
\end{lemma}
\begin{proof}    It is proved in  \cite[\S2]{AADT} (see also \cite[\S2.2]{ball}, or \cite[Lem.14]{ICERM}) that  $\ell_{wt}(p/q)< \ell_{wt}(p'/q')$ for all (rational) $p'/q'\in I_{\bE} \less \{p/q\}$.  In other words $p/q$ is the unique point of shortest (weight) length in $I_{\bE}$.   Suppose that $p'/q' \in I_{\bE}$ and write 
$$
CF(p/q) = [s_0;s_1,\dots, s_n],  \quad CF(p'/q') = [s'_0;s'_1,\dots, s'_{n}, \dots] \quad \ s_n \geq 2.
$$

We claim that
\begin{itemlist}\item  {\it $s_i = s_i'$ for all $i<n$. }

This holds because otherwise at least one of the points $[s_0;s_1,\dots, s_i + \eps]$, where $\eps = 1, 0, -1$ would lie strictly between $p'/q'$ and $p/q$ and hence be in the interval $I_{\bE}$, even though its length is $<\ell_{wt}(p/q)$.  For example,\footnote{
See the beginning of Section~\ref{ss:arith} for further information on ordering continued fractions.} 
 if $p/q= [1;3,1,2]$ and $p'/q' = [1;2,8]$  then $p/q < [1;3] <[1;2,8]$.

 \item {\it $s'_n \ge s_n-1$, and if  $s'_n = s_n-1$ then $\sum_{i>n}s_i'\ge 2$. Hence $p'>p,q'>q$.}
If $s'_n \le s_n-2$ we can argue as before that the point   
$[s_0;s_1,\dots, s_n-1]$  lies strictly between $p/q$ and $p'/q'$.  Further we cannot have 
$p'/q' = [s_0;s_1,\dots, s_n-1]$ (since this has shorter weight 
length than $p/q$), and so the sum of the subsequent entries 
$\sum_{i>n}s_i'$ must be at least two by our convention that the last entry in any continued fraction is at least $2$. 
\end{itemlist}  
Thus $CF(p'/q')$ has the stated form.  Therefore at least one of the  entries of $W(p'/q')$ that correspond to the last block $1^{\times s_n}$ of $W(p/q)$ is $> 1$.  It now follows from the recursive definition of the weight sequence (where we construct it starting from the last block) that each entry in  $W(p'/q')$ is at least as large as the corresponding entry in $W(p/q)$.  The result now holds because $p/q \ne p'/q'$.
\end{proof}

\begin{lemma} \label{lem:md1}
 Let $\bE, \bE'$ be distinct perfect classes such that $J_{\bE'}\subset J_\bE$, and suppose also that both points in $\p J_\bE$ are unobstructed.  Then:
 \begin{itemlist}
\item[{\rm (i)}]  We must have
 $p'/q'\in I_\bE$ and $m/d<m'/d'$.
 \item[{\rm (ii)}]  Further $p/q< p'/q' $.
 \end{itemlist}
\end{lemma}

\begin{proof}
Because $\bE'$ is center-blocking by \cite[Prop.2.2.9]{MM}, $\acc_U^{-1}(p'/q')\in J_{\bE'}\subset J_\bE$ so that
$p'/q'\in I_{\bE'}\subset I_\bE$, as claimed.  Further, because  $m'/d' > J_{\bE'}$, we know that $m'/d' > \p^-(J_{\bE})$. 
\MS

Now suppose that $p'/q'< p/q$.  Because $\p J_\bE$ is unobstructed, $\mu_{\bE, b^-}$ is live on the interval $z\in I_{\bE}, z\le p/q$ for $b^-: =  \p^-(J_\bE)$ by \cite[Prop.42]{ICERM}. 
Therefore  we must have $\mu_{\bE',b^-}(z) \le \mu_{\bE,b^-}(z)$ for $b=b^-$ and $z$ close to $p'/q'$.  Moreover, because $\mu_{\bE',b^-}(z) $ is constant for $z>p'/q'$ while $\mu_{\bE,b^-}(z)$ is not,  this inequality is  strict  when $z> p'/q'$.   On the other hand, Lemma~\ref{lem:pos} shows that $\mu_{\bE', b}$ dominates $\mu_{\bE, b}$  for $z$ near $p'/q'$ when $b$ is sufficiently close to $m'/d'$.    Therefore $\mu_{\bE',b}(p'/q') $ must grow faster than $\mu_{\bE,b}(p'/q')$ as $b$ increases to $m'/d'$.  Hence Lemma~\ref{lem:md}~(ii) shows that we must have $m'/d'> m/d$ as claimed in (i). 
\MS

We next prove (i) when  $p'/q'>p/q$.   If also $m'/d' \le m/d$, then the obstruction $\mu_{\bE',b}(p'/q')$ changes no faster than $\mu_{\bE,b}(p'/q')$.  Also we have 
\begin{align*}
&\mu_{\bE',b}(p'/q')\le  \mu_{\bE,b}(p'/q'),\quad \mbox{ when } b = \p^+J_{\bE}, \quad \mbox{ while }\\
&\mu_{\bE',b}(p'/q') >  \mu_{\bE,b}(p'/q'),\quad \mbox{ when } b \approx m'/d'.
\end{align*}
It follows that we must have $m'/d'<  \p^+J_{\bE}$, since otherwise the obstruction $\mu_{\bE',b}(p'/q')$, which dominates at $m'/d'$, still dominates as $b$ decreases from $m'/d'$ to $\p^+J_{\bE}$ since it decreases more slowly than $\mu_{\bE,b}(p'/q')$ as $b$ decreases. (Note that $m'/d'$ is rational and so cannot equal $\p^{+} J_\bE$, which is irrational.) 
Next observe that because $p/q < p'/q'$ we have
 $$
 0\le \bE\cdot \bE' = dd'-mm'-qq' \bw(p/q)\cdot \bw(p'/q') \le dd'-mm'-pq'.
 $$
 by Lemma~\ref{lem:WW}.

Thus we find that
\begin{align*}
 \mu_{\bE,m'/d'}(\acc(m'/d'))& = \frac p{d- m \frac{m'}{d'}} \\
& > V_{m'/d'}(\acc(m'/d')) \quad \mbox{since $\mu_{\bE,m'/d'}$ blocks $m'/d'\in J_\bE$}\\
&= \frac{1 + \acc(m'/d')}{3- m'/d'} 
 > \frac {1+p'/q'}{3- m'/d'}
\end{align*}
where the last inequality uses the fact that $p'/q' < \acc(m'/d')$; see Remark~\ref{rmk:triple0}.
This simplifies to the strict inequality
$$
\frac{q'p}{dd'-mm'} > \frac{p'+q'}{3d'-m'}.
$$
But $ \frac{p'+q'}{3d'-m'}=1$, while $dd'-mm'\ge pq'$ by positivity of intersections.
Hence this is impossible, so that we must have $m'/d'> m/d$.  This proves (i).\MS

It is straightforward to check that the condition $m'/d'>m/d$ implies that $t'(p+q) > t(p'+q')$, i.e. 
$$
((p')^2  - 6p'q' + (q')^2 + 8)(p+q)^2 > (p^2  - 6pq + q^2 + 8)(p'+q')^2.
$$
This simplifies to
$$
(p'+q')^2 pq - p'q'(p+q)^2 > (p'+q')^2 - (p+q)^2,
$$
or equivalently
$$
(pp' - qq')(p'q-q'p) > (p'+q')^2 - (p+q)^2.
$$
If $p'/q' < p/q$ then $p'q-q'p< 0 $.  On the other hand, Lemma~\ref{lem:pq}~(i) implies that $p'+q'> p+q$, and $pp'-qq'>0$ because $p>q, p'>q'$.  Hence we must have $p/q< p'/q'$.  This proves (ii).  
\end{proof}

\NI \begin{proof}[Proof of Proposition~\ref{prop:disj}]
Since the endpoints of $J_{\bE}$ are unobstructed, this interval must either be disjoint from $J_{\bE'}$ or contain it.  If the latter,  
Lemma~\ref{lem:md1} shows that it suffices to consider the case when   $p/q<p'/q'$ and $m/d<m'/d'$. Since $\mu_{\bE',b}$ is live at $p'/q'$ when $b = m'/d'$ while 
$\mu_{\bE,b}$ is live at $p'/q'$ at the smaller value $b_\infty = \p^+J_{\bE}$, there is some $b\in (b_\infty, m'/d')$ where the two obstruction functions agree.   Thus, we have
$$
\frac{p}{d-mb} = \frac{p'}{d'-m'b},\quad i.e.\quad  b= \frac{p'd-pd'}{p'm - pm'}.
$$
If  $p'd-pd'>0$ then $p'm - pm'>0$ and we have
\begin{align*}  
\frac{p'd-pd'}{p'm - pm'}< \frac{m'}{d'}&\; \Longrightarrow  \; p'dd'-p(d')^2 < mm'p' - p(m')^2\\
&\; \Longrightarrow  \;p'(dd'-mm')< p((d')^2 - (m')^2)\\
&\; \Longrightarrow  \; p'pq'< p(p'q'-1) \;\;\mbox{ by } \eqref{eq:weight}, 
\end{align*}
which is impossible. Hence $
p'd-pd'<0, \;\; p'm - pm'< 0$
 which implies $$
 p'/p < d'/d,\qquad  p'/p < m'/m.
 $$
    With $\la: = p'/p$, write
$d'= (\la +\eps) d, m'=(\la + \eps')m$, where $\eps,\eps'>0$, and notice that  because $p'/q'> p/q$ we also have $q' = (\la -\eps'')q$ for some $\eps''>0$. 
Then
$$
3(\la d + \eps d) = \la p + \la m +\eps'm + \la q -\eps''q
$$
so that $3\eps d = \eps' m - \eps''q$, which implies that  
$\eps' m > 3\eps d$.  Further,
\begin{align*}  
\frac{pd'-p'd}{pm' - p'm} = \frac{p(\la d + \eps d) - \la p d }{p(\la m + \eps'm) - m\la p} =\frac{ \eps d}{\eps'm}.
\end{align*}
Thus, because $\eps' m > 3\eps d$, we find that
$$
b_\infty < \frac{ \eps d}{\eps' m}  < \frac{ \eps d}{3\eps d}  = \frac 13
$$
which is impossible since $b_\infty > 1/3$.  This completes the proof.
\end{proof}

 \begin{proof}[Proof of Proposition~\ref{prop:liveZ}~(i), (ii)]
 Claim (i) is proved in Corollary~\ref{cor:5T}.  To prove (ii) it remains to consider the possibility that
 all but finitely many of the steps $(\bE_k)_{k\ge 0}$ of some descending pre-staircase $\Ss^-_{\al_\infty}$ are not live at the limiting $b$-value $b_{\al_\infty}$.  Since there are no overshadowing classes by Lemma~\ref{prop:6T}, this can happen only if there are infinitely many exceptional classes $(\bE'_k)_{k\ge 0}, $ that are live at $b_{\al_\infty}$,  each of which  obscures a finite number of the step centers of $\Ss^-_{\al_\infty}$.
But then this new set of classes  $(\bE'_k)_{k\ge 0},$ forms a staircase for this $b$-value.  Indeed, the break points of  
these obstructions must converge to $\al_\infty$ by \cite[Thm.1.13]{AADT}.
 \end{proof}
 
 \begin{rmk}\rm It seems likely that all but finitely many of the steps in $\Ss^-_{\al_\infty}$ are live at 
 $b_{\al_\infty}$.  However, we have not analyzed the properties of exceptional but non-perfect classes in enough detail to be able to make this claim.    \hfill$\er$
 \end{rmk}

 \subsection{The special rational $b$}\label{ss:13stair}
 By arguing as in Remark~\ref{rmk:symm0}~(iii), one can show that if $b_i$ is a special rational value with $\acc(b_i) = v_i$ as in \eqref{eq:bi} then none of the perfect classes with centers $> a_{\min}$ are obstructive at $b_i$.  However, there are other exceptional classes that affect the capacity function.
 For example  $\bE = 3L - E_0 -  2E_1 - E_{2\dots6} = (3,1;\bbm)$ with $\bbm = (2,1^{\times 5})$ is such a class, 
 which seems to be live for $1/5<b<5/11$ on various intervals above the accumulation point $\acc(b)$; see \cite[Rmk.2.3.8]{MM} for a discussion of its properties.   Thus there could be some as yet undiscovered staircases that accumulate at these $b_i$ from above.

  We now show that these points $(b_i,v_i)$ cannot be  limits of ascending staircases.    
The proof hinges on the properties of the third strand\footnote
{
See Example~\ref{ex:SsU} for the definition of a strand of a staircase.}  of the staircase at $b=1/3$ that accumulates at $a_{\min}$.
This staircase is discussed in detail in \cite[Example 2.3.7]{MM}; see also Remark~\ref{rmk:quasitrip}.
It has a rather different structure from the staircases with accumulation points $>a_{\min}$ discussed above, since it has three intertwined strands, all ascending but with alternating values of $\eps$, so that successive steps need  not be adjacent. Moreover, the three strands cannot be assembled into one ascending, recursively defined staircase. 

We are mostly concerned here with the third strand that has initial seed 
$\bE_{1}=(2,0,5,1)$ with center $g_1/g_0 =5/1$ and $(t,\eps) = (2,-1)$, and subsequent steps 
$\bE_i, i\ge 2,$ with 
\begin{align}\label{eq:3rd}
\mbox{centers } \;\; S^{i-1}(g_1/g_0) = g_i/g_{i-1} ,\quad t=2,\quad \eps = (-1)^{i}, i\ge 2
\end{align}
As usual the degree coordinates $d,m$ are determined by \eqref{eq:pqt}; see \cite[eq.(3.3.2)]{MM}.
Since there seems to be no convenient reference in the literature we begin with the proof that
all the classes in this staircase are perfect.
We also show that these are the only perfect classes with centers $< a_{\min}$.

\begin{lemma} \label{lem:13perfect}
The classes  that form the steps of the staircase of $H_{1/3}$ are perfect classes. Moreover these are the only perfect classes with centers $< a_{\min}$.
\end{lemma}
\begin{proof}
As shown in \cite[Example 2.3.7]{MM}, the three seed classes for $H_{1/3}$ are 
\[ \bE_{seed,0}=(1,0,2,1, 1,-1), \quad \bE_{seed,1}=(2,1,4,1, 1, 1), \quad \bE_{seed,2}=(2,0,5,1,2,-1).\]
The centers of the subsequent classes are given by applying $S$, so that the
 next classes in the first two strands are $(5, 2,11,2, 1,1)$ and $(10,3,23,4,1,-1)$. 
One can easily check that these classes Cremona reduce.\footnote{An integral class $\bE$ such that $c_1(\bE) = 1, \bE\cdot\bE = -1$ is exceptional if and only if it Cremona reduces. This is a transformation on the coefficients of the classes, which for a specific class is easily computed. More details about this process can be found in \cite[Prop.1.2.12]{ball} and \cite[\S4.1]{MM}.}
Thus, the initial classes are perfect. Furthermore, the rest of the classes $\bE_j$ are formed by applying the shift $S$ to $\bE_j$. 
We then apply \cite[Lem.4.1.2]{MM}, which states that $S$ preserves Cremona equivalence provided that the centers of the classes considered are $\ge 5$.  Thus the rest of the classes $\bE_j$ also Cremona reduce, and hence are perfect. 

Now suppose that $\bE = (d,m,p,q)$ is
a perfect class with center $< a_{\min}$. 
Taking $b = 1/3$ and using \eqref{eq:muEb},
 we obtain the inequality 
 \begin{align}\label{eq:perf13}
c_{H_{1/3}}(p/q) \ge \frac {3p}{3d-m} = \frac {3p}{p+q} = \frac {3z}{1+z}.
\end{align}
On the other hand, the function $c_{H_{1/3}}(z), z\in [1,a_{\min}],$ was fully calculated in \cite[Prop.1.19]{AADT}.  It is piecewise linear with outer corners  at
the step centers on the graph of $z\mapsto  \frac {3z}{1+z}$, and inner corners lying strictly below this graph.  Therefore  the class $\bE$ must have its center $p/q$ at one of the steps of this staircase. Since perfect classes are determined by their centers by \cite[Lem.2.2.1]{MM}, this implies that $\bE$  must be one of these steps.
\end{proof}

We saw in \cite[Ex.22]{ICERM} that 
the obstruction function given by  the class $\bE_1 = (2,0,5,1)$ is constant and equal to 
$\mu_{\bE_1,b}(z) = 5/(2-b)$ for $z>5$, and
 when evaluated at $b = 1/5$  goes through the point $(6, 5/2) = \bigl(6,V_{1/5}(6)\bigr)$.  Therefore it obstructs the existence of an ascending staircase at $z=6, b=1/5$.  One might wonder 
how the function $\mu_{\bE_1,b}$ behaves 
 when $b\approx 1/5$   since we know from \cite{ICERM} that  there are sequences of values of $b$ that converge to $1/5$ from both sides that do admit ascending staircases.\footnote{These are the staircases in the $\Ss^E$ and $\Ss^L$ families constructed in \cite[Thm.~54,~58]{ICERM}. This behavior generalizes to the other special rational $b$ as shown in \cite{MM}.}  However it turns out that for all $b\ne1/5$ 
 the graph of the function $z\mapsto \mu_{\bE_1,b}(z)$, which is constant for $5<a_{\min}<z$, meets the volume curve $(z, V_b(z))$ at a point $z_0(b)$ that is strictly $< \acc(b)$ and hence is not obstructive.  We now apply $S$ to show that a similar phenomenon occurs at all the special points $v_i, i\ge 2$.  
 
 \begin{prop}\label{prop:noasc}   For each $i\ge 1$,  the class $\bE_i$ in \eqref{eq:3rd} is not a blocking class, though its obstruction does go through the point  $(\acc(b), V_b(\acc(b))$ where $b$ is the special rational point $b_{i+1} =  \acc_\eps^{-1}(v_{i+1}),$  where $\eps = (-1)^i$. 
In particular, the special rational $b$ have no ascending staircase.
\end{prop}

\begin{rmk}\rm  By \eqref{eq:3rd}, the class  $\bE_i$ has $m_i/d_i > 1/3$ exactly if $i$ is even 
so that the corresponding rational $b$ is  $b_{i+1}> 1/3$.  \hfill$\er$  \end{rmk}

We begin the proof with the following lemma.

\begin{lemma}\label{lem:blockz}  Let $\bE = (d,m,p,q)$ be a perfect class such that
$\mu_{\bE,b}(z)$ is constant on the interval $[p/q, p'/q']$  and suppose that  $p/q< \acc(b) < p'/q'$ for some $b$.
Then $\bE$ blocks  $b$  only if $z^\bE_b> \acc(b)$ where 
\begin{align}\label{eq:blockz}
z^\bE_b = \frac{(3p-d) - b(p-m)}{d-mb}.
\end{align}
\end{lemma}
\begin{proof}  Our assumptions imply that 
the function $(b,z)\mapsto\mu_{\bE,b}(z) = \frac{p}{d-mb}$ does not depend on $z$ near $z=\acc(b)$, 
so that  we can think of $\mu_{\bE,b}(z) $  as a function of $b$ only. 
For each $b$ the graph of $z\mapsto V_b(z) = \sqrt{\frac z{1-b^2}}$ meets  the graph of $z\mapsto \frac{1+z}{3-b}$  at the point $z=\acc(b)$.  It is easy to check that at this point the line has larger slope than the volume curve.  
  Thus  if $\bE$ blocks $b$, the point $z^\bE_b$ where  $ \frac{p}{d-mb} = \frac{1+z}{3-b}$ must be larger than $\acc(b)$.  But this point is precisely given by the formula in \eqref{eq:blockz}.  \end{proof}

\begin{proof}[Proof of Proposition~\ref{prop:noasc}] Notice first that  $g_{i}/g_{i-1} < v_{i+1} = y_{i+1}/y_{i}$ for all $i\ge 1$.  Indeed this holds when $i=1$ since $g_1/g_0 = 5/1$ and $v_2 = 6/1$, 
and therefore it holds for all larger $i$ since the shift $S$ preserves orientation and increases the index $i$ by $1$ on each side. 
A similar argument shows that  $g_i = y_{i+1}-y_i$.   We next claim that
 $\ell_{wt}(g_{i}/g_{i-1}) < \ell_{wt}(y_{i+1}/y_i)$ for all $i$, since it holds when $i=1$ and the calculation
 $$
S\bigl([5;x]\bigr) = S\bigl(\frac{5x+1}x\bigr) = \frac{29x+6}{5x+1} = [5;1,4,x],
$$
where $x = [x_0; x_1,\dots]\ge 0$, shows that 
  applying $S$ increases the weight length of any two numbers $z,z'\in [5,6)$ by the same amount.
 Since $g_i/g_{i-1}< a_{\min}$, the hypotheses of
Lemma~\ref{lem:blockz} holds. 
Therefore the obstruction from the perfect class $\bE_{i} = (d_{i }, m_{i}, g_{i}, g_{i-1})$ in \eqref{eq:3rd} with center $g_{i}/ g_{i-1}$
 is constant near $v_{i+1}$.
 Hence by Lemma~\ref{lem:blockz} it suffices to show that 
 $z^\bE_b\le \acc(b)$ where 
\begin{align*}
z^\bE_b& = \frac{(3(y_{i+1}-y_i) - d_i) - b((y_{i+1}-y_i) - m_i)}{d_i - m_i b}.
\end{align*}
By \eqref{eq:3rd},  $8d_i = 3(y_{i+1}-y_{i-1}) + \eps,\ 8m_i = (y_{i+1}-y_{i-1}) + 3\eps$, where $\eps = (-1)^i$ is as in \eqref{eq:bi}.  Further because $\acc(b)$ is the solution to the equation
$$
\acc(b) + \frac 1{\acc(b)} = (3-b)^2/(1-b^2) - 2,
$$
  $z^\bE_b\le \acc(b)$  exactly if
\begin{align*}
z^\bE_b + \frac{1}{z^\bE_b} \le (3-b)^2/(1-b^2) - 2.
\end{align*}
That is, if 
  \begin{align*}
\frac{( z^\bE_b + 1)^2}{z^\bE_b} \le \frac{(3-b)^2}{1-b^2}
\end{align*}

Since
  \begin{align*}
z^\bE_b + 1 &= \frac{3(y_{i+1}-y_i) - d_i) - b((y_{i+1}-y_i) - m_i) + d_i - m_i b}{d_i - m_i b}\\
& = \frac{3(y_{i+1}-y_i)- b(y_{i+1}-y_i)}{d_i - m_i b} = 
 \frac{(3-b)(y_{i+1}-y_i)}{d_i - m_i b},
\end{align*}
this  simplifies to the inequality
  \begin{align}\label{eq:horrid}
(d_i - m_i b)\Bigl(3(y_{i+1}-y_i) - d_i - b((y_{i+1}-y_i) - m_i)\Bigr) - (y_{i+1}-y_i)^2 (1-b^2)\ge 0.
\end{align}
We claim that this inequality is equivalent to $(b-b_{i+1})^2 \ge 0$  where $b_{i+1}$ is as in 
\eqref{eq:bi}.  Thus we must show that for some constant $C$, the 
coefficients of $b^2, b, 1$ in  \eqref{eq:horrid} are respectively
$$ 
 C(3y_{i+1} + 3y_i + \eps)^2, \quad -2C(3y_{i+1} + 3y_i + \eps)(y_{i+1} + y_i + 3\eps), \quad 
C(y_{i+1} + y_i + 3\eps)^2.
$$
Since $d_i,m_i$ are given by linear expressions in the $y_i$ and the $y_i$ are even placed Pell numbers,
these are quadratic expression in the Pell numbers.  But by 
\cite[Lem.6.6]{FM} these hold in general if and only if they hold for three distinct values of $i$.  (Note that because the $y_i$ correspond to even placed Pell numbers the sign $(-1)^s$ in this result is always $= 1$.)
 But when $i = 1,2,3$ the quantities $ (d_i,m_i, y_{i+1},y_i, \eps)$ are equal to 
 $$
 (2,\ 0,\ 6,\ 1,\ -1), \quad (13,\ 5,\ 35,\ 6,\ 1),\quad (74,\ 24,\ 204,\ 35,\ -1)
$$
and it is straightforward to check that the required identities hold with $C = 1/16$. This completes the proof.
\end{proof} 
 
 \begin{cor}\label{cor:13stair}  Theorem \ref{thm:main}~(iv) holds.
 \end{cor}
 \begin{proof}  This states that the only rational numbers that can be accumulation points of staircases are 
 the points $v_i, i\ge 2$, and if there is such a stair it must descend.  The first claim is proved in 
Corollary~\ref{cor:short2}
while the second is proved in Proposition~\ref{prop:noasc}.  \end{proof}

 \begin{rmk} \rm
 Proposition~\ref{prop:noasc} only considers the third strand of the staircase at $H_{1/3}$. The other two strands of the staircase at $H_{1/3}$ have seeds $\bE_{seed,0}=(1,0,2,1)$ and $\bE_{seed,1}=(2,1,4,1).$ The rest of the classes are determined by taking $(S^k)^\sharp(\bE_{seed,i}).$ The classes in these two strands are also not blocking classes, but unlike the third strand, the obstructions from these strands do not go through the point $\bigl(\acc(b),V_b(\acc(b)\bigr)$ for any $b.$ Thus, these perfect classes do not obstruct either ascending or descending staircases for any value of $b.$  \hfill$\er$
 \end{rmk}
 
 The proof that $b_2=1/5$ has no descending staircase in \cite[Thm.94]{ICERM} shows that for some $\de>0$ and $z\in[6,6+\de)$, the capacity function $c_{1/5}(z)$ equals the obstruction from the 19th ECH capacity, or equivalently, from the nonperfect exceptional class $(d,m;\mathbf{m})=(5,1;2^{\times6},1)$. Experimental evidence suggests that the class $(3,1;2,1^{\times5})$ plays this role for all higher $i$.
 
 \begin{conjecture}\label{conj:specialdesc} For $i>2$, $c_{H_{b_i}}$ has no descending staircase. Furthermore, for some $\de>0$ and $z\in[v_i,v_i+\de)$, the capacity function $c_{H_{b_i}}(z)$ equals the obstruction from the 8th ECH capacity, or equivalently, from the nonperfect exceptional class $(3,1;2,1^{\times5})$.
 \end{conjecture}

\subsection{Stabilization}\label{ss:stab} 

We now show that all the staircases found in this paper stabilize.  

Consider the {\bf stabilized embedding function }
$$
c_{H_b,s}(z) = \inf\ \bigl\{ \la\ \Big| \ E(1,z)\times \R^{2s} \se \la H_b\times \R^{2s}\bigr\} ,\quad s\ge 0.
$$
We always have $c_{H_b} (z) = c_{H_b,0} (z) \ge c_{H_b,1}(z) \ge c_{H_b,2}(z) \ge \dots$ because by taking the product with the identity $\id_{\R^2}$ any embedding 
$\io:  E(1,z)\times \R^{2s} \to \la H_b\times \R^{2s}$ extends to an embedding 
$\io\times \id_{\R^2}:  E(1,z)\times \R^{2s+2} \to \la H_b\times \R^{2s+2}$.

\begin{rmk}\rm We will base our discussion on the arguments in \cite{CGHM} that consider the stabilized embedding function for $\C P^2$.  These arguments extend to the semipositive case,  that is to arbitrary six-manifolds and to monotone manifolds of any dimension.   Thus, unless $b=1/3$, we consider only the case $s = 1$.   However, this restriction is purely technical, and our results should hold for all $s$. \hfill$\er$
\end{rmk}

\begin{lemma}\label{lem:stab}  Let $\bE = (d,m,p,q)$ be a perfect class. 
Then,
$c_{H_b,s}(p/q) \ge \mu_{\bE,b}(p/q)$ for $s=1$ and all $b \in [0,1)$.
When $b = 1/3$, this holds for all $s\ge 1$.
\end{lemma}
\begin{proof}  The analogous result  was proved in \cite{CGH} in the case of the target $\CP^2$, that is when $b=0$. A rather different proof of this result was outlined in  \cite[Rem.3.1.8(ii)]{CGHM}.    The latter methods  extend almost immediately to the current situation.  One simply needs to replace the ambient manifold $\CP^2$ by the  blowup $H_b$ 
in the basic stabilization result \cite[Prop.3.6.1]{CGHM} that specifies when an immersed  curve  in a four-dimensional cobordism  of both genus and Fredholm index zero
persists under stabilization.\footnote
{
A detailed proof of a generalized version of this result that applies explicitly to the semipositive case will appear in a forthcoming paper by McDuff--Siegel~\cite{MSie}.} 
One important point is  that the statement in \cite[Prop.3.6.1]{CGHM} permits both the multiplication of the target by $\R^{2}$ and the perturbation of the symplectic form on the target.
This holds because the crucial compactness statement applies to a generic  path $J_t$ in the space $\Jj(T)$, where $(J_t)_{t\in [0,1]}$ is a path  of admissible almost complex structures that are compatible with a path $\om_t$ of symplectic forms on the target
In particular, one can change $b$ during this deformation. However, it is important that the form remain semipositive, since otherwise one would need to use virtual techniques to 
 control the possible degenerations.    Thus we must either fix $b=1/3$ or take $s=1$.
 
 To
 apply \cite[Prop.3.6.1]{CGHM}, it suffices to
 produce a 
 suitable curve in the negative completion $X_{b_0}$ of $\la H_{b_0}\less \io(E(1,p/q+\de))$ in class $dL - mE_0$ that has one negative end on the $p$-fold cover of the short orbit $\be_1$ on $\p E(1,p/q+\de)$, where  $b_0\in [0,1)$ is any suitable value and $\de$ is a suitably small constant.  This curve necessarily has genus zero because it is constructed by neck-stretching an exceptional sphere.

 We will construct this curve for $b_0 = m/d$, where $(d,m)$ are the degree coordinates of $\bE$, by using the method explained in 
  \cite[\S3.1]{CGHM}.  As we will see, the argument given there is much simplified by the fact that  when $b=m/d$ the class $\bE$ is live at $z=p/q$, so that
   the 
  capacity function $c_{H_{m/d}}(p/q)$ is equal to the obstruction $\mu_{\bE, m/d}(p/q)$. 
  Geometrically this means that, for any $\eps>0$, 
   there is a symplectic embedding $\io:E(1,p/q+\de)\to \la H_{m/d}$, where $\de>0$ is sufficiently small and
   $$
  \la: = c_{H_b}(p/q) +\eps = \frac {p}{d- m^2/d} + \eps=  \frac{pd}{pq-1} + \eps,
  $$
where we use Lemma \ref{lem:basic} (i) and the identities in \eqref{eq:muEb}.
  
  Let $(m_1,\dots,m_n) = (q,\dots,1)$ be the integral weight expansion of $p/q$. By 
  \cite[Prop.2.1.2]{CGHM}, for any $\de,\de'>0$ we may embed the disjoint union of $n$ balls of capacities $(1-\de')m_1,\dots,(1-\de')m_n$ into the interior of $\io(E(1,p/q+\de))$ and then blow them up to obtain a symplectic form $\Tilde{\om}$ on $H_{m/d}\# n {\ov \CP}\,\!^2$ in the class Poincar\'e dual to $\la (L - (m/d)E_0) - (1-\de') \sum_i m_i E_i$. We now consider what happens to the unique representative $C_\bE$ of the class $\bE$ in this blowup manifold as we stretch the neck around the boundary $\p(\io(E(1,p/q+\de)))$ of the ellipsoid.  As described in \cite[\S3.1]{CGHM}, the curve $C_\bE$ converges to a limiting building, whose  
 top component $C_U$ lies in the negative completion $X_{m/d}$ of $H_{m/d}\less \io(E(1,p/q+\de))$.
 Thus $C_U$ has negative end on some orbit set  $\{(\be_1,k)\}\cup \{ (\be_2,\ell)\}$, where 
 this notation means that the ends on the short orbit $\be_1$ have total multiplicity $k$, while
 those on the long orbit $\be_2$ have total multiplicity $\ell$.
 Our aim is to show that (when $\eps,\de,\de'$ are all sufficiently small) there is  only one possibility for $C_U$, namely $C_U$ must be connected, embedded, and have  one negative end on 
$\be_1$ of multiplicity $p$.   Granted this, the result follows from \cite[Prop.3.6.1]{CGHM}.

For all $\de, \eps>0$, the energy $
 \la (d - m^2/d) -  (k + \ell (p/q + \de))
 $
 of the top part $C_U$ is positive but arbitrarily small.  Hence, because $\la = \frac {p}{d-m^2/d} + \eps$, we must have the equality
 $$
 \frac {p}{d-m^2/d} (d - m^2/d) = k + \ell (p/q), \quad \mbox { i.e. } \ p = k + \ell p/q.
  $$
Since $p,q$ are relatively prime and $k,\ell \ge 0$, the only possibilities are: $k=p, \ell=0$ or $k=0, \ell = q$.   Further, if $C_U$ were disconnected or multiply covered, there would  be $d'<d, m'<m$ and either $k<p$ or $\ell<q$ such that
$$
\frac {p}{d-m^2/d} (d' - m' m/d) = k, \;\;  \mbox{ or } \;\; \frac {p}{d-mb} (d' - m' m/d) = \ell (p/q)
 $$ 
 Since $d^2-m^2 = pq -1$, the first of these equations simplifies to
 $p(dd'-mm') = k(pq-1)$, which implies $p|k$ and hence has no solution.  Similarly, the second equation implies that $q|\ell$, so that it also has no solution.
 Therefore $C_U$ must be connected and somewhere injective.
 
  We can now appeal to the properties of the ECH index $I(C)$.  As explained in the survey article \cite[\S3.4]{Hut},  this index $I(C)$ is a generalization of the quantity $c_1(A) + A\cdot A$ for a curve in a closed manifold of homology class $A$, and has the property that  $I(C)\ge 0$ for any immersed curve $C$. Moreover, by \cite[(2.2.28)]{CGHM} and \cite[(2.4)ff]{CGH},
  if the above curve $C_U$   has negative end on  $\{(\be_1,k)\}\cup \{ (\be_2,\ell)\}$ then
  $$
  I(C) = d^2  - m^2  + 3d - m - {\rm gr}(\be_1^k\be_2^\ell),
  $$
   where ${\rm gr}(\be_1^k\be_2^\ell)$ is twice the number of lattice points  $(m,n)$ that lie in the nonnegative quadrant of the plane and (strictly) below the line
   through $(k,\ell)$ with slope $-q/(p+ \de)$.
  Thus, when $\gcd(p,q) = 1$, 
  $$
  {\rm gr}(\be_1^p) = pq + p+q-1,\quad  {\rm gr}(\be_2^q) = pq + p+q+1.
  $$
  Since
  $$
 d^2  - m^2  + 3d - m =  pq -1 + p+q,
 $$
 there can be no curve in $X_{m/d}$ in class  $dL - mE_0$ with negative end on $\be_2^q$, since such a curve would have $I(C) < 0$.  Moreover,  curves in this class with negative end on 
$\be_1^p$ have  $I(C)=0$.  By \cite[\S3.9]{Hut}, this implies that it must have ECH partitions, which in the current situation means that it has one negative end: see for example \cite[Rmk.2.2.1ff]{CGHM}. The fact that $C_U$ has Fredholm index zero follows from \cite[(3.1)]{Hut}, \cite[(2.2.28)]{CGHM}, and the computation of the monodromy angle of $\beta_1$ (called $\gamma_1$) below \cite[(2.2.23)]{CGHM}.
This completes the proof.
\end{proof}

\begin{prop}\label{prop:stab} \begin{itemize}\item[\rm (i)] For all $b\notin {\it Block}$ except the special rational $b$,  the stabilized embedding function $c_{H_b,1}$ has a staircase. If, in addition, $b\notin \p({\it Block})$ and $b\neq1/3$, then $c_{H_b,1}$ has both ascending and descending staircases. 
\item[\rm (ii)] $c_{H_{1/3},s}(z) = c_{H_{1/3}}(z)$ for all $s\ge 1$ and all $1\le z\le a_{\min}$.
\end{itemize}
\end{prop}  
\begin{proof}  We always have $c_{H_b,s}(z) \le c_{H_b}(z)$, because any  embedding $\io: E(1,z)\to \la H_b$  extends to  an embedding $\io\times \id: 
E(1,z)\times \R^{2s}\to \la H_b \times \R^{2s}$.   Lemma~\ref{lem:stab} implies that if $H_b$ has a staircase with steps given by perfect classes $\bE_k$ with centers at the points $p_k/q_k$  then the stabilized function $c_{H_b,1}(z)$ must equal $c_{H_b}(z)$  at the step centers, and for $z\approx p/q$  have $c_{H_b}(z)$ as an upper bound.  
But  by the remarks about overshadowing classes before \eqref{eq:accform}, we know that  $c_{H_b}(z)= \mu_{\bE_k, b}(z)$ for $z\approx p/q$.  Thus
$c_{H_b}(z)$ is determined for $z\approx p/q$  by 
the scaling property  $c_{H_b,1}(\la z) \le \la  c_{H_b,1}(z)$ (as in \cite[Prop.2.1]{AADT}), and by monotonicity (i.e. the fact that $c_{H_b,1}(z)$ is nondecreasing).
Since $c_{H_b,1}(z)$ also has these properties for each $s$ it follows that
 $c_{H_b,1}(z) = c_{H_b}(z)$ in some neighborhood of each step center. 
Thus all the staircases that we have found stabilize. This proves (i).

The argument (i), now applied with any $s\ge 1$, shows that the staircase in $H_{1/3}$ stabilizes.  Thus for each $s$ 
$c_{H_{1/3},s}(z) = c_{H_{1/3}}(z)$ in a neighborhood of the centers of the steps of the $1/3$ staircase.  Since elsewhere 
$c_{H_{1/3},s}(z) \le c_{H_{1/3}}(z)$, it follows from  scaling and monotonicity
that these functions must coincide on the whole interval $1\le z\le a_{\min}$.
 This proves (ii).
\end{proof}

Although the calculation of $c_{H_b,s}$ for an arbitrary $b\in (0,1)$ seems out of reach at present,  the monotone case is more approachable.  Indeed, Hind's folding construction
in \cite{H} implies that  
\begin{align}\label{eq:Hind}
c_{H_{1/3},s}(z)\le \frac{3z}{1+z}, \quad z\ge 1, \ s\ge 1.
\end{align}
As in the case when the target is a ball, the graph of $z\mapsto \frac{3z}{1+z}$ crosses the volume constraint at the staircase accumulation point $a_{\min} = 3 + 2\sqrt2$, and one can conjecture that
\begin{align}\label{eq:stabconj}
 c_{H_{1/3},s}(z)=\frac{3z}{1+z},  \quad z\ge a_{min}, \ s\ge 1.
 \end{align}
 This is known as the {\bf Stabilized embedding conjecture} for $H_{1/3}$.  

\begin{cor} \label{cor:13StabConj} The stabilized embedding conjecture holds for the monotone Hirzebruch surface $H_{1/3}$  on the closure of the set of all points $z>a_{\min}$ that are the centers of perfect classes.
\end{cor}
\begin{proof} For all perfect classes $\bE = (d,m,p,q)$,
 we saw in the proof of Lemma~\ref{lem:13perfect} that
Proposition~\ref{prop:stab} and equation~\eqref{eq:perf13} imply that
$$
c_{H_{1/3},s}(z) \ge 
 \frac {3z}{1+z}, \quad \mbox{ when } z = p/q 
$$
Hence by \eqref{eq:Hind} we must have equality at these points.  The statement in the lemma then follows by continuity.  
\end{proof}

\begin{rmk}\label{rmk:13stab}\rm (i) The set of $z$ for which we know that
$c_{H_{1/3},s}(z) = \frac {3z}{1+z}$ has quite complicated structure because  the function $b\to \acc(b)$ is two-to-one; and it is probably best understood via Figure~\ref{fig:symm}.  For example, even though the only step center in the interval $[w_2,w_1] = [41/7,7]$ with $b>1/3$ is $6$, there are infinitely many step centers corresponding to classes with $b<1/3$ --- indeed all the steps centers  in the complete families $S^\#(\Cc\Ss^{U})$ and $R^\#(\Cc\Ss^{U})$.
Nevertheless, one can check that it is nowhere dense.
\MS

\NI  (ii) It may well not be true that, when $H_b$ has a staircase, $c_{H_b}(z) = c_{H_b,1}(z)$ for $z$ in some neighborhood of $z_\infty: = \acc(b)$ since the capacity function $c_{H_b}(z)$ might be determined near $z_\infty$ by curves that do not stabilize.
    For example, when $\acc(b) < 6$, figures such as \cite[Fig.5.3.1]{ICERM}  seem to show that the obstruction $z\mapsto \frac{1+z}{3-b}$ from the (nonperfect) class $\bE': = 3L-E_0 - 2E_1 - E_{2\dots6}$ is $b$-live for some values of $z>\acc(b)$  that 
are arbitrarily close to $\acc(b)$. However, 
because $\frac{1+z}{3-1/3} > \frac{3z}{1+z}$ when $z>a_{\min}$, it follows from
 \eqref{eq:Hind} that this obstruction does not stabilize when $b = 1/3$. Hence the curve that gives this embedding obstruction does not persist under stabilization for any $b$ 
 since  curves that do stabilize 
are not sensitive to changes in the parameter $b$.
  
An explicit example of this kind was worked out in \cite{CGHM} in the case of $H_0$ (or the ball) and the
 class $\bE'' = 3L- 2E_1 - E_{2\dots7}$.  This class, which obstructs  the capacity function $c_{H_0}$ of the ball for $z\in (\tau^4,7]$, definitely does not stabilize.  Indeed,  if one blows  $H_0$ up seven times inside the ellipsoid $E(1,7+\de)$ and  stretches the neck as in the proof of Lemma~\ref{lem:stab},  
one can show that the  top component of the resulting building has two negative ends, one on the long orbit and one on the short orbit, so that the stabilization result \cite[Prop.3.6.1]{CGHM} does not apply. 
Further, \eqref{eq:Hind} implies that  $c_{H_0,s}(7) \le 21/8$, while the obstruction from $\bE''$ at $z = 7$, if it did stabilize, would be the larger value $8/3$.
\MS

\NI
(iii) Here is another example of an obstruction that does not stabilize even though it is given by the \lq\lq nearly perfect\rq\rq\, class $\bE = 6L - 3E_0 - 2E_{1\dots 7}$.  (See Remark~\ref{rmk:except} for further discussion of this class.)
When $b = 1/3$, 
we have $\mu_{\bE,1/3}(7) = 14/5 > 21/8 = c_{H_{1/3},k}(7)$.  Therefore the obstruction from this class cannot stabilize.  Note that the proof of Lemma~\ref{lem:stab} breaks down  because we would be considering trajectories with negative end on the short orbit of $\p E(1,7+)$ with total multiplicity $14$,  and in this case the ECH partition is $(7,7)$.
Thus these curves do not have a single negative end.
  \hfill$\er$
\end{rmk}

\appendix
\numberwithin{equation}{section}

\section{Continued fractions}\label{app:arith}
For the convenience of the reader, we here collect together some useful facts about continued fractions.
Each rational number $a>1$ has a continued fraction representation $[s_0;s_1,\dots,s_n]$, where  $s_i$ is a positive integer, and $n>0$ unless $a\in \Z$. 
By convention, the last entry in a continued fraction is always taken to be $> 1$, since $[s_0;s_1,\dots,s_n, 1] = [s_0;s_1,\dots,s_n+ 1]$; for example
$$
[1;3,1] = 1 + \frac 1{3 + \frac 11} =  1 + \frac 1{4} = [1;4].
$$

\NI {\bf Order properties:}
If $n$ is odd, then $[s_0;s_1,\dots,s_n]> [s_0;s_1,\dots,s_n + 1]$, in other words, if the last place is odd,  increasing this  entry decreases the number represented.
For example, because $\frac 2{13}< \frac 3{19}$, we have
\begin{align*}
[1;3,6,2] = 1 + \frac 1{3 + \frac 1{ 6 + \frac 12}} = 1 + \frac 1{3 + \frac 2{13}}  >
[1;3,6,3] = 1 + \frac 1{3 + \frac 1{ 6 + \frac 13}} = 1 + \frac 1{3 + \frac 3 {19}}.
\end{align*}
Further, although increasing an even place usually  increases the number represented, the opposite is true
if one increases an even place\footnote{This place would necessarily be the one just after the last; and notice that the initial place is labelled $0$ and hence is considered to be even.} from $0$ to a positive number. For example 
\begin{align}\label{eq:ctfr4}
[1;4] = 1+\frac 14 > [1;4,2] = 1+\frac 1{4 + \frac 12} = 1 + \frac 29 > [1;4,1]= [1;5] = 1 + \frac 15.
\end{align}
Similarly, if one increases an odd place from $0$, the corresponding number increases even if the new entry $s_n$ is $1$.  Thus
\begin{align}\label{eq:odd}
[s_0;s_1,\dots,s_{n-1}]< [s_0;s_1,\dots,s_n] \quad\mbox { if $n$ is odd} \ \forall s_n\ge 1. 
\end{align} 
For example, $2 < [2;1,2] <[2;1] = 3.$\MS

\NI {\bf Length:} 
There are two natural notions of the  length of a continued fraction, namely
\begin{align}\label{eq:WCFlength}
\ell_{CF} ([s_0;\dots, s_n]) : = n+1,\quad \ell_{wt}([s_0;\dots, s_n]): = \sum_{i=0}^n s_i.
\end{align}
The first notion is common in the theory of continued fractions, while  in our previous papers we used the second notion since this describes the length of (or number of nonzero entries in) the corresponding 
weight expansion  $W(p/q)$ and hence plays a central role in the theory; witness the fact proved in \cite[Prop.2.30]{AADT} that if the  function $z\mapsto \mu_{\bE,b}(z)$ is obstructive (that is, greater than the volume) on some interval $I$, then there is a unique point\footnote
{
In the case of a quasi-perfect class, this point is the center; cf. the formulas in \eqref{eq:muEb}.}
  $z_0\in I$ of minimal weight length  at which its slope changes.
Since we will now use both these notions, for clarity we will call $\ell_{CF}(p/q) $ the {\bf CF-length} of $p/q$, while $\ell_{wt}(p/q)$ is its {\bf weight length}. 
\MS

Notice that by Remark~\ref{rmk:recur0}~(i), the accumulation points of our staircases are quadratic irrationals and therefore have infinite continued fractions that are eventually periodic.
\MS

\NI {\bf Weight decomposition:}  The {\bf integral weight decomposition} $W(p/q)$ of a rational number
$p/q= [s_0;\dots,s_n]$ is an array of numbers that are recursively defined\footnote{
Note that this recursion can be read either as defining $q_{\al+1}$ in terms of $q_\al, q_{\al-1}$, or as defining  
 $q_{\al-1}$ in terms of $q_\al, q_{\al+1}$ where the last nonzero entry is $1$.}
 as follows:
\begin{align}\label{eq:Wpq}
W(p/q) &= \bigl(q_0^{\times s_0}, q_1^{\times s_1},\dots, q_n^{\times s_n}\bigr), \quad\mbox{ where }\\ \notag
& q_0 =q, \ q_1= p -s_0 q_0,\dots, q_\al = q_{\al-2} - s_{\al-1}  q_{\al-1}, \al\ge 2,\;\; q_n = 1.
\end{align}
For example, $5/3=[1;1,2]$ and $W(5/3) = (3;2,1,1)=:(3; 2, 1^{\times 2})$. If $W(p/q)$ is $(w_1,\dots,w_N)$ then
\begin{align}\label{eq:Wpq1}
\sum_{i=1}^N w_i = p+q-1, \quad \sum_{i=1}^N w_i^2 = pq,\quad w_1 = q, \;\; w_N = 1.
\end{align}
Using the fact that the first entry of $W(p/q)$ is $q$, we can interpret the other entries as the denominators $q_\al$ of the $\al$th \lq\lq tail\rq\rq\, $p_\al/q_\al$ of  $p/q = [s_0;\dots,s_n]$.  
Namely, if we define $p_\al/q_\al: = [s_\al;\dots,s_n]$, then we have
\begin{align}\label{eq:Wpi}
W(p_\al/q_\al): =  W([s_\al;\dots,s_n])
 = \bigl(q_\al^{\times s_\al} , \dots, q_n^{\times s_n}\bigr), \quad 0\le i\le n.
\end{align}
Moreover, $q_n = 1, q_{n-1} = s_n,$ and the other terms $q_{n-1}, q_{n-2},\dots$ can be calculated from the recursion
$q_{n-\al-1} = s_{n-\al}q_{n-\al} + q_{n-\al+1}$.  In particular, 
$$
q=q_0,\quad  p= p_0 = s_0q_0 + q_1.
$$
Each group of terms $q_\al^{\times s_\al}$ is called a block. 

Note that the entries of $W(p/q)$ are integers. We also sometimes use the weight decomposition of $p/q$ defined as follows: 
\begin{align}\label{eq:wpq}
w(p/q): = \frac 1q W(p/q) = (1,\dots, 1/q)
\end{align}
More generally, the weight expansion of $z\approx p/q$ is defined by the same recursion that defines $w(p/q)$.  For example if $6 <z<7$, $w(z) = (1^{\times 6}, z-6,\dots)$.
For more information, see \cite[Lem.2.2.1]{ball}.

\section{Computations on triples}
\setcounter{thm}{0}
\numberwithin{thm}{section}

We now gather together various computations that are needed for the main argument.
Note that for any index $\bullet = \la, \mu, \rho$  we define $r_\bullet: = p_\bullet + q_\bullet$.

\begin{lemma}\label{lem:prelim0}  The following identities and inequalities hold in any triple $\Tt$.
\begin{itemlist}\item [{\rm(i)}] $t_\rho r_\mu - t_\mu r_\rho = 8p_\la$.
\item [{\rm(ii)}]    $t_\rho d_\mu - t_\mu d_\rho = 3p_\la$,\  $t_\la d_\mu - t_\mu d_\la = 3q_\rho$
\item [{\rm(iii)}]  $m_\rho d_\mu -  m_\mu d_\rho = \eps p_\la$,  $m_\mu d_\la - m_\la d_\mu = \eps q_\rho$
 \item [{\rm(iv)}] $t_\la r_\mu - t_\mu r_\la = 8q_\rho$.
 \item[{\rm (v)}]  $t_\la t_\rho < 2t_\mu$; in particular  $3\le t_\la, t_\rho< t_\mu$.
 \end{itemlist}
 \end{lemma}
\begin{proof}  The formulas (i) and (iv) are proved in  \cite[Lem.4.6]{M1} by direct computation using the adjacency and $t$-compatibility conditions.\footnote
{
One could also check them for the basic triples $\Tt^n_*,n\ge 0,$ 
and show that they are preserved by  symmetries and mutations.}
Formulas (ii) and (iii) can be deduced from (i) since $8d_\bullet = 3r_\bullet + \eps t_\bullet$ while  $8m_\bullet = r_\bullet + 3\eps t_\bullet$ by \eqref{eq:pqt}.  

Using the formulas in \eqref{eq:baU}, it is straightforward to check
that  (v)  holds for the basic triples in $\Cc\Ss^U$.  Moreover (v) continues to holds under mutation.
 Indeed, if  (v) holds for $\Tt$ then it  holds for  $x\Tt$,  provided that
 $
  ( t_\la t_\mu - t_\rho)^2 + t_\mu^2< t_\la( t_\la t_\mu - t_\rho) t_\mu$.  But this
 simplifies to $t_\mu^2 < t_\rho(t_\la t_\mu - t_\rho)$, which holds by assumption.  
A similar argument works for the mutation $y\Tt$.   Therefore (v) holds for all
 triples in $\Cc\Ss^U$.
It follows that  (v) holds for all triples  since the variable $t_\bullet$ is preserved by the symmetries $S,R$ and the given expression is symmetric in $\la,\rho$ (which are interchanged by $R$.)
\end{proof}

\begin{cor}\label{cor:prelim0}  Let  $\Tt$ be any triple and denote $
 r_\bullet: = p_\bullet + q_\bullet.$
Then 
 \begin{equation}\label{eq:rt}
 \frac{r_\mu}{t_\mu} >  \max\bigl(\frac{r_\la}{t_\la}, \frac{r_\rho}{t_\rho}\bigr),
 \end{equation}
 except in the case
 of $y^kR^{\sharp}(\Tt^0_*)$, when we have 
 \begin{equation}\label{eq:rtR0}
 \frac{r_\la}{t_\la}=\frac{r_\mu}{t_\mu} =  3> \frac{r_\rho}{t_\rho}=1/3.
\end{equation}
\end{cor}
\begin{proof}  The inequality \eqref{eq:rt} follows from Lemma~\ref{lem:prelim0}~(i) and (iv) whenever $p_\la, q_\rho > 0$.    This holds in all cases except for the triples $y^kR^{\sharp}(\Tt^0_*)$ since they have $q_\rho=0$.  But they satisfy \eqref{eq:rtR0}; indeed, this holds for  $R^{\sharp}(\Tt^0_*)$ by the formulas in \eqref{eq:baL},
and the fact that  $q_\rho=0$ in $R^{\sharp}(\Tt^0_*)$ implies that the ratio $r/t$ is constant for all terms in the ascending stairs in this triple.
\end{proof}

We next turn to the details need to complete the proof of Lemma~\ref{lem:slopeest3}. To this end, we must
verify  the inequality \eqref{eq:MDTineq1} 
$$
\frac{t_{k+1}^2 - 8}{t_{k+1}r_{k+1}}  > \frac{t_{k}^2 - 8}{t_k r_k},
$$
where $k$ labels the staircase step, in the situations stated below.

\begin{lemma}\label{lem:basecases}\begin{itemlist}\item[{\rm(i)}]
The inequality \eqref{eq:MDTineq1} holds for the first two steps in the descending pre-staircases 
associated to $yx^i\Tt^0_*$ for all $i\ge 0$
\item[{\rm(ii)}]
The inequality \eqref{eq:MDTineq1}  holds for the second and third steps in the descending pre-staircases 
associated to $y^iR^\sharp(\Tt^0_*)$, $i\geq0$.
\end{itemlist}
\end{lemma}
\begin{proof}
To prove (i),  first note that it holds for $i=0$ since the first two steps in $\Ss^{\Tt_y}_u$ are 
$\bB^U_1, \bE_{[7;3,6]}$ with  $(r,t) = (9,5), (158, 62)$.   In general,  the descending staircase $\Ss^U_0$ has steps  $\bE_{x^i[7;4]}, i\ge 0$, and we denote their $(r,t)$ coefficients as $(r_i,t_i)$. They satisfy a recursion with parameter $3$ and
seeds $(9,5),  (r_0,t_0) = (33,13)$ so that $$
 (r_1,t_1) = (90,34),\quad  (r_2,t_2) = (237,89).
 $$
By Remark~\ref{rmk:misc}~(ii) the ratios $r_i/t_i$ increase, and using \eqref{eq:XX'} one finds that
for $i\ge 2$, 
\begin{align}\label{eq:TR}
\frac 52 < \frac {237}{89} < \frac{r_i}{t_i} <\lim \frac{r_i}{t_i}=: \frac RT =  \frac{45 + 39 \sqrt 5}{25+11\sqrt 5} < \frac 83.
\end{align}
When $i\ge 1$, the descending pre-staircase associated to 
$yx^i\Tt^0_*$ has first two steps $\bE_{x^{i-1}[7;4]},$ $ \bE_{yx^{i}[7;4]}$ where the $(r,t)$ components of 
$\bE_{yx^{i}[7;4]}$ are $\bigl(t_{i-1}r_i - 7, t_{i-1}t_i - 3\bigr)$.  Thus it suffices to show that
$$
\frac{(t_{i-1}t_i - 3)^2 - 8}{(t_{i-1}t_i - 3)(t_{i-1}r_i - 7)} > \frac{t_{i+1}}{r_{i+1}} > \frac{t_{i-1}^2-8}{t_{i-1}r_{i-1}}.
$$
We now prove  by induction on $i\ge 1$ that each of these inequalities holds.  

The base case $i=1$ is readily checked. With $i\ge 2$, we assume these inequalities known for $i-1$. 
To simplify notation let us write  $$
t'': = t_i, \ t': = t_{i-1},\  t: = t_{i-2}
\mbox{ so that } \ \ t_{i+1} = 3t'-t =  8t' - 3t,
$$
and similarly for $r$.
Then the second inequality is equivalent to 
$$
(8t'-3t) tr > (8r' - 3r)(t^2 - 8), \quad i.e. \quad 64r' - 24r > 3t'(tr'-rt') = 144 t',
$$
where we have used the fact that $tr'-rt' = t_{i-2}r_{i-1} - r_{i-2}t_{i-1}$ is constant and equal to $48$.  This will hold if $8r' - 3r > 18t'$, i.e. if $8r_{i-1} - 3r_{i-2}> 18t_{i-1}$.  
It is easy to check that  $r_{i-2} < \frac 25 r_{i-1}$ for all $i$, so that $8r_{i-1} - 3r_{i-2}> \frac{34}5 r_{i-1}$.  Thus it suffices to verify that $ \frac{34}5 r_{i-1} \ge 18 t_{i-1}$ for $i\ge 2$.  But this holds when $i=2$, and holds for $i>2$ because $r_i/t_i$ increases.

Now consider the first inequality, which in the current notation is equivalent to
$$
(3r'' - r')\bigl( (t't'' - 3)^2 - 8\bigr) >  (3t''-t') (t' t'' - 3)(t' r'' - 7).
$$
Putting the terms of highest order on the left and simplifying the LHS as above, we obtain
\begin{align*}
&(t')^2 t''\bigl( (3r''-r') t'' - (3t''-t') r''\bigr)  = 48(t')^2 t'' \\
& \qquad \qquad >  (3r''-r')(6t' t'' - 1) -(3t''-t')\bigr(7t' t''+3t' r'' - 21\bigl).
\end{align*}
By \eqref{eq:TR},  we have $t_i> \frac 38 r_i > \frac {15}{16} t_i$ for all $i$.   Therefore  
$3r''-r' < \frac 83 (3t''-t')$ (since $x_{i+1} = 3x''-x' $) and
$7t'' + 3r'' >  \frac{29}2 t''$.  Hence
\begin{align*}
& (3r''-r')(6t' t'' - 1) -(3t''-t')\bigr(7t' t''+3t' r'' - 21\bigl) \\
& 
\qquad < (3t''-t')\bigl( 16  t't'' -   \tfrac{29}2 t't''   + 21 \bigr)
 <  2(3t''-t')t't'' < 48 (t')^2 t''.
\end{align*}
where the first inequality forgets the $-1$, the second holds because $t',t''\ge 5$ and the last holds because $t'' < 3t'$.

\MS

To prove (ii), note first that the descending stairs in $y^iR^\sharp(\Tt^0_*)$ has first two steps
with $(p,q,t)$ coordinates $(2,0,3)$ and $(p_i,q_i,t_i)$ and recursion parameter $t_{i-1}$, where $(p_i,q_i,t_i)$ 
is defined recursively with recursion parameter $3$ and first two terms $
(p_{-1},q_{-1},t_{-1})=(13,2,5), (p_0,q_0,t_0)=(34,5,13)$.
Therefore, we
 need to show that
\[
\frac{(t_{i-1}t_i-3)^2-8}{(t_{i-1}t_i-3)(t_{i-1}r_{i-1})}>\frac{t_i^2-8}{t_ir_i}.
\]
It suffices to show that the left hand side is always greater than $1/3$ and the right hand side is always less than $1/3$. The latter is immediate from the fact that $t_i=3r_i$ by \eqref{eq:rtR0}. For the former, again using the fact that $t_{i-1}=3r_{i-1}$, we need to show
\begin{align*}
(t_{i-1}t_i-3)^2-8&>(t_{i-1}t_i-3)(t_{i-1}^2),
\end{align*}
or, equivalently $t_{i-1}^2t_i^2-6t_{i-1}t_i+1>t_{i-1}^3t_i-3t_{i-1}^2$.
It suffices to show the much weaker inequality $t_{i-1}^2t_i^2-6t_{i-1}t_i>t_{i-1}^3t_i$, 
or, equivalently, $ t_{i-1}(t_i-t_{i-1})>6$. 
But this holds because $t_k\ge 5$ and $t_k-t_{k-1}>1$. This completes the proof.
\end{proof}

The next inequalities are needed in the proof of Lemma~\ref{lem:2T}.

\begin{lemma}\label{lem:prelim2}  The following inequality holds for all triples.
\begin{align}\label{eq:dtineq}
\frac {d_\mu}{d_\rho} - \frac {d_\rho}{d_\mu} > t_\la - 1. 
\end{align}  
\end{lemma}
\begin{proof}
We first check that this holds for $\Tt^*_n, R^{\sharp}(\Tt^n), n\ge 0$, where in the first case we may use the formulas in \eqref{eq:SsUparam}, and in the second we use that, by \eqref{eq:pqt} and \eqref{eq:baL}, the $(d,t)$ coordinates of
$R^{\sharp}(\Tt^n)$ are 
$$
(d_\la,t_\la) = (5(n+1) , 2n+ 5),  (d_\rho, t_\rho) = (5n  , 2n+3), \ 
(d_\mu, t_\mu) = (10n^2+ 25n+13 , 4n^2 + 16n + 13).
$$

To see that they hold for the images of the symmetries, first consider $(S^i)^\sharp(\Tt^*_n)$ so $t_\la=2n+3$. Letting $x:=d_\mu$ and $y:=d_\rho$ of $(S^i)^\sharp(\Tt^*_n)$, \eqref{eq:dtineq} is equivalent to
\[x^2-y^2-xy(2n+2) > 0.\] Considering this as a quadratic in $x$ for fixed $y,$ this holds if
\begin{equation} \label{eq:dtineqapp}
     x > (n+1+\sqrt{n^2+2n+2})y.
     \end{equation}
Further, by Remark~\ref{rmk:quasitrip}, we have that $x=(2n+3)y-d_{seed}$ where $d_{seed}$ is the degree coordinate of $(S^i)^\sharp(\bE^U_{seed,u})=(S^i)^\sharp(-2,0,-5,-1)$ and hence $d_{seed}<0.$ Therefore, \eqref{eq:dtineqapp} is equivalent to
\begin{align*} 
 (2n+3)y-d_{seed} &> (n+1+\sqrt{n^2+2n+2})y \\ (2+n-\sqrt{n^2+2n+2})y &> d_{seed},
\end{align*} 
which holds as $d_{seed}$ is negative and the left hand size is positive. A similar argument holds for $(S^iR)^\sharp(\Tt^*_n).$

Finally we check that they remain true under mutation.  
If \eqref{eq:dtineq} holds for $\Tt$ then for $x\Tt$ we have
$$
\frac {t_\la d_\mu - d_\rho}{d_\mu} - \frac {d_\mu}{t_\la d_\mu - d_\rho} = t_\la - \frac{d_\rho}{d_\mu}-\frac {d_\mu}{t_\la d_\mu - d_\rho}> t_\la -1,
$$
where the inequality holds because each of the last two terms is $<1/2$.  To show this, notice that any two degree values of sequential classes $\bE, \bE', \bE''$ in a staircase satisfy $d''=td'-d$ for $t\geq3$ (see Lemma~\ref{lem:prelim0}~(v)) and $d<d'$. This is even true if $\bE'$ is the first geometric step in the staircase, thinking of $(S^iR^\delta)^\sharp(\bE^U_{\ell/u,seed})$ as the algebraic first step $\bE$ if necessary (see Remark~\ref{rmk:quasitrip}).

For $y\Tt$ we need 
$
\frac {t_\rho d_\mu - d_\la}{d_\rho} - \frac {d_\rho}{t_\rho d_\mu - d_\la}
 > t_\mu - 1.
 $
But $t_\rho d_\mu  - t_\mu d_\rho  = 3p_\la$ by Lemma~\ref{lem:prelim0}~(ii).   So by similar reasoning we have 
$$
\frac {t_\rho d_\mu - d_\la}{d_\rho} - \frac {d_\rho}{t_\rho d_\mu - d_\la}  = \frac {3p_\la -d_\la}{d_\rho} + t_\mu - \frac {d_\rho}{t_\rho d_\mu - d_\la} > t_\mu - 1,
$$
because for any triple $3p-d>0$ by combining \eqref{eq:pqt}, the fact that $p/q>5$, and induction using Corollary~\ref{cor:prelim0}, where we prove the base case $r/t>1/3$ using the formulas for $(S^iR^\de)(2n+6,1)$ in terms of the $y_i$ in the proof of Lemma~\ref{lem:rsforS}.
\end{proof} 

We next find suitable constants $x_0,x_1$ to use in  Lemma~\ref{lem:2T}~(iii).

\begin{lemma}\label{lem:qtrip2} \begin{itemlist}\item[{\rm (i)}] For each $n\geq0$ and $i\geq0$, the limiting $b_\infty$-values of the descending staircases $\Ss$ in the family $(S^i)^\sharp(\Cc\Ss^U\cap[2n+6,2n+8])$ satisfy
\[
\acc(m_{i,n-1}/d_{i,n-1})<\acc(b_\infty)<\acc(m_{i,n}/d_{i,n}),
\]
where $d_{i,n}, m_{i,n}$ are the degree components of $\bE_{i,n}:=(S^i)^\sharp(\bB^U_n)$. Here $\bE_{0,-1}=\bB^U_{-1}=(2,1,4,1,1,1)$.

\item[{\rm (ii)}] Similarly, the limiting $b_\infty$-values of the descending staircases in the family $(S^iR)^\sharp(\Cc\Ss^U\cap[2n+6,2n+8])$ satisfy
\[
\acc(m_{i,n+2}'/d_{i,n+2}')<\acc(b_\infty)<\acc(m_{i,n+1}'/d_{i,n+1}'),
\]
where $d_{i,n}', m_{i,n}'$ are the degree components of $\bE_{i,n}':=(S^iR)^\sharp(\bB^U_n)$.
\end{itemlist}
\end{lemma}

\begin{proof} Notice that we have
\begin{equation}\label{eq:pqEb}
\frac{p_{i,n}}{q_{i,n}}<\partial^+(I_{\bE_{i,n}})\leq\acc(b_\infty)<\partial^-(I_{\bE_{i,n+1}})
\end{equation}
because $\acc(b_\infty)$ is the limit of the step centers of a descending staircase. To obtain the upper bound in case (i),  notice that the class $\bE_{i,n}$ is a step in a staircase accumulating to $\partial^-(I_{\bE_{i,n+1}})>\acc(b_\infty)$. By Corollary~\ref{cor:symm}, for every staircase (whether $b>1/3$ or $b<1/3$) the values $\acc(m_k/d_k)$ always descend to the $z$-coordinate $\acc(b)$ of the accumulation point, thus
\[
\acc(b_\infty)<\partial^-(I_{\bE_{i,n+1}})<\acc(m_{i,n}/d_{i,n}).
\]

 To obtain the lower bound in (i), we would like to apply Lemma~\ref{lem:accmd} to $\bE'$ the middle and $\bE$ the right entry of
 \[
 (S^i)^\sharp(\bE^U_{\ell,seed}, \bB^U_{n-1},\bB^U_n)=((S^i)^\sharp(\bE^U_{\ell,seed}),\bE_{i,n-1},\bE_{i,n}).
 \]
 Thus we would obtain $\acc(m_{i,n-1}/d_{i,n-1}) < p_{i,n}/q_{i,n}<\acc(b_\infty)$, where for the second inequality we use \eqref{eq:pqEb}. The only hypothesis of Lemma~\ref{lem:accmd} that does not automatically hold for $\bE, \bE'$ (because they are the middle and right entries of a quasi-triple, the other hypotheses hold) is $\sqrt{2}p'<p$. Because $S^i(2n+6,1)=(2ny_{i+1}+y_{i+2},2ny_i+y_{i+1})$, this holds if
 \begin{align*}
\sqrt{2}(2(n-1)y_{i+1}+y_{i+2})&>2ny_{i+1}+y_{i+2} \\
& \qquad   \;\;\Longleftrightarrow\;\;
ny_{i+1}>(\sqrt{2}-1)y_{i+1}+\tfrac{1}{2}y_i.
\end{align*}
Therefore if $n\geq1$, we obtain the lower bound $\acc(m_{i,n-1}/d_{i,n-1}) < p_{i,n}/q_{i,n}$.
 
 When $n=0$, it is not true that  $\acc(m_{0,-1}/d_{0,-1})<6$, where $6=p_{0,0}/q_{0,0}$, so we instead show
 \begin{equation}\label{eq:7b}
 \acc(m_{i,-1}/d_{i,-1})<S^i(7,1)<\acc(b_\infty).
 \end{equation}
 For the first inequality in \eqref{eq:7b}, the analogue of \eqref{eq:p'q'<pq} with $(p_i,q_i)=S^i(7,1)=(y_{i+2}+y_{i+1},y_{i+1}+y_i)$ is
\[
\frac{p_{i,-1}^2+q_{i,-1}^2+2}{p_{i,-1}q_{i,-1}-1}<\frac{(y_{i+1}+y_{i+2})^2+(y_i+y_{i+1})^2}{(y_i+y_{i+1})(y_{i+1}+y_{i+2})}.
\]
We may use the fact that $t^2=p^2-6pq+q^2+8$ and $t_{i,-1}=1$ to simplify the LHS to $6+1/(p_{i,-1}q_{i,-1}-1)$. Now because $p^2-6pq+q^2$ is invariant under $S$ and $(y_1+y_2)^2-6(y_1+y_2)(y_0+y_1)+(y_0+y_1)^2=8$, subtracting 6 from both sides gives us
\begin{align*}
\frac{1}{(-2y_{i+1}+y_{i+2})(-2y_i+y_{i+1})-1}&<\frac{8}{(y_i+y_{i+1})(y_{i+1}+y_{i+2})},
\end{align*}
which simplifies to
\begin{align*}
y_{i+1}y_i+2y_{i+1}^2-1+y_{i+2}y_{i+1}&<32y_{i+1}y_i-32y_{i+1}^2+8y_{i+2}y_{i+1}+8,
\end{align*}
where in the second line we have used $y_{i+1}^2=y_{i+2}y_i+1$, which is easy to prove by induction. It is enough to show a stronger inequality with the constants removed, thus we can divide by $y_{i+1}$ to obtain
\[
y_i+2y_{i+1}+y_{i+2}<32y_i-32y_{i+1}+8y_{i+2}, \text{ i.e. }34y_{i+1}<7y_{i+2}+31y_i,
\]
which follows immediately from the definition of the $y_i$ sequence.

For the second inequality, in \eqref{eq:7b}, it suffices to show that $(S^i)^\sharp(\bB^U_0)=\bE_{i,0}$ blocks $S^i(7,1)$. This is shown in the second conclusion of Corollary~\ref{cor:main1iipf}; see also \cite[Cor.4.1.5]{MM}. 
In case (ii), notice that
\begin{equation}\label{eq:p'q'Eb}
\frac{p_{i,n+1}'}{q_{i,n+1}'}<\partial^+(I_{\bE_{i,n+1}'})\leq\acc(b_\infty)<\partial^-(I_{\bE_{i,n}'}).
\end{equation}
We obtain $\acc(b_\infty)<\acc(m_{i,n+1}'/d_{i,n+1}')$ as in (i): notice that $\bE_{i,n+1}'$ is a step in a staircase accumulating to $\partial^-(I_{\bE_{i,n}'})>\acc(b_\infty)$; finish as in the proof of the analogous inequality in (i).

To obtain $\acc(b_\infty)>\acc(m_{i,n+2}'/d_{i,n+2}')$, we can apply Lemma~\ref{lem:accmd} with $\bE'$ the left and $\bE$ the middle entry of
\[
(S^iR)^\sharp(\bE^U_{\ell,seed}, \bB^U_{n+1} ,\bB^U_{n+2})=(\bE_{i,n+2}',\bE_{i,n+1}',(S^iR)^\sharp(\bE^U_{\ell,seed})).
\]
Because $S^iR(2n+6,1)=(2ny_{i+2}+y_{i+1},2ny_{i+1}+y_i)$, the $p$ coordinates satisfy $p'>p$, which is stronger than the hypothesis $\sqrt{2}p'>p$ of Lemma~\ref{lem:accmd}. Because $p_{i,n+1}'/q_{i,n+1}'<\partial^+(I_{\bE_{i,n+1}'})\leq\acc(b_\infty)$ by \eqref{eq:p'q'Eb}, we obtain the desired inequality.
\end{proof}

\begin{cor} \label{cor:qtrip2} Let $d_{i,n}$ and $m_{i,n}$ be the degree coordinates of $(S^i)^\sharp(\bB^U_n)$, let $d_{i,n}'$ and $m_{i,n}'$ be the degree coordinates of $(S^iR)^\sharp(\bB^U_n)$, and let $b_\infty$ be as in Lemma~\ref{lem:qtrip2}. 

\begin{itemlist}\item[{\rm (i)}] If $i$ is even then $m_{i,n-1}/d_{i,n-1}<b_\infty<m_{i,n}/d_{i,n}$ and $m_{i,n+1}'/d_{i,n+1}'<b_\infty<m_{i,n+2}'/d_{i,n+2}'$.

\item[{\rm (ii)}]  If $i$ is odd then $m_{i,n}/d_{i,n}<b_\infty<m_{i,n-1}/d_{i,n-1}$ and $m_{i,n+2}'/d_{i,n+2}'<b_\infty<m_{i,n+1}'/d_{i,n+1}'$.
\end{itemlist}
\end{cor}
\begin{proof} The conclusions follow immediately from Lemma~\ref{lem:qtrip2} and the fact that $\acc$ is orientation preserving/reversing according to whether $b>1/3$ or $<1/3$, or equivalently, according to whether $i+\delta$ is even/odd.
\end{proof}

Our final results complete
the proof of Lemma~\ref{lem:2T}~(iv).

\begin{lemma} \label{lem:xmod}
 Given two classes $\bE=(d,m,p,q,t,\eps)$ and $\bE'=(d',m',p',q',t',\eps)$, let $x=m/d$ and $x'=m'/d'$.
\begin{itemlist}\item[{\rm (i)}] If $(\bE_\la,\bE',\bE)$ is a quasi-triple, then 
$ \frac{1-xx'}{x'-x}=\frac{ \eps p'q}{p_\la}$;  while
\item[{\rm (ii)}] if
$(\bE,\bE',\bE_\rho)$ is a quasi-triple, then 
$ \frac{1-xx'}{x'-x}=-\frac{\eps pq'}{q_\rho}.$
\end{itemlist}
\end{lemma}
\begin{proof} We have
\[ \frac{1-(m/d)(m'/d')}{m'/d'-m/d}=\frac{dd'-mm'}{m'd-md'}.\]
If $(\bE_\la,\bE',\bE)$ is a quasi-triple, then, we have \eqref{eq:adjac1} and Lemma~\ref{lem:prelim0}~(iii) give $dd'-mm'=p'q$ and $m'd-md'=\eps p_\la.$ 
Similarly if $(\bE,\bE',\bE_\rho)$ is a quasi-triple, we have $dd'-mm'=pq'$ and, by Lemma~\ref{lem:prelim0}~(iii),  $m'd-md'=-\eps q_\rho.$
\end{proof}

\begin{lemma} \label{lem:rsforS} \begin{itemlist}\item[{\rm(i)}] For fixed $i>0$ and $n\ge0$,
let $$
(S^i)^\sharp\bB^U_n:=
(d_1,m_1,p_1,q_1), \;\; (S^i)^\sharp\bB^U_{n-1}:=(d_0,m_0,p_0,q_0),\quad
x_i=\frac{m_i}{d_i}, \;\; i=0,1.
$$
Assume $(p_\la,q_\la,t_\la) \in (S^i)^\sharp(C\Ss^U_{n})$ is such that $(p_\la,q_\la) \neq S^i(2n+6,1)$. Then 
\begin{align}\label{eq:x0x1} \left|\frac{1-x_0x_{1}}{x_{1}-x_{0}} \right|>\frac{p_\la}{t_\la-1}.
\end{align} 
\item[{\rm(ii)}] For fixed $i,n$, let \[(S^iR)^\sharp(2n+8,1):=
(d_0,m_0,p_0,q_0), \quad (S^iR)^\sharp(2n+10,1):=(d_1,m_1,p_1,q_1),\] 
and set $x_i=\frac{m_i}{d_i}, i=0,1$. 
Then, the inequality \eqref{eq:x0x1} also holds for any $(p_\la,q_\la,t_\la) \in (S^iR)^\sharp(C\Ss^U_n)$.
\end{itemlist}
\end{lemma}
\begin{proof}
By Remark~\ref{rmk:quasitrip}, we have $\bigl((S^i)^\sharp(1,1),(S^i)^\sharp(2n+6,1),(S^i)^\sharp(2n+8,1)\bigr)$ is a quasi-triple. 
Applying Lemma~\ref{lem:xmod}, we have \[ \left|\frac{1-x_0x_{1}}{x_{1}-x_{0}}\right|=\frac{ p_{0} q_{1}}{y_{i+1}-y_i}\] since  $(S^i)(1,1)=(y_{i+1}-y_i,y_i-y_{i-2}).$ It remains to show that
\[ \frac{p_{0}q_{1}}{y_{i+1}-y_i}>\frac{p_\la}{t_\la-1}.\]
By assumption, we have
$S^i\left({2n+7}/{1}\right) <p_\la/q_\la < S^i\left({2n+8}/{1}\right).$
Further,  $S^i(2n+7,1)= S^i(w_1) = (2ny_{i+1}+y_{i+2}+y_{i+1},2ny_{i}+y_{i+1}+y_i):=(p',q')$ by \eqref{eq:symmvi}. Let $r=p'/q'.$ So $r<p_\la/q_\la.$
We have 
\begin{equation} \label{eq:tlapla}
 t_\la^2/p_\la^2=1+q_\la^2/p_\la^2-6q_\la/p_\la+8>1-6/r+1/r^2=\frac{r^2-6r+1}{r^2}
 \end{equation}
 as $1-6/r+1/r^2$ is an increasing function for $r>a_{min}$. 
 A proof by induction verifies that
 \[ r^2-6r+1=\frac{4n^2+16n+8}{(q')^2},\] so
 \[ t_\la^2/p_\la^2>\frac{4n^2+16n+8}{(p')^2}.\]
 Therefore,
 \[ p_\la/(t_\la-1)<\frac{p'}{\sqrt{4n^2+16n+8}}\frac{2n+3}{2n+2}.\]
 It remains to show that
 \[ \frac{p_{0} q_{1}}{y_{i+1}-y_i}>\frac{p'}{\sqrt{4n^2+16n+8}}\frac{2n+3}{2n+2}.\]
 Substituting in \[p_0=2(n-1)y_{i+1}+y_{i+2}, q_1=2ny_{i}+y_{i+1}, \quad  p'=2ny_{i+1}+y_{i+2}+y_{i+1}\] and simplifying by taking $n=0$ to the terms on the right hand side that decrease in $n$, this is equivalent to
 \begin{align*} 
 (2(n-1)y_{i+1}+y_{i+2})(2ny_{i}+y_{i+1})&>\frac{3}{4\sqrt{2}}(2ny_{i+1}+y_{i+2}+y_{i+1})(y_{i+1}-y_i).
 \end{align*} 
 This holds as for $n \geq 0,$ we have $2ny_{i}+y_{i+1}>y_{i+1}-y_i$ and
 \[ 2(n-1)y_{i+1}+y_{i+2}>\frac{3}{4\sqrt{2}}(2ny_{i+1}+y_{i+2}+y_{i+1}).\]
 This proves (i).
\MS

Towards (ii), notice that  $\bigl((S^iR)^\sharp(2n+8,1),(S^iR)^\sharp(2n+6,1),(S^iR)^\sharp(1,1,1,1)\bigr)$ is also a quasi-triple, so that  \[ \left|\frac{1-x_{1}x_{0}}{x_{1}-x_{0}}\right|=\frac{p_{1}q_{0}}{y_{i+2}-y_i}.\]
Further, we have
$S^iR\left({2n+8}/{1}\right)<p_\la/q_\la<S^iR\left({2n+6}/{1}\right)$ where 
 $(S^iR)(2n+8,1)=(2(n+1)y_{i+2}+y_{i+1},2(n+1)y_{i+1}+y_{i})$.
The rest of the argument is very similar to (i) and is again left to the reader.
\end{proof}

\end{document}